\documentclass[11pt]{amsart}

\usepackage{epigamath}

\usepackage[english]{babel}

\numberwithin{equation}{section}

\usepackage{enumitem}
\usepackage[all]{xy}
\usepackage{tikz-cd}
\usepackage{tikz}
\usetikzlibrary{%
  matrix,%
  calc,%
  arrows%
}

\newtheorem{thm}{Theorem}[section]
\newtheorem{prop}[thm]{Proposition}
\newtheorem{lemme}[thm]{Lemma}
\newtheorem{cor}[thm]{Corollary}
\newtheorem{conj}[thm]{Conjecture}

\theoremstyle{definition}
\newtheorem{defi}[thm]{Definition}
\newtheorem{nota}[thm]{Notation}

\theoremstyle{remark}

\newtheorem{rmk}[thm]{Remark}
\newtheorem{ex}[thm]{Example}
\newtheorem{war}[thm]{Warning}

\DeclareMathOperator{\Z}{\mathbb{Z}}
\DeclareMathOperator{\R}{\mathbb{R}}
\DeclareMathOperator{\N}{\mathbb{N}}

\DeclareMathOperator{\codim}{Codim}
\DeclareMathOperator{\Hom}{Hom}

\DeclareMathOperator{\rk}{rk}

\DeclareMathOperator{\Ext}{Ext}
\DeclareMathOperator{\Image}{Im}
\DeclareMathOperator{\discr}{discr}
\DeclareMathOperator{\Fix}{Fix}
\DeclareMathOperator{\Supp}{Supp}

\DeclareMathOperator{\Vect}{Vect}
\DeclareMathOperator{\tors}{tors}
\DeclareMathOperator{\id}{id}

\DeclareMathOperator{\Sym}{Sym}
\DeclareMathOperator{\Ker}{Ker}

\DeclareMathOperator{\diag}{diag}
\DeclareMathOperator{\Sing}{Sing}
\DeclareMathOperator{\Spec}{Spec}
\DeclareMathOperator{\GL}{GL}
\DeclareMathOperator{\Q}{\mathbb{Q}}

\DeclareMathOperator{\C}{\mathbb{C}}
\DeclareMathOperator{\Pj}{\mathbb{P}}

\DeclareMathOperator{\tr}{tr}

\newcommand{\RR}{{\bf R}}
\newcommand{\F}{\mathbb{F}_{p}}

\newcommand{\eq}[1][r]
{\ar@<-3pt>@{-}[#1]
\ar@<-1pt>@{}[#1]|<{}="gauche"
\ar@<+0pt>@{}[#1]|-{}="milieu"
\ar@<+1pt>@{}[#1]|>{}="droite"
\ar@/^2pt/@{-}"gauche";"milieu"
\ar@/_2pt/@{-}"milieu";"droite"}

\EpigaVolumeYear{6}{2022} \EpigaArticleNr{3} \ReceivedOn{September 15, 2019}
\InFinalFormOn{October 18, 2021}
\AcceptedOn{November 15, 2021}

\title{Integral cohomology of quotients via toric geometry}
\titlemark{Integral cohomology of quotients via toric geometry}

\author{Gr\'egoire Menet}
\address{Laboratoire Paul Painlev\'e, B\^atiment M2, Cit\'e Scientifique, 59655 Villeneuve-d'Ascq}
\email{gregoire.menet@univ-lille.fr}

\authormark{G.~Menet}

\AbstractInEnglish{We describe the integral cohomology of $X/G$ where $X$ is a compact complex manifold and $G$ a cyclic group of prime order with only isolated fixed points. As a preliminary step, we investigate the integral cohomology of toric blow-ups of quotients of $\C^n$. We also provide necessary and sufficient conditions for the spectral sequence of equivariant cohomology of $(X,G)$ to degenerate at the second page. As an application, we compute the Beauville--Bogomolov form of $X/G$ when $X$ is a Hilbert scheme of points on a K3 surface and $G$ a symplectic automorphism group of orders 5 or 7.}



\acknowledgement{This work was supported by Fapesp grant 2014/05733-9 and the ERC-ALKAGE, grant No. 670846.}

\begin{document}



\maketitle

\begin{prelims}

\DisplayAbstractInEnglish

\bigskip

\DisplayKeyWords

\medskip

\DisplayMSCclass

\end{prelims}


\newpage

\setcounter{tocdepth}{1}

\tableofcontents


\section{Introduction}
\subsection{Background and motivation}
Let $X$ be a topological space
endowed with the action of a finite automorphism group $G$. 
We consider $\pi:X\rightarrow X/G$ the quotient map.
It is quite easy to compute the cohomology of $X/G$ with rational coefficients since it is isomorphic to the invariant cohomology $H^*(X/G,\Q)\simeq H^*(X,\Q)^G$. However, switching to integral coefficients, several complications appear. Let $Y$ be a topological space. We denote by $H^*_f(Y,\Z)$ the torsion-free part of the cohomology (see Section \ref{nota1} (vi) for the notation).
\begin{itemize}
\item[(i)]
The first problem is to determine the \emph{torsion} of $H^*(X/G,\Z)$.
\item[(ii)]
The second problem is to find a \emph{basis} of $H^*_f(X/G,\Z)$. We are also interested in the \emph{ring structure} of $H^*_f(X/G,\Z)$
which differs from the one of $H^*_f(X,\Z)^G$.
\end{itemize}

No existing theory solves these problems; nevertheless, we mention a result of Smith which is an important tool for our purpose. 
Smith in \cite{Smith} (see also \cite[Section III.7]{Bredon} 
and \cite{Transfers}) has constructed a push-forward map $\pi_*:H^{*}(X,\Z)\rightarrow H^{*}(X/G,\Z)$ with the following properties:
\begin{equation}
\pi_{*}\circ\pi^{*}=d\id_{H^{*}(X/G,\Z)}, \ \ \ \ \ \ \pi^{*}\circ\pi_{*}=\sum_{g\in G}{g^{*}},
\label{SmithIntro}
\end{equation}
where $d$ is the order of the group. We approach the problem by considering groups of prime order $p$ as a fundamental case before investigating more complicated frameworks.
Then, much information can be obtained from these equations in order to simplify problems (i) and (ii). 
We denote by $H^*_{o-t}(Y,\Z)$ the torsion part of the cohomology without the $p$-torsion part.
The first property that we can deduce is:
$$H^*_{o-t}(X,\Z)^G\simeq H^*_{o-t}(X/G,\Z).$$
Moreover, we can understand the effect of $\pi_*$ on the cup-product in $H^*_f(X/G,\Z)$ (\cite[Lemma 3.6]{Lol3}):
$$\pi_{*}(x_{1})\cdot\,\cdots\,\cdot \pi_{*}(x_{q})=p^{q-1}\pi_{*}(x_{1}\cdot\,\cdots\,\cdot x_{q}),$$
with $(x_{i})_{1\leq i \leq q}$ being elements of $H^{*}_f(X,\Z)^{G}$. 
We can rewrite the problems (i) and (ii) as follows.
\begin{itemize}
\item[(i)]
First of all, we wonder what the \emph{$p$-torsion} of $H^*(X/G,\Z)$ is.
\item[(ii)]
Let $0\leq k$. From (\ref{SmithIntro}), there exists $\alpha_{k}(X)\in\N$ such that:
\begin{equation}
\xymatrix{ 0\ar[r]&\pi_{*}(H^{k}(X,\Z))_f\ar[r] & H^{k}_f(X/G,\Z)\ar[r] & (\Z/p\Z)^{\alpha_{k}(X)}\ar[r]& 0,}
\label{MainExactSequence}
\end{equation}
with $\pi_{*}(H^{k}(X,\Z))_f:=\pi_{*}(H^{k}(X,\Z))/\tors$ (see Section \ref{nota1} (iv) for the notation).
The second problem concerns the computation of the integers $\alpha_{k}(X)$ called \emph{coefficients of surjectivity} which were introduced for the first time in \cite{Lol3}. 
Assume that $X$ is connected, then we can already notice by a direct computation that $\alpha_0(X)=1$.
\end{itemize}
The main purpose of this paper is to solve the mentioned problems when $X$ is a compact complex manifold and $G$ is an automorphism group of prime order $p$ having only isolated fixed points (Theorem \ref{main0} gives an almost complete picture apart that the $p$-torsion of odd degree cohomology groups is only provided for the combinations of two complementary degrees). First of all the provided result represents a first step before a more general theory since cyclic groups can be used as elementary bricks to construct more complicated groups. Secondly, it can be seen as a generalization of the equivariant spectral sequence technique which provides the cohomology of a quotient by a free group action.

The original motivation for studying this problem is the computation of \emph{Beauville--Bogomolov forms} for primitively symplectic orbifolds (see \cite{Lol2}, \cite{Kapfer} and \cite{Lol3}). Many aspects of the theory of irreducible symplectic manifolds have been extended for spaces admitting some singularities. In particular, the second cohomology group can still be endowed with a lattice structure given by the Beauville--Bogomolov form (see \cite{Namikawa} and \cite{Kirshner}). Moreover, the \emph{global Torelli theorem}, which allows to recover some of the geometry of the manifold from the Hodge structure and the lattice structure of the second cohomology group, has also been generalized (see \cite{Lol} and \cite{Bakker}). 
It is one of the reasons to compute the Beauville--Bogomolov forms of the known primitively symplectic complex spaces. 
Let $X$ be an irreducible symplectic manifold endowed with a symplectic automorphism group $G$ of prime order. The space $X/G$ is a primitively symplectic orbifold and the coefficient $\alpha_2(X)$ is required in order to compute its Beauville--Bogomolov form. 
Theorem \ref{main0} gives a criterion for having $\alpha_2(X)=0$.
From this result, we compute the first Beauville--Bogomolov forms of singular primitively symplectic complex spaces of dimension greater than 4. To do so, we consider the quotient of a Hilbert scheme of $n \leq p-1$ points on a K3 surface by an automorphism of order $p=5$ or $7$ (Theorem \ref{M5} and \ref{M7}).

The method to determine the integral cohomology of the quotient $X/G$ is based on the information provided by the \emph{Poincar\'e duality} applied to a resolution $\widetilde{X/G}$.
Indeed, 
the Poincar\'e duality says that $H^*_f\left(\widetilde{X/G},\Z\right)$ is an unimodular lattice (see Section \ref{remindersL} for the definition of unimodular). This information allows us to compute the coefficients of surjectivity when we are able to express $H^*_f\left(\widetilde{X/G},\Z\right)$ as function of $H^*_f(X/G,\Z)$.
In \cite{Lol3}, this technique was used when $\widetilde{X/G}$ is a blow-up. However, not all singularities can be resolved with a single blow-up. To get around this problem, we consider \emph{toric blow-ups}.
If $G$ is a linear abelian group acting on $\C^n$, the quotient $\C^n/G$ is a toric variety and we can resolve its singularities via a toric blow-up (see for instance \cite[Section 8.2]{Danilov}). When the action of $G$ on $X$ has only isolated fixed points the notion of toric blow-up can easily be extended to the quotient $X/G$ (see Section \ref{toricblowblow}). The first main result of this paper is a partial description of the integral cohomology of the toric blow-up $\widetilde{\C^n/G}$ of $\C^n/G$ (see Section \ref{CohomologyToric}). Afterwards, we apply this result to obtain a better understanding of $H^*\left(\widetilde{X/G},\Z\right)$ (see Section \ref{cortoriblowblow}).

The toric blow-ups are used in order to prove Theorem \ref{main0}
which provides $H^*(X/G,\Z)$ 
when the spectral sequence of \emph{equivariant cohomology} degenerates at the second page. This hypothesis is studied in the second important result of this paper which gives necessary and sufficient conditions, in term of the $\Z[G]$-module structure of $H^*(X,\Z)$, for having this spectral sequence which degenerates at the second page. It can also be of independent interest for readers that are concerned with spectral sequences.

\subsection{The main results}\label{intromain}
The main results are the following.
\subsubsection*{$\F[G]$-module structure}
Theorem \ref{gene19} shows the specificity of the $\F[G]$-module structure of $T\otimes\F$ for 
$T$ a free $\Z$-module endowed with the action of an automorphism group $G$ of prime order $p$. 
It is a generalization of \cite[Proposition 5.1]{SmithTh} in the case $p>19$. In particular, it allows to state all the results of \cite{Lol3} without any restriction on the prime number $p$.
\subsubsection*{Cohomology of toric blow-ups}
Proposition \ref{exactU} shows that the cohomology $H^*\left(\widetilde{\C^n/G},\Z\right)$ of a toric blow-up of $\C^n/G$ is torsion-free and concentrated in even degrees. Among others, Theorem \ref{MainCnG} computes the discriminant of the lattice generated by the exceptional cycles.  
\subsubsection*{Integral cohomology of the quotient $X/G$}
\begin{thm}\label{main0}
Let $X$ be a compact complex manifold endowed with the action of an automorphism group $G$ of prime order $p$ with only a finite number of fixed points $\eta(G)$. We assume that $H^*(X,\Z)$ is $p$-torsion-free and
that the spectral sequence of equivariant cohomology with coefficients in $\F$ degenerates at the second page. Then:
\begin{itemize}
\item[(i)]
$\alpha_k(X)=0$ for all $1\leq k \leq 2n$,
\item[(ii)]
$H^{2k}(X/G,\Z)$ is $p$-torsion-free for all $0\leq k \leq n$,
\item[(iii)]
$\tors_p \left(H^{2k+1}(X/G,\Z)\oplus H^{2n-2k+1}(X/G,\Z)\right)=(\Z/p\Z)^{\eta(G)-\ell_+^{2k}(X)}$, for all $1\leq k \leq n-1$.
 \end{itemize}
\end{thm}
The integers $\ell_+^{2k}(X)$ are among what we call the \emph{Boissi\`ere--Nieper-Wisskirchen--Sarti invariants} which characterize the $\F[G]$-module structure of $H^*_f(X,\Z)\otimes\F$. They were first introduced, in this context, by \cite{SmithTh}. We recall their definition in Section \ref{CRinvariant}. The invariant $\ell_+^{2k}(X)$ can be shortly expressed as the number of trivial irreducible representations that compose the representation given by the action of $G$ on $H^{2k}(X,\Z)$.  

Theorem \ref{main0} is a simplification of Theorem \ref{main2} where the condition on the spectral sequence can be replaced by conditions on the Boissi\`ere--Nieper-Wisskirchen--Sarti invariants. Moreover, we do not need a complex structure on all $X$, but only around the fixed points of $G$ (\emph{cf.} Remark \ref{hypopo}).

This theorem can be applied to a large class of examples as complete intersections (see Section \ref{QPM}).
\subsubsection*{Degeneration of the spectral sequence of equivariant cohomology}
Theorem \ref{degenemain} provides necessary and sufficient conditions for the degeneration of the spectral sequence of equivariant cohomology at the second page in terms of the Boissi\`ere--Nieper-Wisskirchen--Sarti invariants and the number of fixed points of $G$. It is in the same spirit as \cite[Corollary 4.3]{SmithTh}.
\subsubsection*{New Beauville--Bogomolov forms of primitively symplectic orbifolds}
A hyperk\"ahler manifold is said of \emph{$K3^{[m]}$-type} when it is equivalent by deformation to a Hilbert scheme of $m$ points on a K3 surface.
\begin{thm}\label{M5}
Let $X$ be a hyperk\"ahler manifold of $K3^{[m]}$-type and $G$ a symplectic automorphism group of order 5.
We denote $M_5^m:=X/G$. Assume that $2\leq m\leq 4$, then the Beauville--Bogomolov lattice $H^2(M_5^m,\Z)$ is isomorphic to $U(5)\oplus U^2\oplus (-10(m-1))$ and the Fujiki constant of $M_5^m$ is $\frac{5^{m-1}(2m)!}{m!2^m}$.
\end{thm}
\begin{thm}\label{M7}
Let $X$ be a hyperk\"ahler manifold of $K3^{[m]}$-type and $G$ a symplectic automorphism group of order 7.
We denote $M_7^m:=X/G$. Assume that $2\leq m\leq 6$, then the Beauville--Bogomolov lattice $H^2(M_7^m,\Z)$ is isomorphic to $U\oplus\begin{pmatrix} 4 & -3 \\ -3 & 4\end{pmatrix} \oplus (-14(m-1))$ and the Fujiki constant of $M_7^m$ is $\frac{7^{m-1}(2m)!}{m!2^m}$.
\end{thm}
We underline that these theorems give the first examples of Beauville--Bogomolov forms of singular primitively symplectic varieties of dimension strictly bigger than $4$ (see \cite{Kapfer}, \cite{Lol2} and \cite{Lol3} for other examples in dimension 4).

\subsection{Organization of the paper}
In Section \ref{ZGmodo}, we define and study the basic properties of the Boissi\`ere--Nieper-Wisskirchen--Sarti invariants which characterize the $\Z[G]$-modules and the $\F[G]$-modules when $G$ is a group of prime order. These invariants are one of the main tools of this paper. In particular, they allow to describe the second page of the spectral sequence of equivariant cohomology with coefficients in $\Z$ (\emph{cf.} Proposition \ref{equivarcoho}) and in $\F$ (\emph{cf.} Proposition \ref{Lemma3.1}). In Section \ref{bigger19}, we also remark that all the results of \cite{Lol3} can be generalized to the case when $p>19$.

In Section \ref{toritor}, we examine the integral cohomology of toric blow-ups of $\C^n/G$ and $\Pj^n/G$, for $G$ a linear action of prime order (\emph{cf.} Proposition \ref{exactU} and Theorem \ref{MainCnG}). As an application we show how toric blow-ups of isolated quotient singularities in complex spaces modify the integral cohomology (\emph{cf.} Corollaries~\ref{CorCohomology1bis}, \ref{CorCohomology2} and~\ref{CorCohomology1ter}).

In Section \ref{mainresult}, we apply the previous results to compute the integral cohomology of quotients. In Section \ref{oddmainsc}, we compute explicitly the odd coefficients of surjectivity. In Section \ref{coeffeven}, we provide a general expression of the even coefficients of surjectivity and provide an upper bound in terms of the Boissi\`ere--Nieper-Wisskirchen--Sarti invariants (\emph{cf.} Proposition \ref{ineforte2}). Section \ref{degespec} is devoted to the proof of Theorem \ref{degenemain} and Section \ref{MainThSec} to the proof of Theorem \ref{main0}.

In Section \ref{ExSe}, we give examples of applications. In Section \ref{K3Sec}, we describe the integral cohomology of $X/G$ for $X$ a K3 surface and $G$ an automorphism group of prime order with only isolated fixed points. In Section \ref{K32dege}, we give examples of spectral sequences which degenerate at the second page and Section \ref{ProofBB} is devoted to the proof of Theorems~\ref{M5} and~\ref{M7}. Finally, in Section \ref{QPM}, we show that Theorem \ref{main0} can be applied to the quotients of complete intersections.

\subsection{Notation about $\Z$-modules and integral cohomology}\label{nota1}
Let $T$ be a $\Z$-module and $G$ a finite group of prime order $p$ acting on $T$ linearly.
\begin{itemize}
 \item[(i)]
We denote by $T^G$ the invariant submodule of $T$,
\item[(ii)]
by $\tors T$ the submodule of $T$ generated by the torsion elements,
\item[(iii)]
by $\tors_p T$ the submodule of $T$ generated by the elements of $p$-torsion,
\item[(iv)]
by $T_f$ the torsion-free part of $T$, that is $T_f:=\frac{T}{\tors T}$.
\item[(v)]
We define the rank of $T$ as the rank of $T_f$ and we denote $\rk T:=\rk T_f$.
\end{itemize}
Let $Y$ be a topological space and $G$ an automorphism group. Let $k\in\N$.
\begin{itemize}
\item[(vi)]
We denote by $H^k_f(Y,\Z)$ the torsion-free part of $H^k(Y,\Z)$.
\item[(vii)] $H^*(Y,\Z)$ denotes the direct sum of all the cohomology groups, $H^{2*}(Y,\Z)$ the direct sum of all the cohomology groups of even degrees and $H^{2*+1}(Y,\Z)$ the direct sum of all the cohomology groups of odd degrees. We adopt the same notation for coefficients in $\F$.
\item[(viii)]
$h^*(Y,\Z):=\rk H^*(Y,\Z)$, $h^{2*}(Y,\Z):=\rk H^{2*}(Y,\Z)$ and $h^{2*+1}(Y,\Z):=\rk H^{2*+1}(Y,\Z)$. We adopt the same notation for coefficients in $\F$ replacing $\rk$ by $\dim_{\F}$.
\item[(ix)]
When $\tors_p H^k(Y,\Z)$ is a $\F$-vector space of finite dimension, we denote: $$t_p^k(Y):=\dim_{\F} \tors_p H^k(Y,\Z).$$
\end{itemize}
\subsection{Convention on automorphisms}
Let $\mathcal{C}$ be a category. 
Let $X$ be an object of $\mathcal{C}$.
In this paper, an \emph{automorphism} $\phi$ on $X$ refers 
to an invertible morphism of $\mathcal{C}$ from $X$ to itself.
For instance, if $X$ is a $C^{\infty}$-manifold, an automorphism $\phi$ on $X$ is a $C^{\infty}$-diffeomorphism.
\subsection{Small reminder on lattices}\label{remindersL}
We recall some basic notions regarding lattices which are used in this paper (see for example \cite[Chapter 8.2.1]{Dolgachev} for more details). 
A lattice $T$ is a free $\Z$-module endowed with a non-degenerate bilinear form $T\times T\rightarrow\Z$ (symmetric or skew-symmetric). 
Without an explicit mention, the bilinear form of the lattice will be denoted by "$\cdot$".
We denote by $\discr T$ the \emph{discriminant} of $T$,
which is the absolute value of the determinant of the bilinear form of $T$. We say that $T$ is \emph{unimodular} if its discriminant is 1.
A sublattice $N$ of $T$ is said to be \emph{primitive} if $T/N$ is torsion-free. 

Let $N^\vee:=\Hom(N,\Z)$ be the \emph{dual lattice} of $N$; it can be seen as a sub-module of $N\otimes \Q$ via the natural isomorphism $N\otimes \Q\rightarrow \Hom(N,\Q):\ x\mapsto\ (y\mapsto x\cdot y)$.
We also denote by $A_{N}:=N^\vee/N$ the \emph{discriminant group} of $N$. Let $N$ be a sublattice of a lattice $T$; the \emph{saturation} of $N$ in $T$ is the primitive sublattice $\overline{N}$ of $T$ containing $N$ and such that $\overline{N}/N$ is finite.

If $N$ is a sublattice of a lattice $T$ of the same rank, we have the basic formula:
\begin{equation}
\#\frac{T}{N}=\sqrt{\frac{\discr N}{\discr T}}.
\label{BasicLatticeTheory}
\end{equation} 
If $T$ is an unimodular lattice and $L$ is a primitive sublattice, then 
\begin{equation}
\discr L=\discr L^\bot.
\label{BasicLatticeTheory2}
\end{equation}
We still assume that $T$ is unimodular.
Then, the map defined by the composition $T\subset (L\oplus L^\bot)\otimes\Q\rightarrow L\otimes\Q$ provides the following isomorphism:
\begin{equation}
\frac{T}{L\oplus L^\bot}\rightarrow A_{L}.
\label{DiscrUni}
\end{equation}
\subsection{Main idea and sketch of the proof of Theorem \ref{main0}}\label{sketch}
Let $X$ be a complex manifold and $G$ an automorphism group of prime order. Let $M:=X/G$. Let $\pi:X\rightarrow M$ be the quotient map. Let $r:\widetilde{M} \rightarrow M$ be a resolution of singularities of $M$. To simplify the exposition, we assume in this section that $X$ is a surface.
The main idea is to deduce the integral cohomology of $M$ from the integral cohomology of $\widetilde{M}$ via the \emph{Poincar\'e duality}. 
\subsubsection{The exceptional lattice}
By Poincar\'e duality, the group $H^2_f\left(\widetilde{M},\Z\right)$ endowed with the cup-product is a unimodular lattice.  
From this fundamental information, the objective will be to obtain information on $H^2(M,\Z)$.
We define the second \emph{exceptional lattice} of $r$ by: $$N_{2,r}:=r^*\left[\pi_*(H^2(X,\Z))\right]^{\bot},$$
see \cite[Section 5.1]{Lol3}. The lattice $N_{2,r}$ is the primitive lattice generated by the exceptional divisors of $r$ and the lattice $N_{2,r}^{\bot}$ is the saturation of $r^*\left[\pi_*(H^2(X,\Z))\right]$.
The unimodularity of $H^2_f\left(\widetilde{M},\Z\right)$ connects $N_{2,r}$ and $N_{2,r}^{\bot}$ by the following equation (see (\ref{BasicLatticeTheory2})):
\begin{equation}
\discr N_{2,r}=\discr N_{2,r}^{\bot}.
\label{rapporteur1}
\end{equation}
The objective will be to compute $\discr N_{2,r}$ and $\discr N_{2,r}^{\bot}$ separately to obtain information on the coefficient of degeneration $\alpha_2(X)$ (see (\ref{MainExactSequence})).
\subsubsection{Toric blow-up}
Now, we explain how to compute $\discr N_{2,r}$. We assume that $G$ has only isolated fixed points.
Let $V:=X\smallsetminus \Fix G$ and $U:=V/G$. We have the following exact sequence:
$$\xymatrix{H^{2}\left(\widetilde{M},U,\Z\right)\ar[r]^{g}&H^{2}\left(\widetilde{M},\Z\right)\ar[r]^{f}&H^{2}(U,\Z).}$$
We are going to connect $N_{2,r}$ with $\Image g$.
To do so, we want a good understanding of $g$ and $f$. We remark that this is possible if we choose for $r$ a \emph{toric blow-up} (see Section \ref{toricblowblow}).
The main objective of Section \ref{toritor} is to understand $\Image g$ and $\Image f$. 
Corollary \ref{CorCohomology1bis} shows that $g$ is injective and $f$ is surjective. 
Moreover, Corollary~\ref{CorCohomology2} leads to the following exact sequence:
$$\xymatrix{0\ar[r]&\Image g\ar[r]&N_{2,r}\ar[r]^{f\ \ \ \ }&(\Z/p\Z)^{N}\ar[r]&0.}$$
By (\ref{BasicLatticeTheory}) it follows that:
$$\discr N_{2,r}=\frac{\discr \Image g}{p^{2N}}.$$
By Corollary \ref{CorCohomology1ter}, we have that:
$$\discr\Image g=p^{\eta(G)},$$
where $\eta(G)$ is the number of fixed points by $G$.
Moreover from Corollary \ref{CorCohomology2}, we can expressed $N$ in terms of the torsion of $H^{2}\left(\widetilde{M},\Z\right)$ and $H^{2}(U,\Z)$.
We have:
$$N=t^2_p(U)-t^2_p\left(\widetilde{M}\right),$$
(see Section \ref{nota1} (ix) for the notation).
Hence:
$$\log_p\discr N_{2,r}=\eta(G)-2t^2_p(U)+2t^2_p\left(\widetilde{M}\right).$$ 

\subsubsection{Spectral sequence of equivariant cohomology}
The torsion of $H^{2}(U,\Z)$ can be computed using the spectral sequence of the equivariant cohomology.
This is one of the objective of Section \ref{ZGmodo}. 
The second page of the spectral sequence of the equivariant cohomology can be expressed in terms of the Boissi\`ere--Nieper-Wisskirchen--Sarti invariants (see Proposition \ref{Lemma3.1} and \ref{equivarcoho}).
\subsubsection{Conclusion}
Finally, $\discr \pi_*(H^2(X,\Z))$ can also be expressed in terms of the Boissi\`ere--Nieper-Wisskirchen--Sarti invariants by Proposition \ref{vraisans19}:
$$\discr \pi_*(H^2(X,\Z))=p^{\ell_+^2(X)}.$$
In Section \ref{coeffofresolution}, we define the \emph{coefficient of resolution} $\beta_{2}(X)$ which allows to complete the computation:
$$\beta_{2}(X):=\dim_{\F}\tors\frac{H^{2}_f\left(\widetilde{M},\Z\right)}{r^*(H^{2}_f(M,\Z))}.$$
Then by (\ref{BasicLatticeTheory}):
$$\log_p\discr N_{2,r}^{\bot}=\ell_+(X)-2(\beta_{2}(X)+\alpha_{2}(X)).$$
Considering all the ingredients mentioned previously, we are able to extract from (\ref{rapporteur1}) the results of Theorem~\ref{main0}. This is done in Section \ref{mainresult}.

\subsection*{Acknowledgments} 
I want to thank my former advisor Dimitri Markouchevitch who gave me the idea of toric blow-ups for resolving quotient singularities, Thomas Megarbane for a very useful discussion about cyclotomic rings and Alessandro Duca for several English writing advises. I am also grateful to Jean-Pierre Demailly and Simon Brandhorst for very pertinent remarks during my talks, respectively in the ALKAGE workshop and in the Japanese-European symposium on symplectic varieties and moduli spaces (see respectively Remark \ref{hypopo} and Proposition \ref{mainlefsch}). I also want to thank the referees for their help to improve the writing of this paper. 

\section{Invariants of a $\Z[G]$-module}\label{ZGmodo}
\subsection{Definition of the Boissi\`ere--Nieper-Wisskirchen--Sarti invariants}\label{defiinvar}
 We recall here the definition of invariants introduced by Boissi\`ere, Nieper-Wisskirchen and Sarti in \cite{SmithTh}.
 
Let $p$ be a prime number, $T$ a $\F$-vector space of finite dimension and $G=\left\langle \phi\right\rangle$ an automorphism group of prime order $p$.
The minimal polynomial of $\phi$, as an endomorphism of $T$, divides $X^{p}-1=(X-1)^{p}\in\F[X]$, hence $\phi$ admits a Jordan normal form. We can decompose $T$ as a direct sum of some $\F[G]$-modules $N_{q}$ of dimension $q$ for $1\leq q\leq p$, where $\phi$ acts on $N_{q}$ in a suitable basis by a matrix of the following form:
$$\begin{pmatrix}
1 & 1 & & &  \\
 & \ddots & \ddots &{\fontsize{0.8cm}{0.5cm}\selectfont\text{0}} & \\
 & & \ddots & \ddots & \\
 & {\fontsize{0.8cm}{0.5cm}\selectfont\text{0}}& & \ddots & 1\\
 & & & & 1
\end{pmatrix}.$$
\begin{defi}
We define the integer $\ell_{q}(T)$ as the number of blocks of size $q$ in the Jordan decomposition of the $\F[G]$-module $T$, so that $T\simeq \oplus_{q=1}^{p} N_{q}^{\oplus \ell_{q}(T)}$. 

If $T$ is a finitely generated $\Z$-module endowed with the action of an automorphism group of prime order, we define $\ell_{q}(T):=\ell_q(T\otimes\F)$. We call the $\ell_{q}(T)$ the Boissi\`ere--Nieper-Wisskirchen--Sarti invariants.
\end{defi}
One of the uses of these invariants is the computation of the cohomology of the group $G$ with coefficients in a $\F$-vector space of finite dimension.
\begin{lemme}[\emph{cf.} \protect{\cite[Lemma~3.1]{SmithTh}}]\label{HGNq}
\leavevmode
\begin{itemize}
\item[(i)]
For $q<p$, we have $H^i(G,N_q)=\F$ for $i\geq 0$.
\item[(ii)]
For $q=p$, $H^0(G,N_p)=\F$ and $H^i(G,N_p)=0$, for all $i>0$.
\end{itemize}
\end{lemme}
\begin{nota}
Let $X$ be a topological space endowed with the action of an automorphism group $G$ of prime order $p$. Assume that $H^k(X,\F)$ has finite dimension for all $k\geq0$. Then,
for all $k\in \mathbb{N}$ and all $1\leq q\leq p$, we denote:
$$\ell_q^k(X):=\ell_q(H^k(X,\F)),\ \text{and}\ \ell_{q,t}^k(X):=\ell_q(\tors_p H^k(X,\Z)).$$ 
We also set $\ell_q^*(X):=\sum_{k\geq0}\ell_q^k(X)$.
\end{nota}
From Lemma \ref{HGNq}, we can express the cohomology of $G$ with coefficients in $H^*(X,\F)$ as follows.
\begin{prop}\label{Lemma3.1}
~
\begin{itemize}
\item[(i)]
For $i>0$, we have $\dim_{\F} H^i(G,H^k(X,\F))=\sum_{q<p}\ell_q^k(X)$, for all $k\geq0$.
\item[(ii)]
For $i=0$, we have $\dim_{\F} H^0(G,H^k(X,\F))=\sum_{q\leq p}\ell_q^k(X)$, for all $k\geq0$.
\end{itemize}
\end{prop}
 \subsection{Boissi\`ere--Nieper-Wisskirchen--Sarti invariants of a free $\Z$-module}\label{CRinvariant}
In this section we are going to provide properties of the Boissi\`ere--Nieper-Wisskirchen--Sarti invariants when $T$ is free $\Z$-module (Theorem \ref{gene19}).
To do so, we introduce some additional invariants $\ell_+(T)$ and $\ell_-(T)$ which are intrinsically defined from the $\Z$-module structure of $T$ and we show that there are equal to the Boissi\`ere--Nieper-Wisskirchen--Sarti invariants $\ell_1(T)$ and $\ell_{p-1}(T)$ apart for $p=2$. The invariants $\ell_+(T)$ and $\ell_-(T)$ are needed in the case $p=2$; for this reason they will be more convenient to use as notation. 

 Let $\xi_{p}$ be a primitive $p^\mathrm{th}$ root of unity. We set $K:=\Q(\xi_p)$ and 
$\mathcal{O}:= \Z[\xi_{p}]$.  
Let $G=\left\langle \phi\right\rangle$ be a group of prime order $p$. 
Let $A$ be an $\mathcal{O}$-ideal in $K$, the $\Z[G]$-module structure of $A$ is defined by $\phi\cdot x = \xi_{p}x$ for $x\in A$. For any $a\in A$, we denote by $(A,a)$ the $\Z$-module $A\oplus\Z$ whose $\Z[G]$-module structure is defined by $\phi\cdot(x,k)=(\xi_{p}x+ka,k)$.

The following result is a slight adaptation of \cite[Theorem 74.3]{Reiner}.
\begin{thm}\label{CurtisReiner}
Let $T$ be a free $\Z$-module of finite rank and $G$ an automorphism group of prime order $p$. Then, we have an isomorphism of $\Z[G]$-modules:
\begin{equation}
T\simeq \bigoplus_{i=1}^{r} (A_i,a_{i})\bigoplus_{i=r+1}^{r+s}A_i\bigoplus \Z^{\oplus t},
\label{CR}
\end{equation}
where $r$, $s$, $t$ are integers, $A_i$ are ideals of $\mathcal{O}$ and $a_{i}\notin (\xi_{p}-1)A_i$, for all $1\leq i\leq r+s$.
Moreover, the integers $r$, $s$ and $t$ are uniquely determined by the $\Z[G]$-module structure of $T$. 
\end{thm}
\begin{proof}
The previous result is \cite[Theorem 74.3]{Reiner} where the assumptions on the ideals $A_i$ have been modified. 
Indeed, in \cite[Theorem 74.3]{Reiner}, the ideals $A_i$ are assumed to be only $\mathcal{O}$-ideals in $K$. 
We show that the $A_i$ can be chosen as ideals of $\mathcal{O}$ without lost of generality.

Let $A$ be an $\mathcal{O}$-ideal in $K$.
Let $q$ be a common denominator of the elements of $A$ (this exists since $A$ is a finite $\Z$-module).
Then, the map $a\rightarrow qa$ is an isomorphism of $\Z\left[G\right]$-module from $A$ to $qA$, $qA$ being an ideal of $\mathcal{O}$.
Similarly, let $b\in A$ such that $b$ is not in $(\xi_p-1)A$. 
Then the map $(a,k)\rightarrow(qa,k)$ is an isomorphism of $\Z\left[G\right]$-module between $(A,b)$ and $(qA,qb)$.
\end{proof}
\begin{nota}\label{nota0}
We set $\ell_-(T):=s$ and $\ell_+(T):=t$.  

\end{nota}
\begin{lemme}\label{r}
Let $A$ be an ideal of $\mathcal{O}$ and $a\in A$ such that $a\notin (\xi_{p}-1)A$. Then:
$$(A,a)\otimes\F\simeq N_p.$$
\end{lemme}
\begin{proof}
Let $\sigma=\phi^{p-1}+\cdots+\phi+\id$.
If $p=2$, we have $\sigma\cdot(0,1)=(a,2)$. Since $a\notin 2\Z$, the image $\sigma\cdot\overline{(0,1)}$ in $(A,a)\otimes\F$ is non-zero.
Since $N_1\subset \Ker \sigma$, necessarily $(A,a)\otimes\F=N_2$.

Now we assume that $p\geq 3$.
We can compute that:
$$\sigma\cdot(0,1)=\left(\left(\sum_{j=0}^{p-2}(p-j-1)\xi_p^j\right)a,p\right).$$
If $\sigma\cdot(0,1)$ was divisible by $p$, we would have:
$$\left(\sum_{j=0}^{p-2}(p-j-1)\xi_p^j\right)a=py,$$
for $y\in A$. However, we have $(\sum_{j=0}^{p-2}(p-j-1)\xi_p^j)(\xi_{p}-1)=-p$.
Hence, we would have $a\in (\xi_{p}-1)A$ which is a contradiction. 
Therefore, $\sigma\cdot(0,1)$ is not divisible by $p$. So its image $\sigma\cdot\overline{(0,1)}$ in $(A,a)\otimes\F$ is non-zero.
Moreover, we can apply the Jordan decomposition to $(A,a)\otimes\F$. Since for all $q<p$, we have $N_q\subset \Ker \sigma$,
necessarily $(A,a)\otimes\F\supset N_p$. However since $\mathcal{O}/A$ is finite, $A$ is a free $\Z$-module of rank $p-1$ and $(A,a)$ is a free $\Z$-module of rank $p$. It follows that $\dim_{\F} (A,a)\otimes\F=\dim_{\F} N_p=p$; therefore $(A,a)\otimes\F= N_p$.
\end{proof}
Let $A$ be a non trivial ideal of $\mathcal{O}$.
Then, the group $\mathcal{O}/A$ is a finite group and the \emph{norm} $N(A)$ of $A$ is defined by $N(A):= \#\mathcal{O}/A$.
\begin{lemme}\label{number}
Let $P$ be a prime ideal of $\mathcal{O}$ such that $p$ divides $N(P)$, then $P=(\xi_p-1)$.
\end{lemme}
\begin{proof}
If $p=2$, then $\mathcal{O}=\Z$ and the prime ideals of $\mathcal{O}$ are given by $p'\Z$ with $p'$ a prime number.
If 2 divides $N(P)$ then $P=2\Z$.

Now we assume that $p\geq3$. 
First note that $P\cap\Z$ is a prime ideal of $\Z$. Hence:
$$P\cap\Z=p'\Z,$$
with $p'$ a	prime number.
Moreover, 
since $(1,\xi_p,\xi_p^2,...,\xi_p^{p-2})$ is a basis of $\mathcal{O}$ as a $\Z$-module, we have:
\begin{equation}
\mathcal{O}/p'\mathcal{O}\simeq \Z/p'\Z[\xi_p]:=(\Z/p'\Z[X])/(\overline{\Phi_p}),
\label{lemmetrivial}
\end{equation}
where $\overline{\Phi_p}\in \Z/p'\Z[X]$ is the reduction modulo $p'$ of the $p^\mathrm{th}$ cyclotomic polynomial. 
In particular:
\begin{equation}
 N(p'\mathcal{O})=p'^{p-1}.
\label{lemmetrivial2}
\end{equation}
Since $p'\mathcal{O}\subset P$, we have that
$\left. N(P) \right| N(p'\mathcal{O}).$
Since $p$ divide $N(P)$, we have by (\ref{lemmetrivial2}) that $p'=p$.
Therefore $\overline{\Phi_p}=(X-1)^{p-1}$. In particular, $(\Z/p\Z[X])/(\overline{\Phi_p})$ has only one prime ideal which is 
$(\overline{X-1})$, with $\overline{X-1}$ the reduction modulo $p$ of $X-1$. By (\ref{lemmetrivial}), it means that $\mathcal{O}/p\mathcal{O}$ has only one prime ideal which is $(\xi_p-1)/p\mathcal{O}$. Since $P/p\mathcal{O}$ is a prime ideal of $\mathcal{O}/p\mathcal{O}$,
it follows that $(\xi_p-1)/p\mathcal{O}=P/p\mathcal{O}$. Since both $(\xi_p-1)$ and $P$ are ideals containing $p\mathcal{O}$, we obtain that:
$$P=(\xi_p-1).$$
\end{proof}
\begin{lemme}\label{s}
Let $A$ be an ideal of $\mathcal{O}$ then, we have an isomorphism of $\F[G]$-module:
$$A\otimes\F\simeq N_{p-1}.$$
\end{lemme}
\begin{proof}
When $p=2$, we have  $\mathcal{O}=\Z$ and $A$ is principal, so the result follows immediately.

We assume that $p\geq3$.
We know from the proof of \cite[Proposition 5.1]{SmithTh} that: 
\begin{equation}
\mathcal{O}\otimes\F\simeq N_{p-1}.
\label{boissequa}
\end{equation}
If $p$ does not divide $N(A)$, the map $A\otimes\F\rightarrow \mathcal{O}\otimes\F\simeq N_{p-1}$ induced by the inclusion $A\subset \mathcal{O}$ is an injection. Hence by a dimension argument, it is an isomorphism. So by (\ref{boissequa}), we have our result in this case.

Now assume that $p|N(A)$.
By \cite[Theorem 18.10]{Reiner}, $A$ can be written as the product of prime ideals $P_1,...,P_i$.
Moreover by \cite[Theorem 20.2]{Reiner}, we have: 
$$N(A)=N(P_1)\cdots N(P_i).$$
Therefore, there exists $1\leq j\leq i$ such that $p$ divides $N(P_j)$. 
Then by Lemma \ref{number}, $P_j=(\xi_p-1)$. 
We can consider $A'=P_1\cdots P_{j-1}P_{j+1}\cdots P_i$. Since $P_j$ is principal, we have an isomorphism of $\Z[G]$-modules $A\simeq A'$.
By induction, we will obtain $A''$ an ideal of $\mathcal{O}$ isomorphic to $A$ as a $\Z[G]$-module and such that $p$ does not divide $N(A'')$.
Then, we are back to the first case.
\end{proof}
As a consequence, we obtain the following generalization of \cite[Proposition 5.1]{SmithTh} for $p>19$.
The main improvement is that Lemmas \ref{r} and \ref{s} are proven without assuming $\Z[\xi_p]$ being a PID. 
\begin{thm}\label{gene19}
Let $T$ be a free $\Z$-module endowed with the action of an automorphism group of prime order $p$. 
Let $r$ be the integer obtained by applying the decomposition of Theorem \ref{CurtisReiner} to $T$.
Then: 
\begin{itemize}
\item[(i)]
$\ell_{p}(T)=r$,
\item[(ii)]
if $p\geq 3$, $\ell_+(T)=\ell_{1}(T)$ and $\ell_-(T)=\ell_{p-1}(T)$,
\item[(iii)]
if $p=2$, $\ell_1(T)=\ell_+(T)+\ell_-(T)$,
\item[(iv)]
for all $2\leq q \leq p-2$, $\ell_q(T)=0$,
\item[(v)]
$\rk T=\ell_+(T)+(p-1)\ell_{-}(T)+p\ell_{p}(T)\ \text{and}\ \rk T^G=\ell_+(T)+\ell_{p}(T).$
\end{itemize}
\end{thm}
\begin{proof}
 Tensorizing (\ref{CR}) by $\F$, we obtain (i), (ii), (iii) and (iv) from Lemmas \ref{r} and \ref{s}. Statement (v) is proved in the proof of \cite[Lemma 1.8]{Mong}.
\end{proof}
\begin{rmk}
Let $T$ be a free $\Z$-module.
From the previous theorem, we see that the integers $\ell_+(T)$ and $\ell_-(T)$ are more relevant in the case $p=2$;
for this reason, it will be use preferentially in the rest of the paper. The integers $\ell_+(T)$ and $\ell_-(T)$ are also designated as Boissi\`ere--Nieper-Wisskirchen--Sarti invariants of $T$.
\end{rmk}

When $T$ is a lattice, there is an easy technique to compute the Boissi\`ere--Nieper-Wisskirchen--Sarti invariants in practice. Assume that we know
$\rk T$, $\rk T^G$, $\discr T^G$ and $\discr (T^G)^\bot$, then the Boissi\`ere--Nieper-Wisskirchen--Sarti invariants
can be computed using (\ref{BasicLatticeTheory}), Theorem  \ref{gene19} (v) and the proposition below.
\begin{prop}[\emph{cf.} \protect{\cite[Lemma~1.8]{Mong}}]\label{Mongprop}
Let $T$ be a lattice endowed the action of an automorphism group $G=\left< \phi\right>$ of prime order $p$. Then:
$$\frac{T}{T^G\oplus^{\bot}\Ker \sigma}\simeq(\Z/p\Z)^{\ell_p(T)},$$
where $\sigma= \phi^{p-1}+\cdots+\phi+\id$.
\end{prop}
As previously, these invariants can be used to compute the cohomology of $G$.
The following proposition is given in \cite[Proposition 4.1]{Lol3} for $p\leq 19$. 
\begin{lemme}\label{equiv+}
Let $T$ be a $p$-torsion-free $\Z$-module of finite rank endowed with the action of an automorphism group $G$ of prime order $p$. 
Then for all $i\in\mathbb{N}^{*}$:
\begin{itemize}
\item[(i)]
$H^{0}(G,T)=T^{G}$,
\item[(ii)]
$H^{2i-1}(G,T)=(\Z/p\Z)^{\ell_{-}(T_f)}$,
\item[(iii)]
$H^{2i}(G,T)=(\Z/p\Z)^{\ell_{+}(T_f)}$.
\end{itemize}
\end{lemme}
\begin{proof}
The proof of (i) and (ii) is identical to the one given in \cite[Proposition 4.1]{Lol3}.
The proof of statement (iii) can be copy word by word from the proof of \cite[Proposition 4.1 (iii)]{Lol3} considering Lemma \ref{r} instead of \cite[Proposition 2.2 (ii)]{Lol3}.
\end{proof}
We adopt the specific following notation when $T$ is a cohomology group.
\begin{nota}\label{nota2}
Let $X$ be a topological space endowed with the action of an automorphism group $G$ of prime order $p$. We assume that $H^k(X,\Z)$ is finitely generated for all $k\geq0$.
Then, for all $k\in \mathbb{N}$, we denote
$\ell_-^k(X)=\ell_-(H^k_f(X,\Z))$, $\ell_+^k(X)=\ell_+(H^k_f(X,\Z))$ and $\ell_{p,f}^k(X):=\ell_p(H^k_f(X,\Z)\otimes\F)$. 
\end{nota}

\begin{nota}
We also set
$\ell_{\bullet}^*(X):=\sum_{k\geq0}\ell_{\bullet}^k(X)$, $\ell_{\bullet}^{2*}(X):=\sum_{k\geq0}\ell_{\bullet}^{2k}(X)$ and $\ell_{\bullet}^{2*+1}(X):=\sum_{k\geq0}\ell_{\bullet}^{2k+1}(X)$, where $\bullet=-$,$+$ or $p,f$.
\end{nota}
All the invariants $\ell_{\bullet}^{\bullet}(X)$ are designated as Boissi\`ere--Nieper-Wisskirchen--Sarti invariants of $X$ for all possible $\bullet$.
\begin{war}
The invariant $\ell_{p,f}^k(X)$ is denoted $\ell_{p}^k(X)$ in \cite[Notation 2.8]{Lol3}. It is relevant to distinguish both of them when considering the $p$-torsion of $H^*(X,\Z)$.
\end{war}
\begin{rmk}
Note that $\ell_q(H^k_f(X,\Z)\otimes\F)=0$ for all $1<q<p-1$ because of Theorem \ref{gene19} (iv). Hence there is no interest to define
$\ell_{q,f}^k(X)$ for $1<q<p-1$. Also the terms $\ell_1(H^k_f(X,\Z)\otimes\F)$ and $\ell_{p-1}(H^k_f(X,\Z)\otimes\F)$ are not relevant because they coincide when $p=2$ leading to a loss of information. This is why we use $\ell_-^k(X)$ and $\ell_+^k(X)$ instead (see the proposition below for the different relations between the invariants).
\end{rmk}
The universal coefficient theorem and Theorem \ref{gene19} provide the following proposition.
\begin{prop}\label{torsioninvar}
For all $k\in \mathbb{N}$, we have:
\begin{itemize}
\item[(i)]
when $p\geq3$, 
$\ell_1^k(X)=\ell_{+}^k(X)+\ell_{1,t}^k(X)+\ell_{1,t}^{k+1}(X)$
and for $p=2$,
$\ell_1^k(X)=\ell_{+}^k(X)+\ell_{-}^k(X)+\ell_{1,t}^k(X)+\ell_{1,t}^{k+1}(X),$
\item[(ii)]
when $p\geq3$, 
$\ell_{p-1}^k(X)=\ell_{-}^k(X)+\ell_{p-1,t}^k(X)+\ell_{p-1,t}^{k+1}(X)$,
\item[(iii)]
for all prime number $p$,
$\ell_{p}^k(X)=\ell_{p,f}^k(X)+\ell_{p,t}^k(X)+\ell_{p,t}^{k+1}(X)$,
\item[(iv)]
for all $2\leq q\leq p-2$, $\ell_q^k(X)=\ell_{q,t}^k(X)+\ell_{q,t}^{k+1}(X)$.
\end{itemize}
\end{prop}
\begin{proof}
By the universal coefficient theorem, we have, for all $k\in \N$, the following two exact sequences:
\begin{equation}
\xymatrix{ 0\ar[r]&\Ext^1(H_{k-1}(X),\Z) \ar[r] & H^{k}(X,\Z)\ar[r]& \Hom(H_k(X),\Z)\ar[r]&0,}
\label{crote1}
\end{equation}
\begin{equation}
\xymatrix{ 0\ar[r]&\Ext^1(H_{k-1}(X),\F) \ar[r] & H^{k}(X,\F)\ar[r]& \Hom(H_k(X),\F)\ar[r]&0.}
\label{crote2}
\end{equation}
Moreover, the maps of these exact sequences are morphisms of $G$-module (see for instance \cite[p.~196]{Hatcher}).
Furthermore, we have the canonical isomorphisms (which respect the $G$-module structure):
$$\Ext^1(H_{k-1}(X),\Z)\simeq \Ext^1(H_{k-1}(X),\F)\simeq \Hom(\tors_p H_{k-1}(X),\F),$$
$$\Hom(H_k(X),\F)\simeq \Hom(H_k(X)_f,\F)\oplus \Hom(\tors_p H_k(X),\F),$$
$$\Hom(H_k(X)_f,\F)\simeq \Hom(H_k(X)_f,\Z)\otimes\F\simeq \Hom(H_k(X),\Z)\otimes\F.$$
Moreover from (\ref{crote1}), we have the following isomorphisms of $G$-module:
$$\tors_p H^k(X,\Z)\simeq \Ext^1(H_{k-1}(X),\Z)\ \text{and}\ H^k_f(X,\Z)\simeq \Hom(H_k(X),\Z).$$
So, we obtain with (\ref{crote2}) the following isomorphism of $G$-module:
\begin{equation}
H^{k}(X,\F)\simeq \tors_p H^{k}(X,\Z)\oplus H^k_f(X,\Z)\otimes\F\oplus \tors_p H^{k+1}(X,\Z).
\label{rapporteur0}
\end{equation}
Then, the results follow from Proposition \ref{gene19}.
\end{proof}
We can also express the cohomology of $G$ with coefficients in $H^k(X,\Z)$.
\begin{prop}\label{equivarcoho}
For all $k\geq0$ and $i>0$, we have:
\begin{itemize}
\item[(i)]
$H^{0}(G,H^k(X,\Z))=H^k(X,\Z)^{G}$, \textit{i.e.} $\tors_p H^{0}(G,H^k(X,\Z))=(\Z/p\Z)^{\sum_{q\leq p}\ell_{q,t}^k(X)}$,
\item[(ii)]
$H^{2i-1}(G,H^k(X,\Z))=(\Z/p\Z)^{\ell_{-}^k(X)+\sum_{q< p}\ell_{q,t}^k(X)}$,
\item[(iii)]
$H^{2i}(G,H^k(X,\Z))=(\Z/p\Z)^{\ell_{+}^k(X)+\sum_{q< p}\ell_{q,t}^k(X)}$.
\end{itemize}
\end{prop}
\begin{proof}
Since $H^k_f(X,\Z)$ and $\tors_p H^k(X,\Z)$ are stable under the action of $G$, we have:
$$H^{i}(G,H^k(X,\Z))=H^{i}(G,H^k_f(X,\Z))\oplus H^{i}(G,\tors_p H^k(X,\Z)),$$
for all $k$ and $i$ in $\N$.
Then, the result follows from Lemmas \ref{HGNq} and \ref{equiv+}.
\end{proof}
  \subsection{Boissi\`ere--Nieper-Wisskirchen--Sarti invariants and Lefschetz fixed point theorem}
 The following proposition is due to Simon Brandhorst. Among others, it has allowed to simplify the proof of Theorem \ref{degenemain} via Corollary \ref{equiv1}.
 \begin{prop}\label{mainlefsch}
 Let $X$ be a compact complex manifold and $G$ an automorphism group of prime order $p$. Then
 $$\chi(\Fix G)=\ell_+^{2*}(X)+\ell_-^{2*+1}(X)-\ell_+^{2*+1}(X)-\ell_-^{2*}(X),$$
 where $\chi(\Fix G)$ is the Euler characteristic of $\Fix G$.
 \end{prop}
 \begin{proof}
 Let $n=\dim_{\C}X$ and $g\in G$ be a generator, the Lefschetz fixed point theorem provides:
 \begin{equation}
 \chi(\Fix G)=\sum_{k=0}^{2n}(-1)^{k}\tr(g_{|H^k(X,\R)}).
 \label{plefsch}
 \end{equation}
By Theorem \ref{CurtisReiner}, we have an isomorphism of $\Z[G]$-modules:
$$H^k_f(X,\Z)\simeq \left(\bigoplus_{j=1}^{r} (A_j,a_{j})\right)\oplus\left(\bigoplus_{i=1}^{s}A_i\right)\oplus \Z^{\oplus t},$$
where $r$, $s$, $t$ are integers, $A_i$, $A_j$ are ideals of $\mathcal{O}$ and $a_{j}\notin (\xi_{p}-1)A_j$, for all $1\leq j\leq r$ and $1\leq i\leq s$.
By definition of the Boissi\`ere--Nieper-Wisskirchen--Sarti invariants (Notation \ref{nota0}, \ref{nota2} and Theorem \ref{gene19} (i)), it can be rewritten:
$$
H^k_f(X,\Z)\simeq\left(\bigoplus_{j=1}^{\ell_{p,f}^k(X)} (A_j,a_{j})\right)\oplus\left(\bigoplus_{i=1}^{\ell_-^k(X)}A_i\right)\oplus \Z^{\oplus \ell_+^k(X)}.$$

Hence, after tensorizing by $\R$, we obtain an isomorphisms of $\R[G]$-module:
$$H^k(X,\R)\simeq\left(\bigoplus_{j=1}^{\ell_{p,f}^k(X)} (A_j,a_{j})\otimes\R\right)\oplus\left(\bigoplus_{i=1}^{\ell_-^k(X)}A_i\otimes\R\right)\oplus \R^{\oplus \ell_+^k(X)}.$$
Since $A\otimes\R=\mathcal{O}\otimes\R$ for all ideal $A$, we have:
$$H^k(X,\R)\simeq\left(\bigoplus_{j=1}^{\ell_{p,f}^k(X)} (\mathcal{O},a_{j})\otimes\R\right)\oplus(\mathcal{O}\otimes\R)^{\ell_-^k(X)}\oplus \R^{\oplus \ell_+^k(X)}.$$
Furthermore, the minimal polynomial of $g_{|\mathcal{O}\otimes\R}$ is $X^{p-1}+X^{p-2}+...+1$ and the minimal polynomial of $g_{|(\mathcal{O},a_{j})\otimes\R}$ is $X^p-1$ for all $j$. 
Because of their degrees the previous polynomials are also the characteristic polynomials and we obtain:
$$\tr (g_{|\mathcal{O}\otimes\R})=-1\ \text{and}\ \tr(g_{|(\mathcal{O},a_{j})\otimes\R})=0.$$
Then (\ref{plefsch}) becomes:
$$\chi(\Fix G)=\sum_{k=0}^{2n}(-1)^{k}(\ell_+^k(X)-\ell_-^k(X)).$$
 \end{proof}
 \begin{rmk}
 Note that Proposition \ref{mainlefsch} remains true if we assume that $X$ is a $2n$-dimensional compact connected orientable $\C^{\infty}$-manifold and $G$ an automorphism group of prime order $p$ with only isolated fixed points such that $X$ behaves around the fixed points of $G=\left\langle g\right\rangle$ as a complex manifold and $G$ as a biholomorphic morphism group (i.e.: for each fixed point $x$ there exists an open set $x\in U_x$ and a diffeomorphism $f:U_x\rightarrow \mathcal{U}_x\subset \C^n$ such that $f\circ g\circ f^{-1}$ is biholomorphic). 
 \end{rmk}
\begin{rmk}\label{geneHurwitz}
Note that we can find again the generalized Hurwitz formula $\chi(X)=p\chi(X/G)+(1-p)\chi(\Fix G)$ by combining Proposition \ref{mainlefsch} and Theorem \ref{gene19} (v).
\end{rmk}
The next corollary is a direct consequence of Proposition \ref{mainlefsch}. It will be very useful in practice to deal with the spectral sequence of equivariant cohomology (see Section \ref{degespec}).
\begin{cor}\label{equiv1}
Let $X$ be a compact complex manifold and $G$ an automorphism group of prime order $p$ with only $\eta(G)$ isolated fixed points. Then:
$$\eta(G)=\ell_+^{2*}(X)+\ell_-^{2*+1}(X)\ \Leftrightarrow\ \ell_+^{2*+1}(X)=\ell_-^{2*}(X)=0.$$
\end{cor}

 \subsection{Application to the degeneration of the spectral sequence of equivariant cohomology}\label{equivarsection}
Let $X$ be a CW-complex and $G$ an automorphism group on $X$ ($G$ permutes the cells). 
Let $EG\rightarrow BG$ be an universal $G$-bundle in the category of CW-complexes. 
Denote by $X_{G}=EG\times_{G}X$ the orbit space for the diagonal action of $G$ on the product $EG\times X$ 
and $f: X_{G}\rightarrow BG$ the map induced by the projection onto the first factor. 
The map $f$ is a locally trivial fiber bundle with typical fiber $X$ and structure group $G$.
We define the $G$-equivariant cohomology of $X$ with coefficients in a ring $\Lambda$ (in this paper $\Lambda$ is $\Z$ or $\F$ with $p$ a prime number) by $H_{G}^{*}(X,\Lambda):=H^{*}(EG\times_{G}X,\Lambda)$. In particular, observe that if $G$ acts freely on $X$, then the canonical map $EG\times_{G} X\rightarrow X/G$ is a homotopy equivalence and so we obtain $H_{G}^{*}(X,\Lambda)=H^{*}(X/G,\Lambda).$
Moreover, the Leray--Serre spectral sequence associated to the map $f$ gives a spectral sequence converging to the equivariant cohomology (see \cite[Chapter~VII, Section~7]{cohogroup}):
$$E_{2}^{p,q}:=H^{p}(G;H^{q}(X,\Lambda))\Rightarrow H_{G}^{p+q}(X,\Lambda).$$ 
 
In particular, the degeneration of this spectral sequence at the second page has interesting consequences (see for instance \cite[Theorem 1.1]{Lol3}). We also recall the following result of Boissi\`ere, Nieper--Wisskirchen and Sarti which will be used several time in this paper.
\begin{prop}[\emph{cf.} \protect{\cite[Corollary~4.3]{SmithTh}}]\label{BNSFormula}
Let $X$ be a compact connected orientable $C^{\infty}$-manifold of dimension $n$ and $G$ an automorphism group of prime order $p$.
If the spectral sequence of equivariant cohomology with coefficients in $\F$ degenerates at the second page, then:
$$h^*(\Fix G,\F)=\sum_{1\leq q < p} \ell_{q}^*(X).$$
\end{prop}
It is interesting to compare the previous equation with the equation provided by the Lefschetz fixed point theorem (Proposition \ref{mainlefsch}).
\begin{cor}\label{lefschetz2}
 Let $X$ be a compact complex manifold and $G$ an automorphism group of prime order $p$. We assume that the spectral sequence of equivariant cohomology with coefficients in $\F$ degenerates at the second page. Then:
 $$h^{2*}(\Fix G,\F)\geq \ell_+^{2*}(X)+\ell_-^{2*+1}(X),\ \text{and}\ h^{2*+1}(\Fix G,\F)\geq \ell_+^{2*+1}(X)+\ell_-^{2*}(X).$$
 If moreover $H^*(X,\Z)$ is $p$-torsion-free:
 $$h^{2*}(\Fix G,\F)=\ell_+^{2*}(X)+\ell_-^{2*+1}(X),\ \text{and}\ h^{2*+1}(\Fix G,\F)=\ell_+^{2*+1}(X)+\ell_-^{2*}(X).$$
\end{cor}
\begin{proof}
By Propositions \ref{BNSFormula} and \ref{torsioninvar} (resp. if $H^*(X,\Z)$ is $p$-torsion-free), we get:
$$h^*(\Fix G,\F)=\sum_{1\leq q < p} \ell_{q}^*(X)\underset{(=)}{\geq} \ell_{+}^*(X)+\ell_{-}^*(X).$$
 Adding or subtracting the equation of Proposition \ref{mainlefsch}, we obtain our result.
\end{proof}
We have the following particular case when $G$ has only isolated fixed points (this result will be useful in Section \ref{degespec}).
\begin{cor}\label{degenerate1}
Let $X$ be a compact complex manifold and $G$ an automorphism group of prime order $p$ with only isolated fixed points. We assume that the spectral sequence of equivariant cohomology with coefficients in $\F$ degenerates at the second page. Then:
$$\ell_+^{2*+1}(X)=\ell_-^{2*}(X)=0.$$
\end{cor}
Using Propositions \ref{Lemma3.1}, \ref{equivarcoho} and \ref{torsioninvar}, we can compare the spectral sequences with coefficients in $\Z$ and $\F$. The end of this section is devoted to the proof of the following proposition.
\begin{prop}\label{degeequ}
 Let $X$ be a CW-complex endowed with the action of an automorphism group $G$ of prime order $p$. We assume that $H^k(X,\Z)$ is finitely generated for all $k\in\N$.The following statements are equivalent.
 \begin{itemize}
 \item[---]
 The spectral sequence of equivariant cohomology of $(X,G)$ with coefficients in $\F$ degenerates at the second page.
 \item[---]
 The spectral sequence of equivariant cohomology of $(X,G)$ with coefficients in $\Z$ degenerates at the second page.
 \end{itemize}
\end{prop} 
From now and until the end of this section, $X$ is a CW-complex endowed with the action of an automorphism group $G$ of prime order $p$, with $H^k(X,\Z)$ finitely generated for all $k\geq0$. We first introduce a tool to measure the degeneration of our spectral sequence.
We define the \emph{dimensions of degeneration} as follows. Let $E_2^{d,q}:=H^{d}(G,H^q(X,\Z))$ (resp. $\overline{E}_2^{d,q}:=H^{d}(G,H^q(X,\F))$) be the second page of the spectral sequence of equivariant cohomology with coefficients in $\Z$ (resp. in $\F$). Similarly, we denote by $E_r^{d,q}$ (resp. $\overline{E}_r^{d,q}$) the space of the $r^\mathrm{th}$ page of the sequence of equivariant cohomology with coefficients in $\Z$ (resp. in $\F$). By Proposition \ref{Lemma3.1}, $\overline{E}_2^{d,q}$ is a $\F$-vector space of finite dimension. By Proposition \ref{equivarcoho}, $E_2^{d,q}$ is also a $\F$-vector space when $d>0$. Moreover, we set $t_p E_2^{0,k}:=\dim_{\F}\tors_p E_2^{0,k}$ and $t_p E_\infty^{0,k}:=\dim_{\F}\tors_p E_\infty^{0,k}$.
\begin{defi}\label{defidegene}
We define the $k^\mathrm{th}$ \emph{dimension of degeneration for $\Z$-coefficients} (resp. \emph{for $\F$-coefficients}) by the non-negative integer $u_k(X)$:
$$u_k(X)=\dim_{\F}\left(\bigoplus_{d+q=k,\, d>0} E_2^{d,q}\right)-\dim_{\F}\left(\bigoplus_{d+q=k,\, d>0} E_{\infty}^{d,q}\right)+\left(t_p E_2^{0,k}-t_p E_\infty^{0,k}\right).$$
(resp. $\overline{u}_k(X)$ such that): 
$$\overline{u}_k(X)=\dim_{\F}\left(\bigoplus_{d+q=k} \overline{E}_2^{d,q}\right)-\dim_{\F}\left(\bigoplus_{d+q=k} \overline{E}_{\infty}^{d,q}\right).$$
\end{defi}
\begin{rmk}\label{explication}
As, we will see below,
the role of these integers will be to measure the distance of the spectral sequence from being degenerate at the second page.
Note that the terms 
with $d=0$ is slightly different between $u_k(X)$ and $\overline{u}_k(X)$. It is because, the $E_r^{0,k}$ are not $\F$-vector spaces.
We can write $E_r^{0,k}=(E_r^{0,k}/\tors_p E_r^{0,k})\oplus\tors_p E_r^{0,k}$. When $r\geq 2$, the terms $\rk(E_r^{0,k}/\tors_p E_r^{0,k})$ do not bring information on the degeneration of the spectral sequence. Indeed, we have $\rk(E_r^{0,k}/\tors_p E_r^{0,k})=\rk (E_{r+1}^{0,k}/\tors_p E_{r+1}^{0,k})$ for all $r\geq2$. 
Hence, it is more practical to avoid these terms in the definition of $u_k(X)$.
\end{rmk}
For a better understanding, we recall the shape of the second page of our spectral sequence.
\begin{equation}
   \xymatrix{\vdots&\vdots&\vdots&\vdots &\vdots &\vdots &\cdots\\
		E_2^{0,4}\ar[rrd]&E_2^{1,4}\ar[rrd]&E_2^{2,4}\ar[rrd]&E_2^{3,4}\ar[rrd] &E_2^{4,4}\ar[rrd] &E_2^{5,4}&\cdots\\
	E_2^{0,3}\ar[rrd]&E_2^{1,3}\ar[rrd]&E_2^{2,3}\ar[rrd]&E_2^{3,3}\ar[rrd] &E_2^{4,3}\ar[rrd] &E_2^{5,3}&\cdots\\
	E_2^{0,2}\ar[rrd]&E_2^{1,2}\ar[rrd]&E_2^{2,2}\ar[rrd]&E_2^{3,2}\ar[rrd] &E_2^{4,2}\ar[rrd] &E_2^{5,2} &\cdots\\
	E_2^{0,1}\ar[rrd]&E_2^{1,1}\ar[rrd]&E_2^{2,1}\ar[rrd]&E_2^{3,1}\ar[rrd] &E_2^{4,1}\ar[rrd] &E_2^{5,1} &\cdots\\
	E_2^{0,0}&E_2^{1,0}&E_2^{2,0}&E_2^{3,0} &E_2^{4,0} &E_2^{5,0} &\cdots }
	\label{spectral}
\end{equation}	
\begin{lemme}\label{degelemmabis}
The spectral sequence of equivariant cohomology with coefficients in $\Z$ $($resp. in $\F)$ degenerates at the second page if and only if 
$u_k(X)=0$ $($resp. $\overline{u}_k(X)=0)$ for all $k\in\mathbb{N}$.
\end{lemme}
\begin{proof}
We prove the result for the spectral sequence of equivariant cohomology with coefficient in $\Z$ (the proof is identical for the 
spectral sequence with coefficient in $\F$). One direction is trivial, if the spectral sequence degenerates at the second page then $u_k(X)=0$ for all $k\in \mathbb{N}$.

Now, we assume that the spectral sequence does not degenerate at the second page and we will show that there exist $k\in \N$ such that $u_k(X)>0$. First note that $(\dim_{\F} E_{r}^{d,q})_r$ is a decreasing sequence for all $d>0$ and $q\geq0$.
If the spectral sequence does not degenerate at the second page then there exists a differential at a page $E_r$, $r\geq2$, which is not trivial:
$$\delta_{(d-r,q+r-1)}^{(d,q)}: E_r^{d-r,q+r-1}\rightarrow E_r^{d,q}.$$

If $\delta^{(d-r,q+r-1)}_{(d,q)}$ is not trivial it means that $\Image  \delta^{(d-r,q+r-1)}_{(d,q)}\neq 0$. However 
$$E_{r+1}^{d,q}=\Ker\delta^{(d+r,q-r+1)}_{(d,q)}/\Image \delta^{(d,q)}_{(d-r,q+r-1)},$$
with 
the differential $\delta^{(d+r,q-r+1)}_{(d,q)}:E_r^{d,q}\rightarrow E_r^{d+r,q-r+1}$.
Since $\Image \delta^{(d,q)}_{(d-r,q+r-1)} \neq 0$, this means that
\[\dim_{\F} E_{r+1}^{d,q}<\dim_{\F}\Ker\delta^{(d+r,q-r+1)}_{(d,q)}\leq\dim_{\F} E_{r}^{d,q}\leq \dim_{\F} E_{2}^{d,q}.\]
Hence $\dim_{\F} E_{\infty}^{d,q}<\dim_{\F} E_2^{d,q}$. So $u_{d+q}(X)>0$.
\end{proof}
We also mention the following property which will be used in Section \ref{degespec}.
\begin{lemme}\label{lol12}
Let $k>1$. If $\overline{u}_{k-1}(X)=\overline{u}_{k+1}(X)=0$ then $\overline{u}_k(X)=0$.
\end{lemme}
\begin{proof}
We have to show that $\overline{E}_2^{d,q}=\overline{E}_{\infty}^{d,q}$ for all $d$, $q$ such that $d+q=k$.
To do so we prove recursively that $\overline{E}_2^{d,q}=\overline{E}_r^{d,q}$ for all $r\geq2$.
Of course it is true for $r=2$. Assume that it is true for $r$, we will prove it for $r+1$.
We have:
$$\xymatrix{\overline{E}_r^{d-r,q+r-1}\ar[r]^{\ \ \ \ \ \delta_a}&\overline{E}_r^{d,q}\ar[r]^{\delta_l\ \ \ \ \ }&\overline{E}_r^{d+r,q-r+1}}.$$
Since $\overline{u}_{k-1}(X)=0$, we have $\delta_a=0$. Indeed, if it was not the case,
$\Ker \delta_a$ would be a proper sub-vector space of $\overline{E}_r^{d-r,q+r-1}$.
Hence, we would have: $$\dim_{\F}\overline{E}_{\infty}^{d-r,q+r-1}\leq \dim_{\F}\overline{E}_{r+1}^{d-r,q+r-1}<\dim_{\F}\overline{E}_{r}^{d-r,q+r-1}\leq\dim_{\F}\overline{E}_{2}^{d-r,q+r-1} .$$
This contradicts $\overline{u}_{k-1}(X)=0$. With exactly the same argument, we have that $\delta_l=0$. 
It follows that $\overline{E}_r^{d,q}=\overline{E}_{r+1}^{d,q}$, which ends the proof.
\end{proof}
\begin{lemme}\label{degelemmater}
Assume that there exists $n\in \N$ such that $H^k(X,\Z)=0$ for all $k>n$. Let $k\geq n+2$, if $u_{k-1}(X)=u_{k+1}(X)=0$ then $u_k(X)=0$.
\end{lemme}
\begin{proof}
Let $r\geq2$.
The proof is exactly the same as the one of Lemma \ref{lol12} with an additional complication which imposes $k\geq n+2$.
Indeed, as explain in Remark \ref{explication}, knowing that $u_{k-1}(X)=0$ does not imply anything on $E_r^{0,k-1}/\tors_p E_r^{0,k-1}$ since 
this group is not "encoded" in $u_{k-1}(X)$.
In particular, we can have $u_{k-1}(X)=0$ and the differential $\delta:E_r^{0,k-1}\rightarrow E_r^{r,k-r}$ which is not trivial (note that even if we would have added the term $\rk (E_2^{0,k-1}/\tors_p E_2^{0,k-1})$ in the definition of $u_{k-1}(X)$ it would not have solved this problem).
To avoid this problem, we need that $E_2^{0,k-1}=0$. This is the case if $k\geq n+2$ by Proposition \ref{equivarcoho} (i).
\end{proof}
We can express the difference between the two kinds of dimensions of degeneration using our invariants. 
\begin{lemme}\label{ZFdege}
We have: 
$$\overline{u}_k(X)=u_{k}(X)+u_{k+1}(X),$$
for all $k\in\mathbb{N}$.
\end{lemme}
\begin{proof}
The equation of the lemma is a consequence of the universal coefficient theorem and our computations of the cohomology of the group $G$.
By Proposition \ref{Lemma3.1}:
\begin{equation}
\overline{u}_k(X)=\sum_{i=0}^k\sum_{q<p}\ell_q^i(X)+\ell_p^k(X)-\dim_{\F} H^k(X_G,\F).
\label{ub1}
\end{equation}
By Proposition \ref{equivarcoho}:
\begin{equation}
u_{2k}(X)=\sum_{i=0}^{k-1}\ell_+^{2i}(X)+\sum_{i=0}^{k-1}\ell_-^{2i+1}(X)+\sum_{i=0}^{2k}\sum_{q<p}\ell_{q,t}^i(X)+\ell_{p,t}^{2k}(X)-t_p^{2k}(X_G),
\label{XG1}
\end{equation}
\begin{equation}
u_{2k+1}(X)=\sum_{i=0}^{k}\ell_-^{2i}(X)+\sum_{i=0}^{k-1}\ell_+^{2i+1}(X)+\sum_{i=0}^{2k+1}\sum_{q<p}\ell_{q,t}^i(X)+\ell_{p,t}^{2k+1}(X)-t_p^{2k+1}(X_G),
\label{XG2}
\end{equation}
where the notation $t_p^k$ is defined in Section \ref{nota1} (ix).
Moreover the universal coefficient theorem provides a relation between the two kinds of dimensions of degeneration.
Indeed, by (\ref{rapporteur0}), we have for $k\in\N$:
$$\dim_{\F} H^k(X_G,\F)=t_p^k(X_G)+t_p^{k+1}(X_G)+\rk H_f^k(X_G,\Z).$$
By Proposition \ref{equivarcoho} (i) and Theorem \ref{gene19} (v), we have:
$$\rk H_f^k(X_G,\Z)=\rk H_f^k(X,\Z)^G=\ell_+^{k}(X)+\ell_{p,f}^{k}(X).$$
It follows from (\ref{XG1}), (\ref{XG2}) and Proposition \ref{torsioninvar} that:
\begin{align*}
\dim_{\F} H^{2k}(X_G,\F)&=t_p^{2k}(X_G)+t_p^{2k+1}(X_G)+\rk H_f^{2k}(X_G,\Z).\\
&=\sum_{i=0}^{k-1}\ell_+^{2i}(X)+\sum_{i=0}^{k-1}\ell_-^{2i+1}(X)+\sum_{i=0}^{2k}\sum_{q<p}\ell_{q,t}^i(X)+\ell_{p,t}^{2k}(X)-u_{2k}(X)\\
&-u_{2k+1}(X)+\sum_{i=0}^{k}\ell_-^{2i}(X)+\sum_{i=0}^{k-1}\ell_+^{2i+1}(X)+\sum_{i=0}^{2k+1}\sum_{q<p}\ell_{q,t}^i(X)+\ell_{p,t}^{2k+1}(X)\\
&+\ell_+^{2k}(X)+\ell_{p,f}^{2k}(X)\\
&=\sum_{i=0}^{2k}\sum_{q<p}\ell_q^i(X)+\ell_{p}^{2k}(X)-u_{2k}(X)-u_{2k+1}(X).
\end{align*}
We obtain the same result for odd degrees, hence we have for all $k\geq0$:
$$\dim_{\F} H^{k}(X_G,\F)=\sum_{i=0}^k\sum_{q<p}\ell_q^i(X)+\ell_{p}^{k}(X)-u_{k}(X)-u_{k+1}(X).$$
Then, we obtain our result by (\ref{ub1}).
\end{proof}
Finally, Proposition \ref{degeequ} is a direct consequence of Lemmas \ref{ZFdege} and \ref{degelemmabis}.
 \subsection{Application to the cohomology of quotients when $p>19$}\label{bigger19}
 Because of Theorem \ref{gene19}, all the statements of \cite{Lol3} can be stated without any restriction on the prime number $p$.
In particular, we have the following propositions which will be used several times in the paper.
\begin{prop}\label{vraisans19}
Let $X$ be a compact connected orientable manifold and $G$ be an automorphism group of prime order $p$. We denote by $\pi:X\rightarrow X/G$ the quotient map. Let $T$ be a unimodular sublattice of $H^*(X,\Z)$ stable under the action of $G$; then:
$$\discr \pi_*(T)_f=p^{\ell_+(T)}.$$
\end{prop} 
\begin{proof}
It is a direct consequence of \cite[Corollary 3.8]{Lol3}, since we have shown in Theorem \ref{gene19} (v) that $\rk T^G-\ell_{p}(T)=\ell_+(T)$ for all prime number $p$.
\end{proof}
\begin{prop}\label{vraisans19ter}
Let $X$ be a compact connected orientable manifold of dimension $n$ and $G$ be an automorphism group of prime order $p$. For all $0\leq k\leq n$, we have:
$$\alpha_k(X)+\alpha_{n-k}(X)\leq\ell_+^k(X).$$
Moreover, if $X/G$ is smooth the previous inequality becomes an equality.
\end{prop}
\begin{proof}
As before, 
it is a direct consequence of \cite[Corollary 3.10]{Lol3}, since we have shown in Theorem \ref{gene19} (v) that 
$\rk H^{k}_f(X,Z)^G-\ell^k_{p,f}(X)=\ell_+^k(X)$
for all prime number $p$.
\end{proof}
  \begin{rmk}\label{vraisans19bis}
  More generally, in all the statements \cite[Propositions 2.9, 3.14 and 5.2]{Lol3},
	the assumption on the prime number $p$ can be removed.
  \end{rmk}
The main results of \cite{Lol3} can also be stated without any restriction on the prime number $p$. We recall the definition of \emph{simple fixed points} of an automorphism group of a complex manifold \cite[Definition 5.9]{Lol3}. A fixed point $x$ is said to be \emph{simple} if the local action of $G$ around $x$ corresponds to the action of one of the diagonal matrices $\diag(1,...,1,\xi_p,...,\xi_p)$ in $0\in \C^n$ with $\xi_p$ a $p^\mathrm{th}$ root of the unity. If $\Fix G$ has several connected components, we denote by $\codim \Fix G$ the codimension of the component of higher dimension.
 \begin{thm}\label{recall1}
Let $X$ be a compact complex manifold of dimension $n$ and $G$ an automorphism group of prime order $p$. Let $c:=\codim \Fix G$.
Assume that
\begin{itemize}
\item[(i)]
$H^*(X,\Z)$ and  $H^*(\Fix G,\Z)$ are $p$-torsion-free,
\item[(ii)]
$h^*(\Fix G,\Z)\geq \ell_{+}^*(X)+\ell_{-}^*(X),$
\item[(iii)]
all fixed points of $G$ are simple,
\item[(iv)]
$c\geq\frac{n}{2}+1$.
\end{itemize}
Then, for all $2n-2c+1\leq k \leq2c-1$,
the $k^\mathrm{th}$ coefficient of surjectivity $\alpha_k(X)$ of $X$ vanishes.
\end{thm}
The condition on the codimension of $\Fix G$ can be slightly improved (we can allow that $\Fix G$ contains a connected component of codimension $\left\lceil \frac{n}{2}\right\rceil$).
\begin{thm}\label{recall3}
Let $X$ be a compact complex manifold of dimension $n$ and $G$ an automorphism group of prime order $p$. 
Assume that:
\begin{itemize}
\item[(i)]
$H^*(X,\Z)$ and  $H^*(\Fix G,\Z)$ are $p$-torsion-free.
\item[(ii)]
$h^{2*}(\Fix G,\Z)\geq \ell_{+}^{2*}(X)+\ell_-^{2*+1}(X)$ if $n$ is even and
$h^{2*+1}(\Fix G,\Z)\geq \ell_{+}^{2*+1}(X)+\ell_-^{2*}(X)$ if $n$ is odd.
\item[(iii)]
All fixed points of $G$ are simple.
\item[(iv)]
$\codim \Fix G=\left\lceil \frac{n}{2}\right\rceil$.
\item[(v)]
When $n$ is even, we assume that $\Fix G$ contains only one connected component of dimension $\frac{n}{2}$, denoted by $\Delta$.
Let $j:\Fix G\hookrightarrow X$ be the inclusion. Let $\left[\Delta\right]:=j_*(1)$, with $j_*:H^{0}(\Delta,\Z)\rightarrow H^{n}(X,\Z)$. We assume that $\left[\Delta\right]$ is not of $p$-torsion and is not divisible by $p$.
\end{itemize}
Then the $n^\mathrm{th}$ coefficient of surjectivity $\alpha_n(X)$ of $X$ vanishes.
\end{thm}
\begin{proof}[Proof of Theorems \ref{recall1} and \ref{recall3}]
This is the version of \cite[Theorems 6.1 and 6.12]{Lol3} without the assumption $p\leq19$.
Because of Propositions \ref{vraisans19}, \ref{vraisans19ter} and Remark \ref{vraisans19bis}, the proof given in \cite{Lol3} remains valid without any assumption on $p$.

Moreover, the K\"ahler condition can also be removed. Indeed, this condition was needed to compute the integral cohomology of blow-ups of $X$ using \cite[Theorem 7.31]{Voisin} which is stated for a K\"ahler manifold. However \cite[Theorem 7.31]{Voisin} can be replaced by \cite[Theorem 4.1]{Li} which provides the integral cohomology of a blow-up of an almost complex manifold without any K\"ahler assumption.

Furthermore, the condition on $\left[\Delta\right]$ in \cite[Theorem 6.12]{Lol3} have been slightly weakened (in \cite{Lol3}, $\Delta$ is denoted by $\Sigma$). In \cite[Theorem 6.12]{Lol3}, it is assumed that $\left[\Delta\right]$ is primitive in $H^{n}(X,\Z)$; however assuming that $\left[\Delta\right]$ is not divisible by $p$ is enough without any change in the proof. Indeed, this condition on $\left[\Delta\right]$ is only used in \cite[Lemmas 6.13]{Lol3} to show that $H^n(X\smallsetminus \Fix G,\Z)$ is $p$-torsion-free.
\end{proof}
\begin{rmk}\label{conditionth19}
By Corollary \ref{lefschetz2}, the numerical condition (ii) of the previous theorems can be replaced by "the spectral sequence of equivariant cohomology with coefficients in $\F$ degenerates at the second page".
\end{rmk}
\section{Toric blow-up}\label{toritor}
\subsection{Reminders on toric geometry}
Our main references are \cite{Danilov}, \cite{Fulton} and \cite{Oda}.

Let $M$ be a lattice.
A set $\sigma$ in $M_{\mathbb{Q}}:=M\otimes\mathbb{Q}$ is called a \emph{cone}, if there exist finitely many vectors $v_{1},\ldots,v_{n}\in M$ such that $\sigma=\mathbb{Q}^{+}v_{1}+\cdots+\mathbb{Q}^{+}v_{n}$. 
The \emph{dimension} of $\sigma$ is defined to be the dimension of the subspace $\Vect(\sigma)$. 
If $H\subset M_{\mathbb{Q}}$ is a hyperplane which contains the origin $0\in  M_{\mathbb{Q}}$ such that $\sigma$ lies in one of the closed half-spaces of $M_{\mathbb{Q}}$ bounded by $H$, then the intersection $\sigma\cap H$ is again a cone which is called a \emph{face} of $\sigma$. If $\left\{0\right\}$ is a face of $\sigma$, we say that $\sigma$ has a \emph{vertex} at $0$. Let $\sigma$ be a cone, we denote by $\sigma^{\vee}:=\left\{\left.f\in \Hom(M_{\mathbb{Q}},\mathbb{Q})\ \right|\  f(\sigma)\geq0\right\}$ the \emph{dual cone} of $\sigma$.
\begin{defi}[\emph{Fan}]
Let $M$ be a lattice.
A \emph{fan} $\Sigma$ in $M_{\mathbb{Q}}$ is a finite set of cones which satisfies the following conditions:
\begin{itemize}
\item[---]
every cone $\sigma\in\Sigma$ has a vertex at 0;
\item[---]
if $\tau$ is a face of a cone $\sigma\in\Sigma$, then $\tau\in\Sigma$;
\item[---]
if $\sigma, \sigma'\in \Sigma$, then $\sigma\cap \sigma'$ is a face of both $\sigma$ and $\sigma'$.
\end{itemize}
\end{defi}

\begin{defi}[\emph{Vocabulary on Fan}]
Let $\Sigma$ be a fan in $M_{\mathbb{Q}}$. 
\begin{itemize}
\item[---]
We define the \emph{support} of $\Sigma$ by $\left|\Sigma\right|=\bigcup_{\sigma\in\Sigma}\sigma$ and we say that $\Sigma$ is \emph{complete} if $\left|\Sigma\right|=M_{\mathbb{Q}}$. 
\item[---]
Let $\Sigma'$ be another fan such that every cone of $\Sigma'$ is contain in some cone of $\Sigma$ and $\left|\Sigma'\right|=\left|\Sigma\right|$; we call $\Sigma'$ a \emph{subdivision} of $\Sigma$.
\item[---]
Let $\sigma\in \Sigma$ be a cone. We say that $\sigma$ is \emph{regular according to $M$} if it is generated by a subset of a basis of $M$.
We say that $\Sigma$ is \emph{regular according to $M$} if every $\sigma\in\Sigma$ is regular according to $M$.
\item[---]
Let $\Sigma(1)$ be the cones of $\Sigma$ of dimension 1.
The
\emph{Stanley-Reisner ring} $R_{\Sigma}$ of $\Sigma$ is the commutative ring generated by elements $x_{\sigma}$ with $\sigma\in\Sigma(1)$ and the relations $x_{\sigma_1}x_{\sigma_2}\cdot\cdot\cdot x_{\sigma_r}=0$ for all distinct $\sigma_1,\sigma_2,\ldots, \sigma_r\in\Sigma(1)$ that do not generate a cone of $\Sigma$.
\end{itemize}
\end{defi}

\begin{defi}[\emph{Affine toric variety}]
Let $M$ be a lattice and $\sigma\subset M_{\mathbb{Q}}$ a cone.
We denote by $\mathbb{C}[\sigma\cap M]$ the set of all the expressions $\sum_{m\in\sigma\cap M} a_{m}x^{m}$ with almost all $a_{m}=0$.
The affine scheme $\Spec \mathbb{C}[\sigma\cap M]$ is called an \emph{affine toric variety}; it is denoted by $X_{\sigma}$.
\end{defi}
\begin{defi}[\emph{Toric variety}]
Let $M$ and $N$ be lattices dual to one another, and let $\Sigma$ be a fan in $N_{\mathbb{Q}}$.
With each cone $\sigma\in\Sigma$ we associate an affine toric variety $X_{\sigma^{\vee}}=\Spec\mathbb{C}[\sigma^{\vee}\cap M]$. By \cite[Section 2.6.1]{Danilov}, if $\tau$ is a face of $\sigma$, then $X_{\tau^{\vee}}$ can be identified with an open subvariety of $X_{\sigma^{\vee}}$. These identifications allow us to glue together the $X_{\sigma^{\vee}}$ (as $\sigma$ ranges over $\Sigma$) to form a variety, which is denoted by $X_{\Sigma}$ and is called the \emph{toric variety associated to $\Sigma$ and $N$}.
\end{defi}

\begin{thm}[\emph{cf.} \protect{\cite[Theorems 1.10 and 1.11]{Oda}}]\label{toric}
Let $N$ be a lattice and let $X_{\Sigma}$ be a toric variety determined by a fan $\Sigma$ in $N_{\mathbb{Q}}$. Then:
\begin{itemize}
\item[(1)]
$X_{\Sigma}$ is complete if and only if $\Sigma$ is complete. 
\item[(2)]
$X_{\Sigma}$ is smooth if and only if $\Sigma$ is regular according to $N$.
\end{itemize}
\end{thm}
In this section we will be interested in the integral cohomology of toric varieties, for this reason we recall the following well known result.
\begin{thm}[\emph{cf.} \protect{\cite[Theorem 10.8]{Danilov}}]\label{smoothcompletecoho}
Let $X$ be a complete smooth toric variety. Then $H^*(X,\Z)$ is torsion-free and concentrated in even degrees.
\end{thm}
As mentioned in \cite[Section 5.7]{Danilov}, toric varieties can also be characterized by a torus action.
For each $\sigma\in \Sigma$, we have an inclusion $\C[\sigma^{\vee}\cap M]\subset\C[M]$ which provides an action of $\mathcal{T}:=\Spec \C[M]\simeq (\C^*)^m$ on $X_{\sigma^{\vee}}$, with $m=\dim M_{\Q}$. These actions are compatible on each $X_{\sigma^{\vee}}$ and give an action of $\mathcal{T}$ on all $X_\Sigma$. It can be shown that this property characterizes toric varieties: if a normal variety $X$ contains a torus $\mathcal{T}$ as dense open subvariety, and the action of $\mathcal{T}$ on itself extends to an action on $X$ then $X$ is of the form $X_\Sigma$.
Considering this action, an important tool is the equivariant cohomology of $X_\Sigma$ under the action of $\mathcal{T}$ denoted by $H^*_{\mathcal{T}}(X_\Sigma,\Z)$ (see Section \ref{equivarsection} for the definition of equivariant cohomology).
This tool is used to prove the next proposition which will be needed in Section \ref{CohomologyToric}.
 \begin{prop}\label{genetoric}
 Let $\Sigma$ and $\Sigma'$ be two regular fans such that $\Sigma'\subset \Sigma$. Let $X_\Sigma$ and $X_{\Sigma'}$ be the associated toric varieties.
Let $j: X_{\Sigma'}\hookrightarrow X_\Sigma$ be the natural embedding. Assume that $H^*\left(X_\Sigma,\Z\right)$ and $H^*\left(X_{\Sigma'},\Z\right)$ are concentrated in even degrees, then $j^*:H^*\left(X_\Sigma,\Z\right)\rightarrow H^*\left(X_{\Sigma'},\Z\right)$ is surjective.
 \end{prop}
 \begin{proof}
 The proof is based on a well known result that we will recall here. 
For each 1-dimensional cone $\sigma\in \Sigma$, we can construct a $\mathcal{T}$-invariant divisor $V(\sigma)$ in $X_\Sigma$ (see for instance \cite[Section 3]{Fulton}). Let $R(\Sigma)$ be the Stanley-Reisner ring of $\Sigma$. 
 Then there is a natural ring morphism:
 $$c_\Sigma: R(\Sigma)\rightarrow H^*_{\mathcal{T}}(X_\Sigma,\Z)$$
 defined on the 1-dimensional cones, by sending $x_\sigma$ to the equivariant cohomology class associated to the divisor $V(\sigma)$.
 When $X_\Sigma$ is smooth, $c_\Sigma$ is an isomorphism (this is due to \cite[Section 2.2]{Brion2} where the result is stated with rational coefficients, however the proof is still true considering integral coefficients; see \cite[Sections 2.3 and 3]{Franz2} for the same result in a more general setting).
 Then, we have the following commutative diagram:
   $$\xymatrix{ R(\Sigma)\eq[r]_{c_\Sigma\ \ \ \ \ \ }\ar@{->>}[d]& H_{\mathcal{T}}^*(X_\Sigma,\Z)\ar@{->>}[r]\ar[d]^{j^*}&H^*(X_\Sigma,\Z)\ar[d]^{j^*}\\
   R(\Sigma')\eq[r]_{c_{\Sigma'}\ \ \ \ \ }& H_{\mathcal{T}}^*(X_{\Sigma'},\Z)\ar@{->>}[r]&H^*(X_{\Sigma'},\Z).}$$
   The map $R(\Sigma)\rightarrow R(\Sigma')$ is surjective because $\Sigma'\subset \Sigma$. Moreover, the maps $H_\mathcal{T}^*(X_\Sigma,\Z)\rightarrow H^*(X_\Sigma,\Z)$ and $H_\mathcal{T}^*(X_{\Sigma'},\Z)\rightarrow H^*(X_{\Sigma'},\Z)$ are surjective because the cohomologies of $X_\Sigma$ and $X_{\Sigma'}$ are concentrated in even degrees (see for instance \cite[Lemma 5.1]{Franz}).
By commutativity of the diagram, it follows that $j^*:H^*(X_{\Sigma},\Z)\rightarrow H^*(X_{\Sigma'},\Z)$ is surjective.
 \end{proof}
\subsection{Definition of toric blow-ups}\label{toricblowblow}
\subsubsection*{For toric varieties}

Let $X_\Sigma$ be a toric variety. By Theorem \ref{toric} (2), to resolve the singularities of $X_\Sigma$, we only need to consider $\Sigma'$ a regular subdivision of $\Sigma$. In \cite[Section 8.2]{Danilov}, Danilov explains that we can always find a regular subdivision $\Sigma'$ such that $f:X_{\Sigma'}\rightarrow X_\Sigma$ verifies the following properties:
\begin{itemize}
\item[---]
$f$ is an isomorphism over the smooth locus of $X_\Sigma$;
\item[---]
$f$ is a projective morphism.
\end{itemize}
Such a transformation $f$ is called a \emph{toric blow-up} of $X_\Sigma$. 

\begin{ex}\label{exemple}
The variety $\mathbb{C}^{n}$ is an affine toric variety given by the lattice $M:=\Z^{n}$ and the cone $\sigma=(\mathbb{Q}^{+})^n$. 
Let $G\subset \GL(n,\C)$ be a finite abelian group with $n>1$. Since $G$ is finite abelian group, there exists a basis of $\C^n$ in which all the elements of $G$ can be expressed as diagonal matrices.  
Let $g\in G$, we have $g=\diag(\xi^{a_1},...,\xi^{a_1})$,
with $\xi$ an $m^\mathrm{th}$ root of unity and $0\leq a_i \leq m$.
Let $M^{g}:=\left\{\left.(x_1,...,x_n)\in M\ \right|\ \sum_{i=0}^nx_ia_i\equiv 0\mod m\right\}$ and $M^G:=\cap_{g\in G}M^{g}$.
The quotient $\mathbb{C}^{n}/G$ is also an affine toric variety given by the lattice $M^{G}$ and the cone $\sigma$.
In particular, $\mathbb{C}^{n}/G$ is the toric variety associated to the fan $\Sigma$ in $(M^{G})^{\vee}_{\mathbb{Q}}$ containing the cone $\sigma^{\vee}$ and all its faces. Hence, the singularities of $\mathbb{C}^{n}/G$ can be resolved by a toric blow-up.
\end{ex}

\subsubsection*{For isolated quotient singularities}
\begin{defi}\label{quotientPoint}
Let $X$ be a topological space. A point $x\in X$ is called an \emph{isolated complex quotient point} if there exists $W_x\subset X$ an open set containing $x$, $\mathcal{W}\subset \C^{n}$ with $n>1$ an open set containing $0$, $G$ an automorphism group of finite order on $\mathcal{W}$ which acts freely on $\mathcal{W}\smallsetminus \left\{0\right\}$
and $h:W_x\rightarrow \mathcal{W}/G$ an homeomorphism with $h(x)=\overline{0}$. The quadruple $(W_x,\mathcal{W},G,h)$ is called a \emph{local uniformizing system} of $x$.  
\end{defi}
Let $X$ be a topological space. Let $x\in X$ be an isolated complex quotient point and $(W_x,\mathcal{W},G,h)$ a local uniformizing system of $x$ such that $G$ is an abelian group.
 Let $f:\widetilde{\C^n/G}\rightarrow \C^n/G$ be a toric blow-up of $\C^n/G$.
We can glue $f^{-1}(\mathcal{W}/G)$ to $X$ in $\mathcal{W}\smallsetminus \left\{0\right\}/G\simeq W_x\smallsetminus \left\{x\right\}$. We obtain a map $\widetilde{X}\rightarrow X$ that we call a \emph{toric blow-up} of $X$ in $x$ according to $(W_x,\mathcal{W},G,h)$. 

\begin{rmk}
As defined here, a toric blow-up is not unique. In order to define the toric blow-up of an orbifold with any singularities, we would need to require some universal properties for our toric blow-up.
\end{rmk}

\subsection{Integral cohomology of toric blow-ups of $\C^n/G$, with $G$ a cyclic group}\label{CohomologyToric}
In this section, we use the notation of Example \ref{exemple}.
Let $n>1$ and $G\subset \GL(n,\C)$ be a finite group of prime order $p$ with only $0$ as fixed point. 
The action of $G$ extends to an action on $\mathbb{P}^n$. Let $\xi_p$ be a $p^\mathrm{th}$ root of the unity. Without loss of generality we can assume that $G=\left\langle \phi\right\rangle$ with $\phi=\diag (\xi_p^{\alpha_1},\ldots,\xi_p^{\alpha_n})$ and $1\leq \alpha_i\leq p-1$ for all $i\in \left\{1,\ldots,n\right\}$.
It provides:
$$\xymatrix@R0pt{ \Pj^n\ar[r]&\Pj^n\\
(a_0:a_1:\cdots :a_n)\ar@{|->}[r]&(a_0:\xi_p^{\alpha_1}a_1:\cdots :\xi_p^{\alpha_n}a_n).
}$$
Then $(1:0:\cdots:0)$ is an isolated fixed point of the action of $G$ on $\Pj^n$; we denote $0:=(1:0:\cdots:0)$.
If we identify $\C^n$ with the chart $a_0\neq 0$, this action on $\Pj^n$ is an extension of the action on $\C^n$. 
Hence, if we denote by $\overline{\Sigma}$ the fan of $\Pj^n/G$ the natural embedding $\C^n/G\hookrightarrow\Pj^n/G$ corresponds to the inclusion of the fans $\Sigma\subset\overline{\Sigma}$.

Let $f:\widetilde{\mathbb{C}^{n}/G}\rightarrow \mathbb{C}^{n}/G$ and $\overline{f}:\widetilde{\mathbb{P}^{n}/G}\rightarrow \mathbb{P}^{n}/G$ be toric blow-ups of $\mathbb{C}^{n}/G$ and $\mathbb{P}^{n}/G$ respectively such that they coincide in $0$; that is the cone $\sigma^\vee$ is subdivided in the same way in $\overline{\Sigma}$ and $\Sigma$. Let $\overline{\Sigma}'$ and $\Sigma'$ be the fans of $\widetilde{\mathbb{P}^{n}/G}$ and $\widetilde{\mathbb{C}^{n}/G}$, it follows an inclusion $\Sigma'\subset\overline{\Sigma}'$ and an open embedding $j:\widetilde{\mathbb{C}^{n}/G}\hookrightarrow \widetilde{\mathbb{P}^{n}/G}$.  

Let $\overline{\Sigma}_*=\overline{\Sigma}\smallsetminus\left\{\sigma^{\vee}\right\}$; it is the fan of $\mathbb{P}^{n}/G\smallsetminus \left\{0\right\}$. We consider $\overline{\Sigma}'_*$ the regular subdivision of $\overline{\Sigma}_*$ such that $\overline{\Sigma}'_*\subset \overline{\Sigma}'$; it is the fan of $\widetilde{\mathbb{P}^{n}/G}^*:=\widetilde{\mathbb{P}^{n}/G}\smallsetminus \overline{f}^{\ -1}(0)=\widetilde{\mathbb{P}^{n}/G}\smallsetminus j(f^{-1}(0))$. We denote $i: \widetilde{\mathbb{P}^{n}/G}^*\hookrightarrow \widetilde{\mathbb{P}^{n}/G}$ the inclusion.
We also denote $(\mathbb{C}^{n}/G)^*:=\widetilde{\mathbb{C}^{n}/G}\smallsetminus f^{-1}(0)=\mathbb{C}^{n}/G\smallsetminus\left\{0\right\}$. 
\begin{rmk}\label{CohoDani}
By Theorem \ref{smoothcompletecoho},
 $H^*\left(\widetilde{\mathbb{P}^{n}/G},\Z\right)$ is torsion-free and concentrated in even degrees.
\end{rmk}
\begin{prop}\label{equiU}
We have:
\begin{equation*}
 H^{k}\left((\mathbb{C}^{n}/G)^*,\Z\right)=
\begin{cases}
\Z & \text{for } k= 0,\\
0 & \text{for } k = 2m-1,\ 1\leq m\leq n-1,\\
\Z/p\Z & \text{for } k = 2m,\ 1\leq m\leq n-1,\\
\Z & \text{for } k = 2n-1,\\
0 & \text{for } k=2n.
\end{cases}
\end{equation*}
 \end{prop}
 \begin{proof}
The space $(\C^{n}/G)^*$ is homotopy equivalent to $\mathbb{S}^{2n-1}/G$. The space $\mathbb{S}^{2n-1}/G$ is the lens space $L(n-1,p)$ and its cohomology is well known.
\end{proof}
 As a consequence, we also obtain the relative cohomology $H^*(\mathbb{C}^{n}/G,(\mathbb{C}^{n}/G)^*,\Z)$.
 \begin{prop}\label{exactU}
 We have:
\begin{equation*}
 H^{k}\left(\mathbb{C}^{n}/G,(\mathbb{C}^{n}/G)^*,\Z\right)=
\begin{cases}
0 & \text{for } k\in\left\{0,1\right\},\\
\Z/p\Z & \text{for } k = 2m-1,\ 2\leq m\leq n,\\
0 & \text{for } k = 2m,\ 1\leq m\leq n-1,\\
\Z & \text{for } k=2n.
\end{cases}
\end{equation*}
 \end{prop}
 \begin{proof}
The quotient $\C^n/G$ is contractible, so its cohomology is concentrated in degree 0:
\[H^*(\C^n/G,\Z)=H^0(\C^n/G,\Z)=\Z.\] 
Then, the result can be obtained from the exact sequence of relative cohomology of the pair $(\mathbb{C}^{n}/G,(\mathbb{C}^{n}/G)^*)$:
 $$\xymatrix{ H^k(\mathbb{C}^{n}/G,\Z)\ar[r]&H^k((\mathbb{C}^{n}/G)^*,\Z)\ar[r]&H^{k+1}(\mathbb{C}^{n}/G,(\mathbb{C}^{n}/G)^*,\Z)\ar[r]&H^{k+1}(\mathbb{C}^{n}/G,\Z),}$$
Hence for $k\geq 1$, we have:
$$H^k((\mathbb{C}^{n}/G)^*,\Z)\simeq H^{k+1}(\mathbb{C}^{n}/G,(\mathbb{C}^{n}/G)^*,\Z).$$
We obtain our result for $k>1$ using Proposition \ref{equiU}.
It only remains to compute $H^{0}(\mathbb{C}^{n}/G,(\mathbb{C}^{n}/G)^*,\Z)$ and $H^{1}(\mathbb{C}^{n}/G,(\mathbb{C}^{n}/G)^*,\Z)$.
The previous exact sequence provides:
 $$\xymatrix{H^0(\mathbb{C}^{n}/G,(\mathbb{C}^{n}/G)^*,\Z)\ar[r]& H^0(\mathbb{C}^{n}/G,\Z)\ar[r]&H^0((\mathbb{C}^{n}/G)^*,\Z)\ar[r]&H^{1}(\mathbb{C}^{n}/G,(\mathbb{C}^{n}/G)^*,\Z)\ar[r]&0.}$$
Since $H^0(\mathbb{C}^{n}/G,\Z)\rightarrow H^0((\mathbb{C}^{n}/G)^*,\Z)$ is not trivial, we have $H^0(\mathbb{C}^{n}/G,(\mathbb{C}^{n}/G)^*,\Z)=0$.
Furthermore, since $H^1(\mathbb{C}^{n}/G,(\mathbb{C}^{n}/G)^*,\Z)$ is torsion-free, we have $H^1(\mathbb{C}^{n}/G,(\mathbb{C}^{n}/G)^*,\Z)=0$.
 \end{proof}
 \begin{prop}\label{torsion}
 The cohomology groups $H^*\left(\widetilde{\mathbb{P}^{n}/G}^*,\Z\right)$ and $H^*\left(\widetilde{\mathbb{C}^{n}/G},\Z\right)$ are torsion-free and concentrated in even degree.
 \end{prop}
 \begin{proof}
We first show that $H^{2k-1}\left(\widetilde{\mathbb{P}^{n}/G}^*,\Z\right)=H^{2k-1}\left(\widetilde{\mathbb{C}^{n}/G},\Z\right)=0$ for all $1\leq k\leq n$.
The main idea of the proof is to apply Proposition \ref{equiU} and Remark \ref{CohoDani} to the long exact sequence of relative cohomology of the couples $\left(\widetilde{\mathbb{P}^{n}/G},\widetilde{\mathbb{P}^{n}/G}^*\right)$ and $\left(\widetilde{\mathbb{C}^{n}/G},(\mathbb{C}^{n}/G)^*\right)$.
\subsubsection*{\normalfont\emph{Assume first that $k<n$}}
We consider the following commutative diagram of embeddings:
\begin{equation}
\xymatrix{ &(\mathbb{C}^{n}/G)^*\ar@{^{(}->}[rd]\ar@{_{(}->}[ld]&\\
\widetilde{\mathbb{P}^{n}/G}^*\ar@{^{(}->}[rd]& & \ar@{_{(}->}[ld]\widetilde{\mathbb{C}^{n}/G}\\
&\widetilde{\mathbb{P}^{n}/G}&}
\label{diagram}
\end{equation}
which induces the following commutative diagram on the cohomology:
$$\xymatrix@R10pt@C10pt{& 0\ar@{=}[d]& & & \\
H^{2k-1}\left(\widetilde{\mathbb{P}^{n}/G},\widetilde{\mathbb{P}^{n}/G}^*,\Z\right)\ar@{=}[d]\ar[r] &H^{2k-1}\left(\widetilde{\mathbb{P}^{n}/G},\Z\right)\ar[d]\ar[r]&H^{2k-1}\left(\widetilde{\mathbb{P}^{n}/G}^*,\Z\right)\ar[r]\ar[d] & H^{2k}\left(\widetilde{\mathbb{P}^{n}/G},\widetilde{\mathbb{P}^{n}/G}^*,\Z\right)\ar@{=}[d]\\
H^{2k-1}\left(\widetilde{\mathbb{C}^{n}/G},(\mathbb{C}^{n}/G)^*,\Z\right)\ar[r] &H^{2k-1}\left(\widetilde{\mathbb{C}^{n}/G},\Z\right)\ar[r]&H^{2k-1}((\mathbb{C}^{n}/G)^*,\Z)\ar[r] & H^{2k}\left(\widetilde{\mathbb{C}^{n}/G},(\mathbb{C}^{n}/G)^*,\Z\right).\\
& & 0\ar@{=}[u] &
}$$
By Remark \ref{CohoDani}, we have: $$H^{2k-1}\left(\widetilde{\mathbb{P}^{n}/G},\Z\right)=0$$ and by Proposition \ref{equiU}, we have: $$H^{2k-1}((\mathbb{C}^{n}/G)^*,\Z)=0.$$
By commutativity of the diagram, the map $H^{2k-1}\left(\widetilde{\mathbb{P}^{n}/G}^*,\Z\right)\rightarrow H^{2k}\left(\widetilde{\mathbb{P}^{n}/G},\widetilde{\mathbb{P}^{n}/G}^*,\Z\right)$ is necessarily 0. 
It follows that: $$H^{2k-1}\left(\widetilde{\mathbb{P}^{n}/G}^*,\Z\right)=0.$$
For the same reason, the map $H^{2k-1}\left(\widetilde{\mathbb{C}^{n}/G},(\mathbb{C}^{n}/G)^*,\Z\right)\rightarrow H^{2k-1}\left(\widetilde{\mathbb{C}^{n}/G},\Z\right)$ is also trivial.
So $H^{2k-1}\left(\widetilde{\mathbb{C}^{n}/G},\Z\right)=0$.
\subsubsection*{\normalfont \emph{When $k=n$}}
The exact sequence obtained from (\ref{diagram}), 
is slightly different. (The coefficient ring of the cohomology is $\Z$; we do not write it to avoid a too large diagram.)
\small
$$
\xymatrix@R10pt@C10pt{
&0\ar@{=}[d]& & &  \Z\ar@{=}[d]& \\
H^{2n-1}\left(\widetilde{\mathbb{P}^{n}/G},\widetilde{\mathbb{P}^{n}/G}^*\right)\ar@{=}[d]\ar[r] &H^{2n-1}\left(\widetilde{\mathbb{P}^{n}/G}\right)\ar[d]\ar[r]&H^{2n-1}\left(\widetilde{\mathbb{P}^{n}/G}^*\right)\ar[r]^{\delta}\ar[d] & H^{2n}\left(\widetilde{\mathbb{P}^{n}/G},\widetilde{\mathbb{P}^{n}/G}^*\right)\ar@{=}[d]\ar[r]^{\ \ \ \alpha}&H^{2n}\left(\widetilde{\mathbb{P}^{n}/G}\right)\ar[r]\ar[d] &0 \\
H^{2n-1}\left(\widetilde{\mathbb{C}^{n}/G},\left(\mathbb{C}^{n}/G\right)^*\right)\ar[r]^{\ \ \ \ \gamma} &H^{2n-1}\left(\widetilde{\mathbb{C}^{n}/G}\right)\ar[r]&H^{2n-1}\left(\left(\mathbb{C}^{n}/G\right)^*\right)\ar[r]^\beta & H^{2n}\left(\widetilde{\mathbb{C}^{n}/G},\left(\mathbb{C}^{n}/G\right)^*\right)\ar[r]&0. &\\
& & \Z\ar@{=}[u] & & &
}$$
 \normalsize
 We have $H^{2n}\left(\widetilde{\mathbb{P}^{n}/G},\Z\right)=\Z$ because $\widetilde{\mathbb{P}^{n}/G}$ is smooth and compact;  $\widetilde{\mathbb{C}^{n}/G}$ and $\widetilde{\mathbb{P}^{n}/G}^*$ being open sub-manifolds of a compact complex manifold of dimension $n$, we get $H^{2n}\left(\widetilde{\mathbb{P}^{n}/G}^*,\Z\right)=H^{2n}\left(\widetilde{\mathbb{C}^{n}/G},\Z\right)=0$.
Necessarily, we have:
\begin{equation}
 \Image \alpha=\Z.
 \label{alpha}
 \end{equation}
 Hence: 
\begin{equation}
 \Image \beta=\Z.
 \label{gammi}
 \end{equation}
This proves:
\begin{equation}
H^{2n}\left(\widetilde{\mathbb{P}^{n}/G},\widetilde{\mathbb{P}^{n}/G}^*,\Z\right)=\Z,
\label{2nP*}
\end{equation}
 and 
 \begin{equation}
 \delta=0. 
 \label{delta}
 \end{equation}
 So $H^{2n-1}\left(\widetilde{\mathbb{P}^{n}/G}^*,\Z\right)=0$. By commutativity of the diagram, we also have $\gamma=0$.
 So (\ref{gammi}) implies that $H^{2n-1}\left(\widetilde{\mathbb{C}^{n}/G},\Z\right)=0$.
\subsubsection*{\normalfont\emph{The cohomology is torsion-free}}
The varieties $\widetilde{\mathbb{P}^{n}/G}^*$ and $\widetilde{\mathbb{C}^{n}/G}$ are smooth toric varieties. Moreover we have seen that their integral cohomology of odd degrees is trivial. Therefore, \cite[Proposition 1.5]{Franz} shows that $H^*\left(\widetilde{\mathbb{P}^{n}/G}^*,\Z\right)$ and $H^*\left(\widetilde{\mathbb{C}^{n}/G},\Z\right)$ are torsion-free.
 \end{proof}
\begin{rmk}
In \cite{Barthel}, more general results related to the cohomology with rational coefficients of toric varieties can be found.
\end{rmk}
\begin{lemme}\label{adoc}
The natural map $H^{2n}(\C^n/G,(\C^n/G)^*,\Z)\rightarrow H^{2n}\left(\widetilde{\C^n/G},(\C^n/G)^*,\Z\right)$ is an isomorphism.
\end{lemme}
\begin{proof}
Indeed, using Propositions \ref{equiU}, \ref{exactU} and \ref{torsion}, the relative cohomology exact sequences of the pairs $(\C^n/G,(\C^n/G)^*)$ and $\left(\widetilde{\C^n/G},(\C^n/G)^*\right)$ provide the following commutative diagram:
$$\xymatrix@R15pt{0\ar[r]&H^{2n-1}((\C^n/G)^*,\Z)\ar@{=}[d]\ar[r]& H^{2n}(\C^n/G,(\C^n/G)^*,\Z)\ar[d]\ar[r] & 0\\
0\ar[r]&H^{2n-1}((\C^n/G)^*,\Z)\ar[r]& H^{2n}\left(\widetilde{\C^n/G},(\C^n/G)^*,\Z\right)\ar[r] & 0.}$$
\end{proof}
The last result of this section makes precise how a toric blow-up modifies the cohomology. 
\begin{thm}\label{MainCnG}
 We consider the following exact sequence:
$$\xymatrix{ H^{2k}\left(\widetilde{\mathbb{P}^{n}/G},\widetilde{\mathbb{P}^{n}/G}^*,\Z\right)\ar[r]^{\ \ \ \ g'^{2k}} & H^{2k}\left(\widetilde{\mathbb{P}^{n}/G},\Z\right)\ar[r] & H^{2k}\left(\widetilde{\mathbb{P}^{n}/G}^*,\Z\right)\ar[r]& H^{2k+1}\left(\widetilde{\mathbb{P}^{n}/G},\widetilde{\mathbb{P}^{n}/G}^*,\Z\right),}$$
with $1\leq k\leq n-1$.
Then:
\begin{itemize}
\item[(i)]
$g'^{2k}$ is injective,
\item[(ii)]
$H^{2k+1}\left(\widetilde{\mathbb{P}^{n}/G},\widetilde{\mathbb{P}^{n}/G}^*,\Z\right)=0$, 
\item[(iii)]
$H^{2k}\left(\widetilde{\mathbb{P}^{n}/G},\widetilde{\mathbb{P}^{n}/G}^*,\Z\right)$ is torsion-free and $H^{2n}\left(\widetilde{\mathbb{P}^{n}/G},\widetilde{\mathbb{P}^{n}/G}^*,\Z\right)=\Z$,
\item[(iv)]
$\Image g'^{2k}\oplus \Image g'^{2n-2k}$ is a sublattice of $H^{2k}\left(\widetilde{\mathbb{P}^{n}/G},\Z\right)\oplus H^{2n-2k}\left(\widetilde{\mathbb{P}^{n}/G},\Z\right)$ of discriminant $p^2$.
\item[(v)]
If $n$ is even, then $\Image g'^{n}$ is a sublattice of $H^{n}\left(\widetilde{\mathbb{P}^{n}/G},\Z\right)$ of discriminant $p$.
\end{itemize} 
\end{thm}

\subsubsection*{Outline of the proof}

The statements (i), (ii) and (iii) will be easily obtained from Proposition \ref{torsion} and the long exact sequence of relative cohomology of $\left(\widetilde{\mathbb{P}^{n}/G},\widetilde{\mathbb{P}^{n}/G}^*\right)$. We provide an outline of the proof of (v); the proof of (iv) being similar. 
\begin{itemize}
\item[---]
Using the long exact sequence of relative cohomology of $\left(\widetilde{\mathbb{P}^{n}/G},\widetilde{\mathbb{P}^{n}/G}^*\right)$ and Proposition \ref{torsion}, we can show that:
$\Image g'^{n}$ is primitive in $H^{n}\left(\widetilde{\mathbb{P}^{n}/G},\Z\right)$.
\item[---]
However, using long exact sequence of relative cohomology of $\left(\widetilde{\mathbb{P}^{n}/G},(\C^n/G)^*\right)$ and Proposition~\ref{equiU}, we can show that
$\Image g'^{n}\oplus (\Image g'^{n})^{\bot}$ is not primitive in $H^{n}\left(\widetilde{\mathbb{P}^{n}/G},\Z\right)$.
We have:
$$H^{n}\left(\widetilde{\mathbb{P}^{n}/G},\Z\right)/\Image g'^{n}\oplus (\Image g'^{n})^{\bot}=\Z/p\Z.$$
\item[---]
Since $H^{n}\left(\widetilde{\mathbb{P}^{n}/G},\Z\right)$ is unimodular, we are able to deduce (v) from the lattice results of Section \ref{remindersL}.
\end{itemize}

\begin{proof}[Proof of (i), (ii) and (iii)]
Let $1\leq k\leq n-1$.
 The statement (i) is a direct consequence of Proposition \ref{torsion} looking at the following exact sequence:
 \begin{center}
\begin{tikzcd}[row sep=normal, column sep=normal]
H^{2k-1}\left(\widetilde{\mathbb{P}^{n}/G}^*,\Z\right)\arrow[r]
&  H^{2k}\left(\widetilde{\mathbb{P}^{n}/G},\widetilde{\mathbb{P}^{n}/G}^*,\Z\right) \ar[r,"g'^{2k}"] \arrow[d, phantom, ""{coordinate, name=Z}]
& H^{2k}\left(\widetilde{\mathbb{P}^{n}/G},\Z\right) \arrow[dl,
"i^*",
rounded corners,
to path={ -- ([xshift=2ex]\tikztostart.east)
|- (Z) [near start]\tikztonodes
-| ([xshift=-2ex]\tikztotarget.west)
-- (\tikztotarget)}] &\\
&H^{2k}\left(\widetilde{\mathbb{P}^{n}/G}^*,\Z\right) \ar[r] & H^{2k+1}\left(\widetilde{\mathbb{P}^{n}/G},\widetilde{\mathbb{P}^{n}/G}^*,\Z\right)\arrow[r]&0.
\end{tikzcd}
\end{center}
Moreover, from Proposition \ref{genetoric} and \ref{torsion}, we know that the map $i^*$ is surjective. It follows: 
 \begin{equation}
 H^{2k+1}\left(\widetilde{\mathbb{P}^{n}/G},\widetilde{\mathbb{P}^{n}/G}^*,\Z\right)=0.
 \label{H2k+1}
 \end{equation}
 The previous exact sequence also provides that $H^{2k}\left(\widetilde{\mathbb{P}^{n}/G},\widetilde{\mathbb{P}^{n}/G}^*,\Z\right)$ is torsion-free since $H^{2k}\left(\widetilde{\mathbb{P}^{n}/G},\Z\right)$ is torsion-free by Remark \ref{CohoDani}. 
 In addition, we have $H^{2n}\left(\widetilde{\mathbb{P}^{n}/G},\widetilde{\mathbb{P}^{n}/G}^*,\Z\right)=\Z$ by (\ref{2nP*}).
\end{proof}

\begin{proof}[Proof of (iv) and (v)]
Using the exact sequence of relative cohomology of the couple $\left(\widetilde{\mathbb{P}^{n}/G},\widetilde{\mathbb{C}^{n}/G}\right)$, we prove exactly as previously that:
\begin{equation}
 H^{2k+1}\left(\widetilde{\mathbb{P}^{n}/G},\widetilde{\mathbb{C}^{n}/G},\Z\right)=0.
 \label{H2k+1CG}
 \end{equation}
Moreover: 
\begin{equation}
H^{2k+1}\left(\widetilde{\mathbb{P}^{n}/G},(\C^n/G)^*,\Z\right)=H^{2k+1}\left(\widetilde{\mathbb{P}^{n}/G},\widetilde{\mathbb{P}^{n}/G}^*,\Z\right)\oplus H^{2k+1}\left(\widetilde{\mathbb{P}^{n}/G},\widetilde{\mathbb{C}^{n}/G},\Z\right)=0,
 \label{H2k+1CGB}
\end{equation}
and
\begin{equation}
H^{2k}\left(\widetilde{\mathbb{P}^{n}/G},(\C^n/G)^*,\Z\right)=H^{2k}\left(\widetilde{\mathbb{P}^{n}/G},\widetilde{\mathbb{P}^{n}/G}^*,\Z\right)\oplus H^{2k}\left(\widetilde{\mathbb{P}^{n}/G},\widetilde{\mathbb{C}^{n}/G},\Z\right).
 \label{H2kCGB}
\end{equation}
 Now, (iv) follows from the following commutative diagram:
\begin{center}
\begin{tikzcd}
0\ar[dr] & & 0\ar[d] & & & \\
& H^{2k}\left(\widetilde{\mathbb{P}^{n}/G},(\mathbb{C}^{n}/G)^*,\Z\right)\ar[dr,"g^{2k}"] & H^{2k}\left(\widetilde{\mathbb{P}^{n}/G},\widetilde{\mathbb{C}^{n}/G},\Z\right)\ar[l,"\overline{p}^{2k}"']\ar[d,"\overline{g}^{2k}"] &\ar[l] 0 & & \\
0\ar[r]& H^{2k}\left(\widetilde{\mathbb{P}^{n}/G},\widetilde{\mathbb{P}^{n}/G}^*,\Z\right)\ar[u,"p'^{2k}"]\ar[r,"g'^{2k}"']& H^{2k}\left(\widetilde{\mathbb{P}^{n}/G},\Z\right)\ar[dr]\ar[r]\ar[d] &H^{2k}\left(\widetilde{\mathbb{P}^{n}/G}^*,\Z\right)\ar[r] &0 \\
& 0\ar[u]& H^{2k}\left(\widetilde{\mathbb{C}^{n}/G},\Z\right)\ar[d] &H^{2k}(\mathbb{C}^{n}/G)^*,\Z)\ar[dr] &\\
  & & 0 & & 0.  
\end{tikzcd}
\end{center}
The zeros in the diagram come from (\ref{H2k+1}), (\ref{H2k+1CG}), (\ref{H2k+1CGB}), Propositions \ref{equiU} and \ref{torsion}, and Remark \ref{CohoDani}.
 Moreover the diagram shows that the exact sequence
 $$\xymatrix{
&H^{2k}\left(\widetilde{\mathbb{P}^{n}/G},\widetilde{\mathbb{P}^{n}/G}^*,\Z\right)\eq[d] & & & \\ 0&H^{2k}\left(\widetilde{\mathbb{C}^{n}/G},(\mathbb{C}^{n}/G)^*,\Z\right)\ar[l]&H^{2k}\left(\widetilde{\mathbb{P}^{n}/G},(\mathbb{C}^{n}/G)^*,\Z\right)\ar[l]&H^{2k}\left(\widetilde{\mathbb{P}^{n}/G},\widetilde{\mathbb{C}^{n}/G},\Z\right)\ar[l]&0\ar[l]
 }$$ splits. Hence the
 maps $\overline{p}^{2k}$ and $p'^{2k}$ are the natural embeddings: 
\begin{align*}
p'^{2k}:H^{2k}\left(\widetilde{\mathbb{P}^{n}/G},\widetilde{\mathbb{P}^{n}/G}^*,\Z\right)&\longrightarrow H^{2k}\left(\widetilde{\mathbb{P}^{n}/G},\widetilde{\mathbb{P}^{n}/G}^*,\Z\right)\oplus H^{2k}\left(\widetilde{\mathbb{P}^{n}/G},\widetilde{\mathbb{C}^{n}/G},\Z\right)\\
\overline{p}^{2k}:H^{2k}\left(\widetilde{\mathbb{P}^{n}/G},\widetilde{\mathbb{C}^{n}/G},\Z\right)&\longrightarrow H^{2k}\left(\widetilde{\mathbb{P}^{n}/G},\widetilde{\mathbb{P}^{n}/G}^*,\Z\right)\oplus H^{2k}\left(\widetilde{\mathbb{P}^{n}/G},\widetilde{\mathbb{C}^{n}/G},\Z\right).
\end{align*}
By commutativity of the diagram, it follows:
\begin{equation}
\Image g^{2k}=\Image g'^{2k}\oplus \Image \overline{g}^{2k}.
\label{imim}
\end{equation}
 By Proposition \ref{torsion}, $H^{2k}\left(\widetilde{\mathbb{P}^{n}/G}^*,\Z\right)$ and $H^{2k}\left(\widetilde{\mathbb{C}^{n}/G},\Z\right)$ are torsion-free. It follows from the diagram that $\Image g'^{2k}$ and $\Image \overline{g}^{2k}$ are primitive sub-groups in $H^{2k}\left(\widetilde{\mathbb{P}^{n}/G},\Z\right)$. However, $H^{2k}((\mathbb{C}^{n}/G)^*,\Z)=\Z/p\Z$ according to Lemma~\ref{equiU}. This means that:
 $$\frac{H^{2k}\left(\widetilde{\mathbb{P}^{n}/G},\Z\right)}{\Image g^{2k}}=\Z/p\Z.$$
 So by (\ref{imim}):
  $$\frac{H^{2k}\left(\widetilde{\mathbb{P}^{n}/G},\Z\right)}{\Image g'^{2k}\oplus \Image \overline{g}^{2k}}=\Z/p\Z.$$
 We are considering $1\leq k\leq n-1$, so in particular, the same result is true for $\Image g'^{2(n-k)}$ and $\Image \overline{g}^{2(n-k)}$ in $H^{2(n-k)}\left(\widetilde{\mathbb{P}^{n}/G},\Z\right)$. This means:
 $$\frac{H^{2k}\left(\widetilde{\mathbb{P}^{n}/G},\Z\right)\oplus H^{2(n-k)}\left(\widetilde{\mathbb{P}^{n}/G},\Z\right)}{\left(\Image g'^{2k}\oplus \Image g'^{2(n-k)}\right)\oplus^\bot \left(\Image \overline{g}^{2k} \oplus\Image \overline{g}^{2(n-k)}\right)}=\left(\Z/p\Z\right)^2,$$
 with the orthogonality which is due to $\widetilde{\mathbb{P}^{n}/G}=\widetilde{\mathbb{P}^{n}/G}^*\cup \widetilde{\mathbb{C}^{n}/G}$. Indeed the cup-product in relative cohomology is a map:
 $$H^{2k}\left(\widetilde{\mathbb{P}^{n}/G},\widetilde{\mathbb{P}^{n}/G}^*,\Z\right)\otimes H^{2(n-k)}\left(\widetilde{\mathbb{P}^{n}/G},\widetilde{\mathbb{C}^{n}/G},\Z\right)\longrightarrow H^{2n}\left(\widetilde{\mathbb{P}^{n}/G},\widetilde{\mathbb{P}^{n}/G}^*\cup\widetilde{\mathbb{C}^{n}/G},\Z\right)=0.$$
By Poincar\'e duality, $H^{2k}\left(\widetilde{\mathbb{P}^{n}/G},\Z\right)\oplus H^{2(n-k)}\left(\widetilde{\mathbb{P}^{n}/G},\Z\right)$ is unimodular, it follows from (\ref{BasicLatticeTheory}), (\ref{BasicLatticeTheory2}) and the primitivity of $\Image g'$ and $\Image \overline{g}$ that: $$\discr \left(\Image g'^{2k}\oplus \Image g'^{2(n-k)}\right)=\discr \left(\Image \overline{g}^{2k} \oplus\Image \overline{g}^{2(n-k)}\right)=p^2.$$
 
 The statement (v) is a particular case of statement (iv) when $n$ is even.
Assume $n$ is even, we also get that $\Image g'^{n}$ and $\Image\overline{g}^{n}$ are primitive in $H^{n}\left(\widetilde{\mathbb{P}^{n}/G},\Z\right)$ with $\Image g^{n}$ which admits a primitive element divisible by $p$. 
Hence:
$$\frac{H^{n}\left(\widetilde{\mathbb{P}^{n}/G},\Z\right)}{\Image g'^{n}\oplus^\bot \Image \overline{g}^{n} }=\Z/p\Z.$$
As before, the unimodularity of $H^{n}\left(\widetilde{\mathbb{P}^{n}/G},\Z\right)$ provides: 
$$\discr \Image g'^{n}=\discr \Image \overline{g}^{n} =p.$$
\end{proof}
 
\begin{rmk}\label{H1rmk}
 We can also mention that $H^0\left(\widetilde{\mathbb{P}^{n}/G},\widetilde{\mathbb{P}^{n}/G}^*,\Z\right)=H^1\left(\widetilde{\mathbb{P}^{n}/G},\widetilde{\mathbb{P}^{n}/G}^*,\Z\right)=0$ because the map  $i^*:H^0\left(\widetilde{\mathbb{P}^{n}/G},\Z\right)\rightarrow H^0\left(\widetilde{\mathbb{P}^{n}/G}^*,\Z\right)$ is an isomorphism.
\end{rmk}
 
\subsection{Application to the integral cohomology of the toric blow-up of isolated quotient singularities}\label{cortoriblowblow}
Now, we apply the previous result to understand better how a toric blow-up modifies the cohomology.
\begin{cor}\label{CorCohomology1bis}
Let $M$ be a topological space with an isolated complex quotient point $x\in M$. We assume that $x$ admits a local uniformizing system $(W_x,\mathcal{W},G,h)$ with $G$ of prime order.
Let $r:\widetilde{M}\rightarrow M$ be a toric blow-up of $M$ in $x$. 
We denote $U_x:=\widetilde{M}\smallsetminus r^{-1}(x)$ and $n:= \dim \mathcal{W}$. 
We consider the following exact sequence:
$$\xymatrix{ H^{2k}\left(\widetilde{M},U_x,\Z\right)\ar[r]^{g_x^{2k}} & H^{2k}\left(\widetilde{M},\Z\right)\ar[r] & H^{2k}(U_x,\Z)\ar[r]& H^{2k+1}\left(\widetilde{M},U_x,\Z\right),}$$
with $1\leq k\leq n-1$.
Then:
\begin{itemize}
\item[(i)]
$g_x^{2k}$ is injective,
\item[(ii)]
$H^{2k+1}\left(\widetilde{M},U_x,\Z\right)=0$, 
\item[(iii)]
$H^{2k}\left(\widetilde{M},U_x,\Z\right)$ is torsion-free and $H^{2n}\left(\widetilde{M},U_x,\Z\right)=\Z$.
\end{itemize} 
\end{cor}
\begin{proof}
The statements (ii) and (iii) are immediate consequence of the excision theorem:
$$H^{k}\left(\widetilde{M},U_x,\Z\right)=H^{k}\left(\widetilde{\mathbb{P}^{n}/G},\widetilde{\mathbb{P}^{n}/G}^*,\Z\right)$$ and Theorem \ref{MainCnG} (ii), (iii).
The statement (i) is a consequence of the following commutative diagram and Proposition \ref{equiU}.
 \begin{equation}
 \xymatrix@R10pt@C15pt{
 0\ar@{=}[r]&H^{2k-1}((\mathbb{C}^{n}/G)^*,\Z)\ar[r]&H^{2k}\left(\widetilde{\mathbb{C}^{n}/G},(\mathbb{C}^{n}/G)^*,\Z\right)\ar[r] & H^{2k}\left(\widetilde{\mathbb{C}^{n}/G},\Z\right)\\
 &H^{2k-1}(U_x,\Z)\ar[u]\ar[r]&\ar@{=}[u]H^{2k}\left(\widetilde{M},U_x,\Z\right)\ar[r]^{g_x} & H^{2k}\left(\widetilde{M},\Z\right)\ar[u]\ar[r]& H^{2k}(U_x,\Z).
}
\label{adokdiagram}
\end{equation}
\end{proof}
The previous corollary allows to describe the integral cohomology of a toric blow-up in several isolated points.
\begin{cor}\label{CorCohomology2}
Let $n\geq2$ and $p$ be a prime number. Let $M$ be a topological space and $F\subset M$ be a finite set of isolated complex quotient points. For all $x\in F$ we assume that there exists a local uniformizing system $(W_x,\mathcal{W},G,h)$ with $\dim \mathcal{W}=n$ and $\# G=p$.
Let $r:\widetilde{M}\rightarrow M$ be a toric blow-up of $M$ in $F$. 
We state $U:=M\smallsetminus F$
and $j:U\hookrightarrow M$. 
Then for all $1\leq k\leq n-1$, there exists an integer $0\leq d_p^k\leq \#F$ such that we have the following exact sequences:
\begin{itemize}
\item[(i)]
$\xymatrix@C30pt{
 0\ar[r] & H^{2k}(M,\Z)\ar[r]^{j^*} & H^{2k}(U,\Z)\ar[r]& (\Z/p\Z)^{d_p^k}\ar[r] & 0,}$
\item[(ii)]
$\xymatrix@C30pt{
 0\ar[r] & H^{2k}(M,\Z)\oplus H^{2k}\left(\widetilde{M},U,\Z\right)\ar[r]^{\ \ \ \ \ \ \ \ \ \ r^*+g^{2k}} & H^{2k}\left(\widetilde{M},\Z\right)\ar[r]& (\Z/p\Z)^{d_p^k}\ar[r] & 0,}$
 
 with $H^{2k}\left(\widetilde{M},U,\Z\right)$ which is torsion-free.
\item[(iii)]
$\xymatrix@C30pt{
 0\ar[r] &(\Z/p\Z)^{\#F-d_p^k}\ar[r] &H^{2k+1}(M,\Z) \ar[r]& H^{2k+1}\left(\widetilde{M},\Z\right)\ar[r] & 0.}$
 \item[(iv)]
 Moreover,
 $r^*:H^1(M,\Z)\rightarrow H^1\left(\widetilde{M},\Z\right)$ is an isomorphism,
 \item[(v)]
 $r^{*}:H^{2n}(M,\Z)\rightarrow H^{2n}\left(\widetilde{M},\Z\right)$ is an isomorphism.
\end{itemize}
\end{cor}
\begin{proof}
For $1\leq k\leq n-1$, we consider the following exact sequence (the coefficient ring $\Z$ is omitted for clarity sake).
\begin{equation}\label{mainequa}
\begin{tikzcd}
&&&&(\Z/p\Z)^{\#F}\ar[d,equal]&\\
H^{2k-1}(U)\arrow[r]\ar[d,equal] & 0\ar[r]\ar[d]&H^{2k}(M)\ar[d,"r^*"]  \ar[r,"j_{2k}"]\ar[dd,phantom,very near end,""{coordinate, name=Y}] & H^{2k}(U)\ar[d,equal]\ar[r] & H^{2k+1}(M,U)\ar[d]\arrow[ddll,
rounded corners,
to path={ -- ([xshift=4ex]\tikztostart.east)
|- (Y) [near start]\tikztonodes
-| ([xshift=-2ex]\tikztotarget.west)
-- (\tikztotarget)}] & \\
H^{2k-1}(U)\arrow[r] & H^{2k}\left(\widetilde{M},U\right)\ar[r,"g^{2k}"] & H^{2k}\left(\widetilde{M}\right)\ar[r,"\widetilde{j}_{2k}"] & H^{2k}(U) \ar[r]\ar[d,phantom,very near start, ""{coordinate, name=X}] & 0\arrow[ddll,
rounded corners,
to path={ -- ([xshift=2ex]\tikztostart.east)
|- (X) [near start]\tikztonodes
-| ([xshift=-4ex]\tikztotarget.west)
-- (\tikztotarget)}]& \\
&& H^{2k+1}(M)\ar[r,"j_{2k+1}"']\ar[d] & H^{2k+1}(U)\ar[d,equal]\ar[r] &H^{2k+2}(M,U)\ar[d,"f^{2k+2}"]\ar[r,equal] &\Z^{\delta_{k,n-1}\#F}\ar[d,equal]\\
&& H^{2k+1}\left(\widetilde{M}\right)\ar[r,"\widetilde{j}_{2k+1}"'] & H^{2k+1}(U)\ar[r] & H^{2k+2}\left(\widetilde{M},U\right)\ar[r,equal] &\Z^{\delta_{k,n-1}\#F}
\end{tikzcd}
\end{equation}
where $\delta_{k,n-1}$ is the Kronecker delta. The relative cohomology groups in the sequence are given by Corollary~\ref{CorCohomology1bis} and Proposition \ref{exactU}. 
The following part of (\ref{mainequa}) provides (i):
$$
\xymatrix{0\ar[r]&H^{2k}(M)\ar[r]&H^{2k}(U)\ar[r]&(\Z/p\Z)^{\#F}.}$$
By commutativity of (\ref{mainequa}), the maps $r^*:H^{2k}(M)\rightarrow H^{2k}\left(\widetilde{M}\right)$ and $g^{2k}$ are injective.
Hence, we can extract from (\ref{mainequa}), the following exact sequence:
$$\xymatrix{0\ar[r]&H^{2k}\left(\widetilde{M},U\right)\ar[r]&H^{2k}\left(\widetilde{M}\right)\ar[r]&H^{2k}(U)\ar[r]&0.}$$
This means that the map $\widetilde{j}_{2k}$ induces the following isomorphism:
$$\widetilde{j}_{2k}:\frac{H^{2k}\left(\widetilde{M}\right)}{g^{2k}\left(H^{2k}\left(\widetilde{M},U\right)\right)}\longrightarrow H^{2k}(U).$$
By commutativity of (\ref{mainequa}), we obtain the following isomorphism:
$$\widetilde{j}_{2k}:\frac{H^{2k}\left(\widetilde{M}\right)}{g^{2k}\left(H^{2k}\left(\widetilde{M},U\right)\right)\oplus r^*\left(H^{2k}(M)\right)}\longrightarrow \frac{H^{2k}(U)}{j_{2k}\left(H^{2k}(M)\right)}.$$
Hence (ii) follows from (i). 

By Lemma \ref{adoc}, the map $f^{2k+2}$ is always an isomorphism. It follows from the commutativity of Diagram (\ref{mainequa}) that:
$$\Image j_{2k+1}\simeq \Image \widetilde{j}_{2k+1}\simeq H^{2k+1}\left(\widetilde{M}\right).$$
Therefore (iii) follows from (i) and the following part of (\ref{mainequa}): 
 $$\xymatrix{H^{2k}(U)\ar[r]&H^{2k+1}(M,U)\ar[r]&H^{2k+1}(M)\ar[r]^{j_{2k+1}}&H^{2k+1}(U).}$$

Moreover, since $g^2$ is injective by Corollary \ref{CorCohomology1bis}, by Proposition \ref{exactU} and Remark \ref{H1rmk}, we also obtain the following diagram: 
 $$\xymatrix@C15pt@R10pt{
 0\ar[r] & H^{1}(M)\ar[r]\ar[d] & H^{1}(U)\ar@{=}[d]\ar[r]&0\\
  0\ar[r] & H^{1}\left(\widetilde{M}\right)\ar[r]& H^{1}(U)\ar[r]&0
 }$$
which provides (iv). 

 Statement (v) can also be proved using the commutativity of (\ref{mainequa}) and the bijectivity of $f^{2n}$ looking at this part of the diagram:
 $$\xymatrix@C15pt@R10pt{
 H^{2n}(M,U)\ar[r]\ar[d]^{f^{2n}} & H^{2n}(M)\ar[r]\ar[d] & H^{2n}(U)\ar@{=}[d]\ar[r]&0\\
  H^{2n}\left(\widetilde{M},U\right)\ar[r] & H^{2n}\left(\widetilde{M}\right)\ar[r]& H^{2n}(U)\ar[r]&0.
 }$$

\end{proof}
The $d_p^k$ will be explicitly computed when $M$ is a quotient in the proof of Theorem \ref{main2}.
\begin{cor}\label{CorCohomology1ter}
Let $M$ be a topological space with an isolated complex quotient point $x\in M$. We assume that $x$ admits a local uniformizing system $(W_x,\mathcal{W},G,h)$ with $G$ of prime order $p$.
Let $r:\widetilde{M}\rightarrow M$ be a toric blow-up of $M$ in $x$. We assume that $\widetilde{M}$ is an $2n$-dimensional compact connected orientable 
manifold. 
We denote $U_x:=\widetilde{M}\smallsetminus r^{-1}(x)$ and $g_x:H^{*}\left(\widetilde{M},U_x,\Z\right)\rightarrow H^{*}\left(\widetilde{M},\Z\right)$.
Then for all $1\leq k\leq n-1$:
\begin{itemize}
\item[(i)]
 $\Image g_x^{2k}\oplus \Image g_x^{2n-2k}$ is a sublattice of $H^{2k}\left(\widetilde{M},\Z\right)\oplus H^{2n-2k}\left(\widetilde{M},\Z\right)$ of discriminant $p^2$.
\item[(ii)]
If $n$ is even, then $\Image g_x^{n}$ is a sublattice of $H^{n}\left(\widetilde{M},\Z\right)$ of discriminant $p$.
\end{itemize}
\end{cor}
\begin{proof}
This corollary is also a consequence of the excision theorem and Theorem \ref{MainCnG} (iii), (iv) and (v). 
By construction of the cup product, we have a commutative diagram (see for instance \cite[p.~209]{Hatcher}):
\begin{center}
\begin{tikzcd}[column sep=60pt]
H^{2k}\left(\widetilde{\mathbb{P}^{n}/G},\widetilde{\mathbb{P}^{n}/G}^*,\Z\right)\otimes H^{2l}\left(\widetilde{\mathbb{P}^{n}/G},\widetilde{\mathbb{P}^{n}/G}^*,\Z\right)\ar[d,"g'^{2k}\otimes g'^{2l}"]\ar[r,"\cup"] & H^{2(k+l)}\left(\widetilde{\mathbb{P}^{n}/G},\widetilde{\mathbb{P}^{n}/G}^*,\Z\right)\ar[d,"g'^{2(k+l)}"]\\
H^{2k}\left(\widetilde{\mathbb{P}^{n}/G},\Z\right)\otimes H^{2l}\left(\widetilde{\mathbb{P}^{n}/G},\Z\right)\ar[r,"\cup"]&H^{2(k+l)}\left(\widetilde{\mathbb{P}^{n}/G},\Z\right),
\end{tikzcd}
\end{center}
where "$\cup$" is the cup product.  
Since $H^{2n}\left(\widetilde{\mathbb{P}^{n}/G},\widetilde{\mathbb{P}^{n}/G}^*,\Z\right)=H^{2n}\left(\widetilde{M},U_x,\Z\right)=\Z$ by Theorem \ref{MainCnG} and Corollary \ref{CorCohomology1bis}, the cup product on $H^{2*}\left(\widetilde{\mathbb{P}^{n}/G},\widetilde{\mathbb{P}^{n}/G}^*,\Z\right)$ and on $H^{2*}\left(\widetilde{M},U_x,\Z\right)$ can be seen as a bilinear form. 

By (\ref{alpha}) and (\ref{delta}), $g'^{2n}:H^{2n}\left(\widetilde{\mathbb{P}^{n}/G},\widetilde{\mathbb{P}^{n}/G}^*,\Z\right)\rightarrow H^{2n}\left(\widetilde{\mathbb{P}^{n}/G},\Z\right)$ is an isomorphism. Therefore, the commutativity of the previous diagram shows that $g'^{2*}: H^{2*}\left(\widetilde{\mathbb{P}^{n}/G},\widetilde{\mathbb{P}^{n}/G}^*,\Z\right)\rightarrow H^{2*}\left(\widetilde{\mathbb{P}^{n}/G},\Z\right)$ is an isometry.  

We show that the same property also holds for $g_x^{2*}:H^{2*}\left(\widetilde{M},U_x,\Z\right)\rightarrow H^{2*}\left(\widetilde{M},\Z\right)$.
As before, we have a commutative diagram:
\begin{center}
\begin{tikzcd}[column sep=50pt]
H^{2k}\left(\widetilde{M},U_x,\Z\right)\otimes H^{2l}\left(\widetilde{M},U_x,\Z\right)\ar[d,"g^{2k}_x\otimes g^{2l}_x"]\ar[r,"\cup"] & H^{2(k+l)}\left(\widetilde{M},U_x,\Z\right)\ar[d,"g^{2(k+l)}_x"]\\
H^{2k}\left(\widetilde{M},\Z\right)\otimes H^{2l}\left(\widetilde{M},\Z\right)\ar[r,"\cup"] & H^{2(k+l)}\left(\widetilde{M},\Z\right).
\end{tikzcd}
\end{center}
Moreover, we have the following exact sequence:
\begin{center}
\begin{tikzcd}
H^{2n}\left(\widetilde{M},U_x,\Z\right)\ar[r,"g_x^{2n}"] & H^{2n}\left(\widetilde{M},\Z\right)\ar[r]& H^{2n}(U_x,\Z).
\end{tikzcd}
\end{center}
Since $\widetilde{M}$ is an $2n$-dimensional compact connected orientable
manifold, we have $H^{2n}(U_x,\Z)=0$.
Since 
$H^{2n}\left(\widetilde{M},U_x,\Z\right)=\Z$, we have that $g_x^{2n}$ is necessarily bijective.
Therefore, $g^{2*}_x: H^{2*}\left(\widetilde{M},U_x,\Z\right)\rightarrow H^{2*}\left(\widetilde{M},\Z\right)$ is an isometry.

Finally, the excision theorem provides an isometry  $H^{2*}\left(\widetilde{\mathbb{P}^{n}/G},\widetilde{\mathbb{P}^{n}/G}^*,\Z\right)\simeq H^{2*}\left(\widetilde{M},U_x,\Z\right)$. Hence Theorem \ref{MainCnG} (iv) and (v) conclude the proof.
\end{proof}

\section{Integral cohomology of quotients by cyclic groups with only isolated fixed points}\label{mainresult}
\subsection{Notation, hypothesis and definition}\label{coeffofresolution}
In this section $X$ is a compact complex manifold of dimension $n$ such that $H^*(X,\Z)$ is $p$-torsion-free 
and $G$ an automorphism group of prime order $p$ with only isolated fixed points. 
\begin{nota}
We set $M:=X/G$ and $\pi:X\rightarrow M$ the quotient map. We define the subsets $V:=X\smallsetminus \Fix G$, $U:=M\smallsetminus \Sing M$ and consider $j:U\hookrightarrow M$. We also denote by $\eta(G):=\#\Fix G$ the number of fixed points. We consider $r:\widetilde{M}\rightarrow M$ a toric blow-up of $M$ in $\Sing M$. 
 \end{nota}
  \begin{rmk}\label{hypopo}
 Actually, all the results of this section remain true if we choose for $X$ a $2n$-dimensional compact connected orientable $\C^{\infty}$-manifold, with $H^*(X,\Z)$ $p$-torsion-free, with $G$ an automorphism group of prime order $p$ which respects an orientation which has only isolated fixed points and such that all points of $\Sing M$ are isolated complex quotient points (see Definition \ref{quotientPoint}).
 \end{rmk}
We also use the notation of Section \ref{nota1}. The main idea is to obtain information on $H^*_f(M,\Z)$ from the Poincar\'e duality on $H^*\left(\widetilde{M},\Z\right)$ (see Section \ref{sketch}).
To do so, we need to understand the behaviour of $r^*(H^{2*}_f(M,\Z))$ inside $H^*\left(\widetilde{M},\Z\right)$. 
This information will be provided by the \emph{coefficient of resolution}.
We have seen from Corollary \ref{CorCohomology2} (ii) that $H^{2k}_f\left(\widetilde{M},\Z\right)/r^*(H^{2k}_f(M,\Z))$ can only have $p$-torsion, moreover its torsion is a $\F$-vector space of finite dimension.
\begin{defi}\label{coeffreso}
We define the \emph{coefficient of resolution}:
$$\beta_{2k}(X):=\dim_{\F}\tors\frac{H^{2k}_f\left(\widetilde{M},\Z\right)}{r^*(H^{2k}_f(M,\Z))}.$$
\end{defi}
A priori, the coefficient of resolution depends of $(X,G)$ and the choice of the toric blow-up $\widetilde{M}$. It corresponds to the number of primitive linearly independent elements in $H^k_f(M,\Z)$ which become divisible by $p$ inside $H^k_f\left(\widetilde{M},\Z\right)$.

 \subsection{The cohomology in odd degrees}\label{oddmainsc}
 \begin{thm}\label{oddmain}
Let $X$ be a compact complex manifold of dimension $n$ endowed with the action of an automorphism group of prime order $p$ with only isolated fixed points. We assume that $H^*(X,\Z)$ is $p$-torsion-free. Then:
$$\alpha_{2k+1}(X)+\alpha_{2n-2k-1}(X)=\ell_+^{2k+1}(X),\ \ \forall\ 0\leq k\leq n-1.$$
\end{thm}
\begin{proof}
We have by (\ref{MainExactSequence}):
$$\frac{H^{2k+1}_f(M,\Z)\oplus H^{2n-2k-1}_f(M,\Z)}{\pi_*\left(H^{2k+1}(X,\Z)\oplus H^{2n-2k-1}(X,\Z)\right)_f}=(\Z/p\Z)^{\alpha_{2k+1}(X)+\alpha_{2n-2k-1}(X)}.$$
By Corollary \ref{CorCohomology2} (iii), we have an isomorphism:
\begin{equation}
r^*:H^{2k+1}_f(M,\Z)\simeq H^{2k+1}_f\left(\widetilde{M},\Z\right).
\label{truc}
\end{equation}
Hence:
$$\frac{H^{2k+1}_f\left(\widetilde{M},\Z\right)\oplus H^{2n-2k-1}_f\left(\widetilde{M},\Z\right)}{r^*\pi_*\left(H^{2k+1}(X,\Z)\oplus H^{2n-2k-1}(X,\Z)\right)_f}=(\Z/p\Z)^{\alpha_{2k+1}(X)+\alpha_{2n-2k-1}(X)}.$$
By Corollary \ref{CorCohomology2} (v), $r^*\pi_*\left(H^{2k+1}(X,\Z)\oplus H^{2n-2k-1}(X,\Z)\right)_f $ and $\pi_*(H^{2k+1}(X,\Z)\oplus H^{2n-2k-1}(X,\Z))_f$
are isometric lattices. Moreover by Proposition \ref{vraisans19}
and \cite[Proposition 3.9]{Lol3}, we have:
$$\log_p\discr \pi_*\left(H^{2k+1}(X,\Z)\oplus H^{2n-2k-1}(X,\Z)\right)_f=\ell_+^{2k+1}(X)+\ell_+^{2n-2k-1}(X)=2\ell_+^{2k+1}(X).$$
Since $H^{2k+1}_f\left(\widetilde{M},\Z\right)\oplus H^{2n-2k-1}_f\left(\widetilde{M},\Z\right)$ is unimodular, by (\ref{BasicLatticeTheory}) we have:
$$\ell_+^{2k+1}(X)=\alpha_{2k+1}(X)+\alpha_{2n-2k-1}(X).$$
\end{proof}
\subsection{Expression for the dimensions of degeneration}\label{torM}
The $p$-torsion of $H^*(U,\Z)$ can be computed using the spectral sequence of equivariant cohomology since the action of $G$ on $V$ is free (see Section \ref{equivarsection}). For this purpose, we need the Boissi\`ere--Nieper-Wisskirchen--Sarti invariants of $V$ to be able to use Proposition \ref{equivarcoho}.
\begin{lemme}\label{ellV2}
We have:
\begin{itemize}
\item[(i)]
$\ell_*^k(V)=\ell_*^k(X)$ for all $0\leq k\leq 2n-2$, with $*=p, -$ or $+$;
\item[(ii)]
$H^{2n-1}(V,\Z)$ is $p$-torsion-free.
\end{itemize}
\end{lemme}
\begin{proof}
Statement (i) follows immediately from the fact that $H^k(V,\Z)=H^k(X,\Z)$ for all $k\leq 2n-2$.
Moreover, we have the following exact sequence:
$$
\xymatrix{0\ar[r]&H^{2n-1}(X,\Z) \ar[r] & H^{2n-1}(V,\Z)\ar[r]&H^{2n}(X,V,\Z)\ar[r]&H^{2n}(X,\Z)\ar[r]&H^{2n}(V,\Z).}
$$
Since $H^{2n-1}(X,\Z)$ is $p$-torsion-free and by Thom's isomorphism $H^{2n}(X,V,\Z)=\Z^{\eta(G)}$, $H^{2n-1}(V,\Z)$ is also $p$-torsion-free. We obtain (ii).
\end{proof}
From Lemma \ref{ellV2} (i) and (ii), we see that the spectral sequences of equivariant cohomology of $(X,G)$ and $(V,G)$ coincide sufficiently to obtain the following expressions for the dimensions of degeneration. We recall that the notation $t_p^{2k+1}(Y)$ for a topological space $Y$ is defined in Section \ref{nota1}.
\begin{lemme}\label{coefdege}
We have for all $1\leq k\leq n-1$:
\begin{itemize}
\item[(i)]
$u_{2k}(X)=\sum_{i=0}^{k-1}\ell_+^{2i}(X)+\sum_{i=0}^{k-1}\ell_-^{2i+1}(X)-t_p^{2k}(U);$
\item[(ii)]
$u_{2k+1}(X)=\sum_{i=0}^{k}\ell_-^{2i}(X)+\sum_{i=0}^{k-1}\ell_+^{2i+1}(X)-t_p^{2k+1}(U).$
\end{itemize}
\end{lemme}
\begin{proof}
We prove (i), the proof of (ii) is identical.
By Definition \ref{defidegene}, we have:
$$u_{2k}(X)=\dim_{\F} \left(\bigoplus_{d+q=2k,\, d>0} E_2^{d,q}\right)-\dim_{\F} \left(\bigoplus_{d+q=2k,\, d>0} E_{\infty}^{d,q}\right)+\left(t_p E_2^{0,2k}-t_p E_\infty^{0,2k}\right),$$
where $E_2^{d,q}$ is the second page of the spectral sequence of equivariant cohomology of $(X,G)$ with coefficient in $\Z$ (see Section \ref{equivarsection}).
First note that $t_p E_2^{0,2k}=0$ by Proposition \ref{equivarcoho} since $H^*(X,\Z)$ is $p$-torsion-free. Hence also $t_p E_\infty^{0,2k}=0$.
By Proposition \ref{equivarcoho},  
$$\dim_{\F} \left(\bigoplus_{d+q=2k,\, d>0} E_2^{d,q}\right)=\sum_{i=0}^{k-1}\ell_+^{2i}(X)+\sum_{i=0}^{k-1}\ell_-^{2i+1}(X).$$
Moreover, by convergence of the spectral sequence:
$$\dim_{\F}\left( \bigoplus_{d+q=2k,\, d>0} E_{\infty}^{d,q}\right)=t^{2k}_p(X_G).$$
Hence it only remains to show that:
$$t^{k}_p(X_G)=t^{k}_p(U),\ \forall\ k\leq 2n-1.$$
It can be seen as follows. We have the convergence of the following spectral sequence (see \cite[Chapter~VII, Section~7]{cohogroup}):
$$H^{d}(G,H^{q}(X,V,\Z))\Rightarrow H^{d+q}(X_G,V_G,\Z).$$
However by Thom's isomorphism, we have $H^{k}(X,V,\Z)=0$ for all $k\leq 2n-1$ and $H^{2n}(X,V,\Z)=\Z^{\eta(G)}$.
It follows that $H^{k}(X_G,V_G,\Z)=0$ for all $k\leq 2n-1$ and $H^{2n}(X_G,V_G,\Z)$ is torsion-free. 
Hence the long exact sequence of the relative cohomology of $(X_G,V_G)$ provides that 
$H^{k}(X_G,\Z)=H^k(V_G,\Z)$ for all $k\leq 2n-2$. 
Moreover, we have:
\begin{equation}
\xymatrix@C16pt{0\ar[r]&H^{2n-1}(X_G,\Z) \ar[r] & H^{2n-1}(V_G,\Z)\ar[r]&H^{2n}(X_G,V_G,\Z)\ar[r]&H^{2n}(X_G,\Z)\ar[r]&H^{2n}(V_G,\Z).}
\label{1fixed}
\end{equation}
Since $H^{2n}(X_G,V_G,\Z)$ is torsion-free $t^{2n-1}_p(X_G)=t^{2n-1}_p(V_G)$.
Finally the fact that $H^k(V_G,\Z)=H^k(U,\Z)$ for all $k$ concludes the proof.
\end{proof}

\subsection{General expression for the coefficients of surjectivity in even degrees}\label{coeffeven}
In this section we follow the method which has been sketched in Section \ref{sketch}. 
We recall the definition of the \emph{exceptional lattice} from \cite[Definition 5.1]{Lol3}. Let $1\leq k\leq n-1$. The $k^\mathrm{th}$ \emph{exceptional lattice} of $r$ is defined by: $$N_{k,r}:=r^*\left[\pi_*\left(H^k(X,\Z)\oplus H^{2n-k}(X,\Z)\right)\right]^\bot_f.$$
We remark that $N_{k,r}^\bot$ is the saturation of $r^*\left[\pi_*\left(H^k(X,\Z)\oplus H^{2n-k}(X,\Z)\right)\right]_f$.
By Poincar\'e duality, the lattice $H^{2k}_f\left(\widetilde{M},\Z\right)\oplus H^{2n-2k}_f\left(\widetilde{M},\Z\right)$ is unimodular.
Hence from (\ref{BasicLatticeTheory2}), we have:
\begin{equation}
\discr N_{2k,r}=\discr N_{2k,r}^\bot.
\label{discrN}
\end{equation}
We are going to compute $\log_p\discr N_{2k,r}$ and $\log_p\discr N_{2k,r}^\bot$ to obtain our general equation.

We recall the exact sequence obtained from Corollary \ref{CorCohomology1bis}:
\begin{equation}
\xymatrix{0\ar[r]&H^{2k}\left(\widetilde{M},U,\Z\right) \ar[r]^{g^{2k}} & H^{2k}\left(\widetilde{M},\Z\right)\ar[r]&H^{2k}(U,\Z)\ar[r]&0.}
\label{inter}
\end{equation}
We can prove using the projection formula for orbifolds (see \cite[Remark 2.8]{Lol}) that:
$$\Image g^{2k}\perp r^*(H^{2n-2k}(M,\Z)),$$
according to the pairing $H^{2k}\left(\widetilde{M},\Z\right)\times H^{2n-2k}\left(\widetilde{M},\Z\right)\rightarrow H^{2n}\left(\widetilde{M},\Z\right)$.
Then, the lattice $N_{2k,r}$ is the saturation of $\Image g^{2k}\oplus \Image g^{2n-2k}$.

By Corollary \ref{CorCohomology1ter} (i), we have $$\log_p\discr \left(\Image g^{2k}\oplus \Image g^{2n-2k}\right)=2 \eta(G).$$
Then (\ref{inter}) and (\ref{BasicLatticeTheory}) provide:
\begin{equation}
\log_p\discr N_{2k,r}=2\left(\eta(G)-t^{2k}_p(U)-t^{2n-2k}_p(U)+t_p^{2k}\left(\widetilde{M}\right)+t_p^{2n-2k}\left(\widetilde{M}\right)\right).
\label{Nr}
\end{equation}
By Proposition \ref{vraisans19} and \cite[Proposition 3.9]{Lol3}, we have:
\begin{equation}
\log_p\discr \pi_*\left(H^{2k}(X,\Z)\oplus H^{2n-2k}(X,\Z)\right)_f=\ell_+^{2k}(X)+\ell_+^{2n-2k}(X)=2\ell_+^{2k}(X).
\label{oddtor0}
\end{equation}
Then by (\ref{MainExactSequence}), Definition \ref{coeffreso} and (\ref{BasicLatticeTheory}), we obtain:
\begin{equation}
\log_p\discr N_{2k,r}^\bot=2\left(\ell_+^{2k}(X)-\alpha_{2k}(X)-\alpha_{2n-2k}(X)-\beta_{2k}(X)-\beta_{2n-2k}(X)\right).
\label{Nrbot}
\end{equation}
Then (\ref{discrN}) and (\ref{Nr}) give the following relation:
\begin{align*}
\ell_+^{2k}(X)-\left(\alpha_{2k}(X)+\alpha_{2n-2k}(X)\right)-&\left(\beta_{2k}(X)+\beta_{2n-2k}(X)\right)=\\
&\eta(G)-\left(t_p^{2k}(U)+t_p^{2n-2k}(U)\right)+\left(t_p^{2k}\left(\widetilde{M}\right)+t_p^{2n-2k}\left(\widetilde{M}\right)\right).
\end{align*}
We rewrite this equation adding $\ell_+^{2*}(X)+\ell_-^{2*+1}(X)-\ell_+^{2k}(X)$ on each sides of the equality:
\begin{align}
\ell_+^{2*}(X)+&\ell_-^{2*+1}(X)-\left(\alpha_{2k}(X)+\alpha_{2n-2k}(X)\right)-\left(\beta_{2k}(X)+\beta_{2n-2k}(X)\right)=\nonumber \\
&\eta(G)+\left(\ell_+^{2*}(X)+\ell_-^{2*+1}(X)-\ell_+^{2k}(X)-t_p^{2k}(U)-t_p^{2n-2k}(U)\right)+\left(t_p^{2k}\left(\widetilde{M}\right)+t_p^{2n-2k}\left(\widetilde{M}\right)\right).
\label{poincareequalmain}
\end{align}
However, from Lemma \ref{coefdege} and \cite[Proposition 3.9]{Lol3}, we have:
$$\ell_+^{2*}(X)+\ell_-^{2*+1}(X)-\ell_+^{2k}(X)-t_p^{2k}(U)-t_p^{2n-2k}(U)=u_{2k}(X)+u_{2n-2k}(X).$$
Hence from (\ref{poincareequalmain}), we have: 
\begin{align*}
\ell_+^{2*}(X)+\ell_-^{2*+1}(X)-\eta(G)= &\left(u_{2k}(X)+u_{2n-2k}(X)\right)+\left(t_p^{2k}\left(\widetilde{M}\right)+t_p^{2n-2k}\left(\widetilde{M}\right)\right)+\\
&\left(\alpha_{2k}(X)+\alpha_{2n-2k}(X)\right)+\left(\beta_{2k}(X)+\beta_{2n-2k}(X)\right).
\end{align*}
So by Proposition \ref{mainlefsch}, we obtain the following lemma.
\begin{lemme}\label{poincareequalmain2}
We have for all $1\leq k\leq n-1$:
\begin{align*}
\ell_+^{2*+1}(X)+\ell_-^{2*}(X)&=\left(u_{2k}(X)+u_{2n-2k}(X)\right)+\left(t_p^{2k}\left(\widetilde{M}\right)+t_p^{2n-2k}\left(\widetilde{M}\right)\right)+\\
&\left(\alpha_{2k}(X)+\alpha_{2n-2k}(X)\right)+\left(\beta_{2k}(X)+\beta_{2n-2k}(X)\right).
\end{align*}
\end{lemme}
In particular, we can state the following proposition.
\begin{prop}\label{ineforte2}
For all $1\leq k\leq n-1$, we have:
$$\alpha_{2k}(X)+\alpha_{2n-2k}(X)\leq \ell_+^{2*+1}(X)+\ell_-^{2*}(X).$$
\end{prop}
\subsection{Degeneration of the spectral sequence of equivariant cohomology}\label{degespec}
\begin{thm}\label{degenemain}
Let $X$ be a compact complex manifold of dimension $n$ and $G$ an automorphism group of prime order $p$ with only isolated fixed points. We assume that $H^*(X,\Z)$ is $p$-torsion-free and $\ell_p^1(X)=0$.  
Then, the following statements are equivalent:
\begin{itemize}
\item[(1)]
the spectral sequence of equivariant cohomology of $(X,G)$ with coefficients in $\F$ degenerates at the second page,
\item[(2)]
$\eta(G)=\ell_+^{2*}(X)+\ell_-^{2*+1}(X)$,
\item[(3)]
$\ell_+^{2*+1}(X)=\ell_-^{2*}(X)=0$,
\item[(4)]
the spectral sequence of equivariant cohomology of $(X,G)$ with coefficients in $\Z$ degenerates at the second page.
\end{itemize}
\end{thm}
\begin{proof}
By Corollary \ref{equiv1}, (2) $\Leftrightarrow$ (3), by Corollary \ref{degenerate1}, (1) $\Rightarrow$ (3) and by Proposition \ref{degeequ}, (1) $\Leftrightarrow$ (4). Hence, we only have to prove that (3) $\Rightarrow$ (4).

By Lemma \ref{poincareequalmain2}, $u_{2k}(X)=0$ for all $0\leq k\leq n-1$.
Since $\ell_+^{2*+1}(X)=0$ and $\ell_-^{2*}(X)=0$, $u_{2k+1}(X)=0$ for all $k\in\N$ by Lemma \ref{coefdege} (ii).
Hence by Lemma \ref{degelemmater}, $u_{2k}(X)=0$ for all $k\geq n+1$.
Because of Lemma \ref{degelemmabis}, it only remains to show that $u_{2n}=0$. The problem is similar to the one encounter in the proof of Lemma \ref{degelemmater}.
The only differentials that can still be non trivial are $d_t:H^0(G,H^{2n-1}(X,\Z))\rightarrow H^{t}(G,H^{2n-t}(X,\Z))$ at the page $t\geq 2$.
However, by \cite[Proposition 3.9]{Lol3}, the additional hypothesis $\ell_p^1(X)=0$ implies that $\ell_p^{2n-1}(X)=0$. Then applying Proposition \ref{equivarcoho} (i) and Theorem \ref{gene19} (v), we obtain that $H^0(G,H^{2n-1}(X,\Z))=0$.
\end{proof}
\begin{rmk}
I conjecture that the condition $\ell_p^1(X)=0$ can be removed.
\end{rmk}
\begin{rmk}\label{fixedpoints}
We can notice that statement (2) is necessarily wrong if $G$ does not have at least 2 fixed points. Indeed $\ell_+^0(X)=\ell_+^{2n}(X)=1$. Therefore, when $\eta(G)<2$, the spectral sequence of equivariant cohomology of $(X,G)$ with coefficients in $\Z$ or $\F$ cannot degenerate at the second page. Actually, it can be seen with a direct computation. We use the notation from (\ref{1fixed}). If the spectral sequence of equivariant cohomology of $(X,G)$ with coefficients in $\Z$ degenerates at the second page, we have $H^{2n}(X_G,\Z)=\Z\oplus(\Z/p\Z)^{\ell_-^{2*+1}(X)+\ell_+^{2*}(X)-1}$. Note that $\ell_-^{2*+1}(X)+\ell_+^{2*}(X)-1\geq1$. Moreover $H^{2n}(X_G,V_G,\Z)=\Z^{\eta(G)}$. 
Since $H^{2n}(V_G,\Z)=H^{2n}(U,\Z)=0$. The exact sequence (\ref{1fixed}) provides that:
$$\ell_-^{2*+1}(X)+\ell_+^{2*}(X)\leq \eta(G).$$
Hence: 
$$2\leq \eta(G).$$
The proof is similar if we consider the spectral sequence of equivariant cohomology with coefficients in $\F$.
\end{rmk}
\subsection{Main theorem}\label{MainThSec}
\begin{thm}\label{main2}
Let $X$ be a compact complex manifold of dimension $n$ such that $H^*(X,\Z)$ is $p$-torsion-free. Let $G$ be an automorphism group of prime order $p$ on $X$ with only isolated fixed points.  

Assume, in addition, one of the following statements:
\begin{itemize}
\item[(1)]
the spectral sequence of equivariant cohomology of $(X,G)$ with coefficients in $\F$ degenerates at the second page, or
\item[(2)]
$\eta(G)=\ell_+^{2*}(X)+\ell_-^{2*+1}(X)$ or,
\item[(3)]
$\ell_+^{2*+1}(X)=\ell_-^{2*}(X)=0$ or,
\item[(4)]
the spectral sequence of equivariant cohomology of $(X,G)$ with coefficients in $\Z$ degenerates at the second page.
\end{itemize}
Then:
\begin{itemize}
\item[(i)]
$\alpha_k(X)=0$ for all $1\leq k \leq 2n$,
\item[(ii)]
$H^{2k}(M,\Z)$ is $p$-torsion-free for all $0\leq k \leq n$,
\item[(iii)]
$t_p^{2k+1}(M)+t_p^{2n-2k+1}(M)=\eta(G)-\ell_+^{2k}(X)$, for all $1\leq k \leq n-1$.
 \end{itemize}
\end{thm}
\subsubsection*{Outline of the proof}
We assume hypothesis (3), because of Theorem \ref{degenemain}, Proposition \ref{degeequ}, Corollaries~\ref{equiv1} and \ref{degenerate1}, it is enough to prove our result with this assumption (in particular, hypothesis (2) is also verified in this case). 

The statements (i) and (ii) will be easily obtained from Theorem \ref{oddmain} and Lemma \ref{poincareequalmain2}. The statement (iii) will be deduced from Corollary \ref{CorCohomology2} via the integers $d_p^k$. Corollary \ref{CorCohomology2} (iii) allows to relate $t_p^{2k+1}(M)$ to $d_p^k$. Moreover, the integer $d_p^k$ can computed using Corollary \ref{CorCohomology2} (iii) and our lattice knowledge on $H^*\left(\widetilde{M},\Z\right)$.

\begin{proof}[Proof of (i) and (ii)]
Because of hypothesis (3) and Theorem \ref{oddmain}, we have:
\begin{equation}
 \alpha_{2k+1}(X)=0,
 \label{oddalpha}
 \end{equation}
 for all $0\leq k \leq n-1$.

Moreover, Lemma \ref{poincareequalmain2} and hypothesis (3) provide:
\begin{equation}
u_{2k}(X)=t_{2k}\left(\widetilde{M}\right)=\alpha_{2k}(X)=\beta_{2k}(X)=0,
\label{lemmaequa}
\end{equation}
for all $1\leq k \leq n-1$. 
This provides (i) with (\ref{oddalpha}) and Remark \ref{alpha0}.

By Corollary \ref{CorCohomology2} (ii), we know that $t_{2k}(M)\leq t_{2k}\left(\widetilde{M}\right)$.
Hence we obtain (ii) from (\ref{lemmaequa}), ($H^{2n}(M,\Z)=\Z$ by Corollary \ref{CorCohomology2} (v)).
\end{proof}

\begin{proof}[Proof of (iii)]
By hypothesis (3) and Lemma \ref{coefdege} (ii), we know that $t_{p}^{2k+1}(U)=0$.
Moreover by (\ref{mainequa}) and Corollary \ref{CorCohomology2} (i), we have:
$$\xymatrix{
0\ar[r]&(\Z/p\Z)^{\eta(G)-d_p^k}\ar[r] & H^{2k+1}(M,\Z)\ar[r]&H^{2k+1}(U,\Z),}$$
where $0\leq d_p^k\leq \eta(G)$.
Hence:
\begin{equation}
t_p^{2k+1}(M)=\eta(G)-d_p^k. 
\label{dd}
\end{equation}
 However, $d_p^k+d_p^{n-k}$ can be computed.
Indeed by Corollary \ref{CorCohomology2} (ii), we have:
$$
\xymatrix@C30pt{
 0\ar[r] & H^{2k}(M,\Z)\oplus H^{2k}\left(\widetilde{M},U,\Z\right)\ar[r]^{\ \ \ \ \ \ \ \ \ \ r^*+g^{2k}} & H^{2k}\left(\widetilde{M},\Z\right)\ar[r]& (\Z/p\Z)^{d_p^k}\ar[r] & 0.}
$$
We obtain:
 \begin{equation}
 \frac{H^{2k}\left(\widetilde{M},\Z\right)\oplus H^{2n-2k}\left(\widetilde{M},\Z\right)}{r^*(H^{2k}(M,\Z))\oplus \Image g^{2k}\oplus r^*(H^{2n-2k}(M,\Z))\oplus \Image g^{2n-2k}}=(\Z/p\Z)^{d_p^k+d_p^{n-k}}.
 \label{res1}
 \end{equation}
 Moreover, we know from (\ref{Nrbot}), (\ref{discrN}), (\ref{lemmaequa}) and (\ref{BasicLatticeTheory}) that:
 \begin{equation}
 \frac{H^{2k}\left(\widetilde{M},\Z\right)\oplus H^{2n-2k}\left(\widetilde{M},\Z\right)}{N_{2k,r}\oplus N_{2k,r}^\bot}=(\Z/p\Z)^{2\ell_+^{2k}(X)}.
 \label{res2}
 \end{equation}
 By (\ref{Nrbot}), (\ref{oddtor0}) and (\ref{lemmaequa}), we have: 
 \begin{equation}
 N_{2k,r}^\bot=r^*(H^{2k}(M,\Z))\oplus r^*(H^{2n-2k}(M,\Z)).
 \label{res3}
 \end{equation}
 However, we have by (\ref{inter}) and (\ref{lemmaequa}):
 \begin{equation}
 \frac{N_{2k,r}}{\Image g^{2k}\oplus\Image g^{2n-2k}}=(\Z/p\Z)^{t_p^{2k}(U)+t_p^{2n-2k}(U)}.
 \label{res4}
 \end{equation}
 It follows from (\ref{res1}), (\ref{res2}), (\ref{res3}), (\ref{res4}), that:
 $$d_p^k+d_p^{n-k}=t_p^{2k}(U)+t_p^{2n-2k}(U)+2\ell_+^{2k}(X).$$
 Hence by (\ref{dd}):
 $$t_p^{2k+1}(M)+t_p^{2n-2k+1}(M)=2\eta(G)-(t_p^{2k}(U)+t_p^{2n-2k}(U)+2\ell_+^{2k}(X)).$$
 Applying Lemma \ref{coefdege} (i) with (\ref{lemmaequa}) and hypothesis (2), we obtain:
 $$t_p^{2k+1}(M)+t_p^{2n-2k+1}(M)=\eta(G)-\ell_+^{2k}(X).$$
\end{proof}

\begin{rmk}\label{alpha0}
Let $X$ be a connected topological space endowed with an automorphism group of prime order. By definition of $\pi_*$ in \cite{Transfers}, we always have $\alpha_0(X)=1$. 
Hence, if in addition $X$ is a compact connected orientable manifold of dimension $n$, we always have $\alpha_n(X)=0$ by Proposition \ref{vraisans19ter}.
\end{rmk}
\begin{rmk}
We could also have provided an expression for $t_p^{2k+1}(M)+t_p^{2n-2k+1}(M)$ in full generality without assumption on the spectral sequence using the coefficients of surjectivity, the coefficients of resolution and the dimensions of degeneration. However, it was very technical and the author has preferred not to bother the reader with such a computation.
\end{rmk}
It would be interesting to know how the torsion is shared between $H^{2k+1}(M,\Z)$ and $H^{2n-2k+1}(M,\Z)$.
\begin{conj}
We have $t_p^{2k+1}(M)=\sum_{j=0}^{k-1}\ell_+^{2j}(X)+\sum_{j=0}^{k-1}\ell_-^{2j+1}(X)$.
\end{conj}
\section{Examples of applications}\label{ExSe}
\subsection{Quotients of surfaces by automorphism groups of prime order}\label{K3Sec}
\begin{prop}\label{CorK3}
Let $X$ be a surface endowed with the action of an automorphism group $G$ of prime order $p$. 
We assume that $H^*(X,\F)$ is concentrated in even degrees and that $\Fix G$ is finite and non-empty. We denote $\eta(G):=\#\Fix G$.
Then:
\begin{itemize}
\item[(i)]
$\alpha_2(X)=0$, 
\item[(ii)]
$H^2(X/G,\Z)$ and $H^4(X/G,\Z)$  are torsion-free and $H^3(X/G,\Z)=\Z/p\Z$.
\item[(iii)]
$\discr H^2(X/G,\Z)=p^{\eta(G)-2}$.
\end{itemize}
\end{prop}
\begin{proof}
By \cite[Proposition 4.5]{SmithTh}, we know that the spectral sequence of equivariant cohomology of $(X,G)$ degenerates at the second page.
Then, Theorem \ref{main2} provides our result with $H^3(X/G,\Z)=(\Z/p\Z)^{\frac{\eta(G)-\ell_+^2(X)}{2}}$. However, by Theorem \ref{degenemain}, we have: 
\begin{equation}
\eta(G)=\ell_+^2(X)+2\ \text{and}\ \ell_-^2(X)=0.
\label{etasurface}
\end{equation}
Hence $H^3(X/G,\Z)=\Z/p\Z$.
Moreover, we obtain (iii) from (i), (\ref{etasurface}) and Proposition \ref{vraisans19}.
\end{proof}
\begin{rmk}\label{b2surface}
Note that by the generalized Hurwitz formula (Remark \ref{geneHurwitz}), we have
\[b_2(X/G)=\eta(G)-2+\frac{b_2(X)+2-\eta(G)}{p}.\]
\end{rmk}
In particular the previous proposition can be applied to a K3 surface. Let $L_{17}$ be the lattice $$L_{17}=\begin{pmatrix} -2 & 1 & 0 & 1 \\ 1 & -2 & 0 & 0 \\ 0 & 0 & -2 & 1 \\ 1 & 0 & 1 & -4\end{pmatrix}.$$
\begin{cor}
Let $X$ be a K3 surface endowed with the action of an automorphism group $G$ of prime order $p$. 
Assume that $\Fix G$ is finite and non-empty. Then for each $p$, the lattice $H^2(X/G,\Z)$ is given by:
\begin{center}
\begin{tabular}{|c|c|c|}
\hline
$p$ & $(H^{2}(X/G,\Z),\cdot)$& $\#\Sing X/G$\\
\hline
 2 & $E_{8}(-1)\oplus U(2)^3$&8\\
\hline
 3 & $U(3)\oplus U^{2}\oplus A_{2}^{2}$&6\\
\hline
 5 & $U(5)\oplus U^{2} $&4\\
\hline
 7 & $U\oplus\begin{pmatrix} 4 & -3 \\ -3 & 4\end{pmatrix}$&3\\
\hline
\multicolumn{3}{c}{Symplectic quotients}
\end{tabular}
\qquad
\begin{tabular}{|c|c|c|}
\hline
$p$ & $(H^{2}(X/G,\Z),\cdot)$&$\#\Sing X/G$\\
\hline
 3 & $U\oplus E_{6}$&3\\
\hline
 5 & $\begin{pmatrix} 2 & 5 \\ 5 & 10\end{pmatrix}\oplus A_4$&4\\
\hline
 7 & $U\oplus \begin{pmatrix} -4 & 3 \\ 3 & -4\end{pmatrix}$&3\\
\hline
 11 & $U$&2\\
\hline
 17 & $U(17)\oplus L_{17}^{\vee}(17)$&7\\
\hline
 19 & $U(19)\oplus\begin{pmatrix} -10 & 9 \\ 9 & -10\end{pmatrix}$&5\\
\hline
\multicolumn{3}{c}{Non-symplectic quotients}
\end{tabular}
\end{center}
\end{cor}
\begin{proof}
The automorphisms of prime order on a K3 surface were classified in \cite{Sarti} and \cite{Sarti2}. For all possible automorphism group $G$, the invariant lattice $H^2(X,\Z)^G$ is computed. Applying Proposition~\ref{CorK3} (i) and \cite[Proposition 3.5]{Lol3}, we obtain the previous tables (the case of the symplectic involution have already been studied in \cite[Proposition 1.1]{Lol2}). 
\end{proof}
For instance, Proposition \ref{CorK3} can also be applied to hypersurfaces in $\Pj^3$ (see Proposition \ref{cohcomplete}). We propose an example.
\begin{ex}
Let $X$ be the hypersurface of dimension $2$ and degree $p$ with equation:
$$x_0^p+x_1^p+x_2^p+x_3^p=0.$$
Let $g$ be the automorphism on $X$ induced by the following automorphism on $\mathbb{P}^{3}$:
$$\xymatrix@R0pt{ \Pj^{2n+1}\ar[r]&\Pj^{2n+1}\\
(a_0:a_1:a_2:a_3)\ar@{|->}[r]&(a_0:\xi_pa_1:\xi_p^la_2:a_3),
}$$
with $\xi_p$ a $p^\mathrm{th}$ root of the unity and $l\not\equiv 1$ or $0$ mod $p$. We denote $G=\left\langle g\right\rangle$.
Then:
\begin{itemize}
\item[(i)]
$\alpha_2(X)=0$, 
\item[(ii)]
$H^2(X/G,\Z)$ and $H^4(X/G,\Z)$  are torsion-free and $H^3(X/G,\Z)=\Z/p\Z$.
\item[(iii)]
$b_2(X/G)=p^2-3p+3$.
\item[(iv)]
$\discr H^2(X/G,\Z)=p^{p-2}$.
\end{itemize}
\end{ex}
\begin{proof}
Statements (i) and (ii) are direct consequence of Proposition \ref{CorK3}.
The fixed points of $G$ are given by $X\cap V(x_1,x_2)$ which consists in $p$ isolated points. The point (iv) follows from Proposition \ref{CorK3}.
Finally, by Remark~\ref{b2surface}, we have $b_2(X/G)=\eta(G)-2+\frac{b_2(X)+2-\eta(G)}{p}$ and by Riemann--Roch theorem (see Remark~\ref{rr}):
$$b_2(X)=p^3-4p^2+6p-2.$$ We obtain statement (iii). 
\end{proof}

\subsection{An example with a group of order higher than $19$}
As an application of Theorem \ref{gene19}, we investigate an example when the order of the group is higher than~19. 
A hyperk\"ahler manifold equivalent by deformation to a Hilbert scheme of $m$ points on a K3 surface will be called a hyperk\"ahler manifold of \emph{$K3^{[m]}$-type}.
\begin{prop}
Let $X$ be a hyperk\"ahler manifold of $K3^{[2]}$-type endowed with an automorphism group of order 23 $($it exists by \cite[Theorem 1.1]{23}$)$. Then:
\begin{itemize}
\item[(i)]
the spectral sequences of equivariant cohomology of $(X,G)$ with coefficients in $\Z$ and in $\mathbb{F}_{23}$ degenerate at the second page;
\item[(ii)]
$\Fix G=\left\{pt_1,pt_2\right\}$ or $\Fix G=R$, 
where $pt_1$ and $pt_2$ are two isolated points and $R$ is a rational curve.
\item[(iii)]
$\alpha_{k}(X)=0$ for all $1\leq k \leq 8$;
\item[(iv)]
$\rk H^2(X/G,\Z)=1$, $\rk H^4(X/G,\Z)=12$ and $\chi(X/G)=16$;
\item[(v)]
$H^{*}_f(X/G,\Z)$ is unimodular.
\end{itemize}
\end{prop}
\begin{proof}
The second cohomology group $H^2(X,\Z)$ endowed with the Beauville--Bogomolov form $B_X$ is a lattice of discriminant 2. 
By \cite[Corollary 5.4]{23}, we have $H^2(X,\Z)^G=\Z D$ with $B_X(D,D)=46$.
Hence by Theorem \ref{gene19} (v), $\ell_{23}^2(X)+\ell_+^2(X)=1$. 
It follows by Proposition \ref{Mongprop} and (\ref{BasicLatticeTheory}) that: 
\begin{equation}
\ell_{23}^2(X)=1\ \text{and}\ \ell_+^2(X)=0.
\label{ell2}
\end{equation}
Therefore, by Theorem \ref{gene19} (v): 
\begin{equation}
\ell_-^2(X)=0.
\label{ell-23}
\end{equation}

Since $\Sym^2 N_p=N_p^{\frac{p+1}{2}}$ for all prime number $p$, we obtain by \cite[Proposition 6.6]{SmithTh} that:
\begin{equation}
\ell_+^4(X)=\ell_-^4(X)=0\ \text{and}\ \ell_{23}^4(X)=12.
\label{ell4}
\end{equation}
Moreover by \cite[Proposition 3.9]{Lol3}:
\begin{equation}
\ell_+^6(X)=\ell_-^6(X)=0\ \text{and}\ \ell_{23}^6(X)=1.
\label{ell6}
\end{equation}
We recall that the cohomology of $X$ is torsion-free by \cite[Theorem 2.2]{Totaro}.
It follows from Propositions~\ref{Lemma3.1} and~\ref{equivarcoho} that the second page of the spectral sequence of equivariant cohomology of $(X,G)$ (with coefficients in $\Z$ or $\mathbb{F}_{23}$) is "almost trivial": \emph{i.e.} only the terms $E^{0,2k}_2$ and $E^{2k,0}_2$ are not trivial for all $k\in\mathbb{N}$. Hence, this spectral sequence necessarily degenerates at the second page. This prove (i).

Then, it follows from Proposition \ref{BNSFormula}, (\ref{ell2}), (\ref{ell-23}), (\ref{ell4}) and (\ref{ell6}) that:
$$h^*(\Fix G,\mathbb{F}_{23})=2.$$
By the universal coefficient theorem (see (\ref{rapporteur0})), we have:
$$
h^*(\Fix G,\mathbb{F}_{23})=2t_{23}^*(\Fix G)+\rk H^*_f(\Fix G,\Z),
$$
where $t_{23}^*(\Fix G)=\dim_{\mathbb{F}_{23}}\tors_{23} H^*(\Fix G,\Z)$.
So:
\begin{equation}
2t_{23}^*(\Fix G)+\rk H^*_f(\Fix G,\Z)=2.
\label{fix23}
\end{equation}
Moreover, by Proposition \ref{mainlefsch} and (\ref{ell2}), (\ref{ell-23}), (\ref{ell4}), (\ref{ell6}), we have:
\begin{equation}
\chi(\Fix G)=2.
\label{fix23bis}
\end{equation}
If we add (\ref{fix23}) and (\ref{fix23bis}), we obtain:
\begin{equation}
\dim_{\mathbb{F}_{23}}\tors_{23} H^*(\Fix G,\Z)+\rk H^{2*}_f(\Fix G,\Z)=2.
\label{fix23ter}
\end{equation}
In particular $\Fix G$ cannot contain any surface because if $S$ is a K\"ahler surface then:
$$\rk H^{2*}_f(S,\Z)\geq 3.$$
If $\Fix G$ contains a curve then it is a rational curve. 
Indeed let $C$ be a curve, then $\rk H^{2*}_f(C,\Z)=2$.
Hence, we obtain by (\ref{fix23bis}) that $H^1(C,\Z)=0$.
Moreover in this case, by (\ref{fix23ter}),
$\Fix G$ cannot have any other connected component.
If $\Fix G$ contains only isolated points, then (\ref{fix23bis}) provides:
$$\#\Fix G=2;$$
that is $\Fix G$ contains only two isolated points. It proves (ii).

Since the odd cohomology of $X$ is trivial, we only need to show the statement (iii) for $k$ even. 
By (\ref{ell2}), (\ref{ell4}), (\ref{ell6}) and Proposition \ref{vraisans19ter}, we obtain that $\alpha_2(X)=\alpha_4(X)=\alpha_6(X)=0$.
Moreover, by Remark~\ref{alpha0}, $\alpha_0(X)=1$ and $\alpha_8(X)=0$. 

The statement (iv) is a direct consequence of (\ref{ell2}), (\ref{ell4}), (\ref{ell6}) and Theorem \ref{gene19} (v) because we have 
$\rk H^k(X/G,\Z)=\rk H^k(X,\Z)^G$, for all $0\leq k\leq 8$. 

We have found with statement (iii) that: $$\frac{1}{23}\pi_*\left(H^0(X,\Z)\right)\bigoplus_{k=1}^{8}\pi_*\left(H^k(X,\Z)\right)_f=H^*_f(X/G,\Z).$$
Hence statement (v) follows from (\ref{ell2}), (\ref{ell4}) and Proposition \ref{vraisans19}.
\end{proof}
\subsection{Examples of degenerations at the second page of equivariant cohomology spectral sequences}\label{K32dege}
Let $S$ be a K3 surface. We denote by $S^{[m]}$ the Hilbert scheme of $m$ points on $S$. We recall that an automorphism group $G$ on $S^{[m]}$ is said to be \emph{natural} if it is induced by an automorphism group on the K3 surface $S$. Let $X$ be a hyperk\"ahler manifold of \emph{$K3^{[m]}$-type} (equivalent by deformation to a Hilbert scheme of $m$ points on a K3 surface). A pair $(X,G)$ is said to be \emph{standard} if it is deformation equivalent to a pair $(S^{[m]},G)$ with $G$ a natural group.
\begin{cor}\label{degecoro}
Let $X$ be a hyperk\"ahler manifold of $K3^{[m]}$-type and $G$ an automorphism group of prime order $3\leq p$ such that $(X,G)$ is a standard pair. Assume that $G$ has only isolated fixed points, then the spectral sequence of equivariant cohomology with coefficients in $\F$ degenerates at the second page.
\end{cor}
\begin{rmk}
When $m=2$, this theorem is a particular case of \cite[Theorem 1.1]{SmithTh}.
\end{rmk}
\begin{rmk}\label{corroHilbert}
Especially, when $G$ verifies the hypothesis of Corollary \ref{degecoro}, we can apply Theorem \ref{main2}. So $\alpha_k(X)=0$, for all $1\leq k\leq 4m$ and $H^{2*}(X/G,\Z)$ is torsion-free.
\end{rmk}
To prove this corollary we need to recall the integral basis of $H^*(S^{[m]},\Z)$ constructed in \cite{Wang}.
Let 
$$Q^{[l+q,l]}=\left\{\left.(\xi,x,\eta)\in S^{[l+q]}\times S\times S^{[l]}\right|\xi\supset \eta,\ \Supp(I_{\eta}/I_{\xi})=\left\{x\right\}\right\},$$
with $l\geq 0$ and $q>0$. We set 
$$\mathfrak{a}_{-q}(\alpha)(A)=\widetilde{p}_{1*}\left(\left[Q^{[l+q,l]}\right]\cdot\widetilde{\rho}^{\ *}\alpha\cdot \widetilde{p}_{2}^{\ *}A\right),$$
for $A\in H^{*}(S^{[l]})$ and $\alpha\in H^{*}(S)$ , where $\widetilde{p}_{1}$, $\widetilde{\rho}$, $\widetilde{p}_{2}$ are the projections from $S^{[l+q]}\times S\times S^{[l]}$ to $S^{[l+q]}$, $S$, $S^{[l]}$ respectively. We also set $\left|0\right\rangle\in H^{*}(S^{[0]},\Z))$ the unit and $1\in H^{0}(S,\Z)$, $x\in H^{4}(S,\Z)$ the generators. For a partition $\lambda=(1^{\lambda_1},2^{\lambda_2},...)$, with $\lambda_r$ the number of parts equal to $r$, we consider $\left|\lambda\right|=\sum_{r\geq1}r\lambda_r$, $\mathfrak{z}_{\lambda}=\prod_{r\geq1}r^{\lambda_r}\lambda_r!$ and
$$\mathfrak{a}_{-\lambda}(\alpha)=\prod_{r\geq1}\mathfrak{a}_{-r}(\alpha)^{\lambda_r}.$$
Let $(\alpha_t)_{t\in\left\{1,...,22\right\}}$ be an integral basis of $H^2(S,\Z)$.
We have by \cite[Theorem 1.1 and Remark 5.6]{Wang} that the classes 
\begin{equation}
\frac{1}{\mathfrak{z}_{\lambda}}\mathfrak{a}_{-\lambda}(1)\mathfrak{a}_{-\mu}(x)\mathfrak{m}_{\nu^1,\alpha_1}\cdots\mathfrak{m}_{\nu^{22},\alpha_{22}}\left|0\right\rangle,\ \left|\lambda\right|+\left|\mu\right|+\sum_{i=1}^{22}\left|\nu^i\right|=m
\label{basis}
\end{equation}
form an integral basis of $H^*(S^{[m]},\Z)$, where $\mathfrak{m}_{\nu^i,\alpha_i}$ is a polynomial of the $\mathfrak{a}_{-j}(\alpha_i)$, $j>0$ with rational coefficients (for instance $\mathfrak{m}_{1,1,\alpha}=\frac{1}{2}(\mathfrak{a}_{-1}(\alpha)^2-\mathfrak{a}_{-2}(\alpha))$; see \cite[Section 4]{Wang} or \cite[Section 1]{Kapfer2} for the precise construction).

Let $\phi$ be an automorphism on $S$. For simplicity, we also denote by $\phi$ the induced automorphism on $S^{[d]}$ for all $d>0$ and on $S^{[l+q]}\times S\times S^{[l]}$.
\begin{lemme}\label{actionbasis}
We have $\phi^*(\mathfrak{a}_{-q}(\alpha)(A))=\mathfrak{a}_{-q}(\phi^*(\alpha))(\phi^*(A))$.
\end{lemme}
\begin{proof}
Since $\phi$ is an automorphism on a compact complex manifold, we have $\phi^*=(\phi^{-1})_*$. Moreover, $\phi^*\left[Q^{[l+q,l]}\right]=\left[Q^{[l+q,l]}\right]$.
Hence:
\begin{align*}
\phi^*\widetilde{p}_{1*}\left(\left[Q^{[l+q,l]}\right]\cdot\widetilde{\rho}^{\ *}\alpha\cdot \widetilde{p}_{2}^{\ *}A\right)&=
\widetilde{p}_{1*}\left(\left[Q^{[l+q,l]}\right]\cdot \phi^*\widetilde{\rho}^{\ *}\alpha\cdot \phi^*\widetilde{p}_{2}^{\ *}A\right)\\
&=\widetilde{p}_{1*}\left(\left[Q^{[l+q,l]}\right]\cdot \widetilde{\rho}^{\ *}\phi^*\alpha\cdot \widetilde{p}_{2}^{\ *}\phi^*A\right).
\end{align*}
\end{proof}
Now, we are ready to prove Corollary \ref{degecoro}.
\begin{proof}[Proof of Corollary \ref{degecoro}]
Let $(X,G)$ be a standard pair. We are going to prove that $\ell_+^{2*+1}(X)=0$ and apply Theorem \ref{degenemain} (iii). The $\Z[G]$-module structure of $H^*(X,\Z)$ is invariant under deformation. Hence the $\ell_+^k$ are also independent under deformation for all integer $k$. Hence, without loss of generality, we can assume that $X=S^{[m]}$ and $G$ is natural.

Since $H^{*}(S^{[m]},\Z)$ is torsion-free by \cite[Theorem 2.2]{Totaro}, we have $H^{*}(S^{[m]},\F)=H^{*}(S^{[m]},\Z)\otimes\F$. Then by (\ref{basis}) the classes
\begin{equation}
\frac{1}{\mathfrak{z}_{\lambda}}\mathfrak{a}_{-\lambda}(1)\mathfrak{a}_{-\mu}(x)\mathfrak{m}_{\nu^1,\alpha_1}\cdots\mathfrak{m}_{\nu^{22},\alpha_{22}}\left|0\right\rangle\otimes\overline{1},\ \left|\lambda\right|+\left|\mu\right|+\sum_{i=1}^{22}\left|\nu^i\right|=m
\label{basisF}
\end{equation}
form a basis of the $\F$-vector space $H^*(S^{[m]},\F)$, with $\overline{1}$ the unit in $\F$.
Moreover by Lemma \ref{actionbasis},
$$\phi^*\frac{1}{\mathfrak{z}_{\lambda}}\mathfrak{a}_{-\lambda}(1)\mathfrak{a}_{-\mu}(x)\mathfrak{m}_{\nu^1,\alpha_1}\cdots\mathfrak{m}_{\nu^{22},\alpha_{22}}\left|0\right\rangle\otimes\overline{1}=\frac{1}{\mathfrak{z}_{\lambda}}\mathfrak{a}_{-\lambda}(1)\mathfrak{a}_{-\mu}(x)\mathfrak{m}_{\nu^1,\phi^*\alpha_1}\cdots\mathfrak{m}_{\nu^{22},\phi^*\alpha_{22}}\left|0\right\rangle\otimes\overline{1},$$
where we use the notation $\phi$ for the morphism on $S$ and on $S^{[m]}$.

Let $\mu=(1^{\mu_1},2^{\mu_2},...)$ and $\lambda=(1^{\lambda_1},2^{\lambda_2},...)$ be two partitions; their sum is defined by the formula $\mu+\lambda:=
(1^{\mu_1+\lambda_1},2^{\mu_2+\lambda_2},...)$.
Let $\lambda$, $\mu$ and $\upsilon$ be three partitions such that $\left|\lambda\right|+\left|\mu\right|+\left|\upsilon\right|=m$. We consider the $\F$-vector space $V_{\lambda,\mu,\upsilon}$ generated by all the elements $\frac{1}{\mathfrak{z}_{\lambda}}\mathfrak{a}_{-\lambda}(1)\mathfrak{a}_{-\mu}(x)\mathfrak{m}_{\nu^1,\alpha_1}\cdots\mathfrak{m}_{\nu^{22},\alpha_{22}}\left|0\right\rangle\otimes\overline{1}$ such that $\sum_{i=1}^{22}\nu^i=\upsilon$. Note that $V_{\lambda,\mu,\upsilon}$ is stable under the action of $\phi$ and we have an isomorphism of $\F[G]$-modules: 
\begin{equation}
V_{\lambda,\mu,\upsilon}\simeq \bigotimes_{i}\Sym^{\upsilon_i} H^2(S,\F),
\label{symy}
\end{equation}
where $\upsilon=(1^{\upsilon_1},2^{\upsilon_2},...)$. Moreover by (\ref{basisF}), we have an isomorphism of $\F[G]$-modules:
\begin{equation}
H^*(S^{[m]},\F)\simeq \bigoplus_{\left|\lambda\right|+\left|\mu\right|+\left|\upsilon\right|=m}V_{\lambda,\mu,\upsilon}.
\label{symy2}
\end{equation}
To end the proof our objective is to prove that the Jordan decomposition of $H^*(S^{[m]},\F)$ does not contain any $N_{p-1}$ term (where the $N_q$ where introduced in Section \ref{defiinvar}); then we will have shown that $\ell_-^{2*}(S^{[m]})=0$.
We can compute that:
\begin{equation}
N_1\otimes N_1=N_1,\ N_1\otimes N_p=N_p\ \text{and}\ N_p\otimes N_p=N_p^{\frac{p+1}{2}}.
\label{tensor}
\end{equation}
Hence it is enough by (\ref{symy}) and (\ref{symy2}) to show that the Jordan decomposition of $\Sym^{j} H^2(S,\F)$ does not contain any $N_{p-1}$ term for all $j\in\N$. 
We know by \cite[Proposition 4.5]{SmithTh} and Theorem \ref{degenemain} that $\ell_-^2(S)=0$. Hence,
$H^2(S,\F)=N_1^{\ell_+(S)}\oplus N_p^{\ell_p(S)}$. 
We have $\Sym^jN_1=N_1$. Moreover by \cite[Proposition 3.2 and 3.3]{Almkvist}:
\begin{equation}
\Sym^dN_p=N_p^{\frac{1}{p}\binom{d+p+1}{d}},
\label{sym}
\end{equation}
if $p$ does not divide $d$ and
\begin{equation}
\Sym^{pl}N_p=N_1\oplus N_p^{\frac{1}{p}\left[\binom{p(l+1)-1}{pl}\right]-1},
\label{sym2}
\end{equation}
Thus, we can prove recursively on $\ell_+(S)+\ell_p(S)$ that the Jordan decomposition of $\Sym^j H^2(S,\F)$ does not contain any $N_{p-1}$ term. If $\ell_+(S)+\ell_p(S)=1$, this is a direct consequence of (\ref{sym}) and (\ref{sym2}). We assume the result for $\ell_+(S)+\ell_p(S)= n$ and we prove it for $\ell_+(S)+\ell_p(S)=n+1$. We can write
\[H^2(S,\F)=\left(N_1^{\ell_+(S)}\oplus N_p^{\ell_p(S)-1}\right)\oplus N_p\text{ or } H^2(S,\F)=\left(N_1^{\ell_+(S)-1}\oplus N_p^{\ell_p(S)-1}\right)\oplus N_1.\]
Hence, we can write:
$$\Sym^j H^2(S,\F)=\bigoplus_{k=0}^j\left[\Sym^k\left(N_1^{\ell_+(S)}\oplus N_p^{\ell_p(S)-1}\right)\otimes \Sym^{j-k}N_p\right],\ \ \ \ \text{or:}$$
$$\Sym^j H^2(S,\F)=\bigoplus_{k=0}^j\left[\Sym^k\left(N_1^{\ell_+(S)-1}\oplus N_p^{\ell_p(S)}\right)\otimes \Sym^{j-k}N_1\right].$$
In both cases the recursive hypothesis, (\ref{tensor}), (\ref{sym}) and (\ref{sym2}) show that the Jordan decomposition of $\Sym^j H^2(S,\F)$ does not contain any $N_{p-1}$ term for $\ell_+(S)+\ell_p(S)=n+1$. This proves the result for any $n\in \N$.
 
Finally, we have shown that $\ell_-^{2*}(S^{[m]})=0$.
Then, the result is a consequence of Theorem \ref{degenemain}.
\end{proof}
\begin{rmk}
Let $G$ be an automorphism group of prime order $p$ on $X$ a hyperk\"ahler manifold of $K3^{[m]}$-type. Note that $p\leq23$, with 23 possible only for non-standard automorphisms. Indeed, we can write from Theorem \ref{gene19} (v):
\begin{equation}
23=\ell_+^2(X)+(p-1)\ell_-^2(X)+p\ell_p^2(X).
\label{23Hilb}
\end{equation}
By \cite[Proposition 10]{Beauville1982}, $\ell_-^2(X)$ or $\ell_p^2(X)$ is non-trivial. Hence $23\geq p-1$, so for $p$ prime $p\leq 23$. If $G$ is standard, then by Theorem \ref{gene19} (v):
$\ell_+^2(X)+\ell_p^2(X)\geq 2$, hence (\ref{23Hilb}) implies that $p\leq 19$.
\end{rmk}
We can also provide another example of application of Theorem \ref{degenemain} when the odd cohomology of the manifold is not trivial. The following proposition is a direct consequence of Theorem \ref{degenemain} (iii) and \cite[Proposition 1]{Erra}.
\begin{prop}
Let $A$ be a 2-dimensional complex torus. 
We consider the following embedding:
\[j:A\times A\hookrightarrow A\times A \times A,\ (x,y)\mapsto (x,y,-x-y).\]
The action of the alternating group $\mathfrak{A}_3$ on $A\times A \times A$ provides an action on $A\times A$ via the embedding $j$.
Then, the spectral sequence of equivariant cohomology with coefficients in $\mathbb{F}_3$ of $(A \times A,\mathfrak{A}_3)$ degenerates at the second page.
\end{prop}
\subsection{Quotients of $K3^{[m]}$-type hyperk\"ahler manifolds by symplectic automorphisms of order 5 and 7}\label{ProofBB}
If we consider $\phi$ a symplectic automorphism of order $5$ (resp. $7$) on $S$, then the fixed locus of the induced automorphism $\phi^{[m]}$ on $S^{[m]}$ has only isolated fixed points when $m\leq 4$ (resp. $m\leq 6$). Moreover, it is shown in \cite[Theorem 7.2.7, Section 7.3]{MongT} and \cite[Theorem 2.5]{Mong2} that all the symplectic automorphisms of order 5 and 7 on a hyperk\"ahler manifold of $K3^{[m]}$-type with $m\leq 6$ are standard.
So by Remark \ref{corroHilbert}, we obtain the following result.
\begin{cor}\label{p3}
Let $X$ be a hyperk\"ahler manifold of $K3^{[m]}$-type with $m\leq 4$ $($resp. $m\leq 6)$ and $G$ a symplectic automorphism group of order $5$ $($resp. $7)$. Then 
$\alpha_{k}(X)=0$ for all $1\leq k \leq 4m$ and $H^{2*}(X/G,\Z)$ is torsion-free.
\end{cor}
We can also compute the torsion of $H^{2*+1}(S^{[m]}/\phi^{[m]},\Z)$
(hence also the torsion of $H^{2*+1}(X/G,\Z)$) 
if we are patient enough using (\ref{basisF}), \cite[Proposition 3.9]{Lol3}, and Theorem \ref{main2} (iii).  We also find the number of fixed points $\eta(\phi^{[m]})$ by Theorem \ref{degenemain}. For instance, we give the computation for $m=2$ and $3$:
\begin{rmk}\label{torsion2}
\leavevmode
\begin{description}
\item[\mathversion{bold}$m=2$ and $p=5$\mathversion{normal}] we have $\ell_+^2(S^{[2]})=2+1=3$ and $\ell_+^4(S^{[2]})=3+2+1=6$. Hence, $\eta(\phi^{[2]})=14$ and:
$$H^3(S^{[2]}/\phi^{[2]},\Z)\oplus H^7(S^{[2]}/\phi^{[2]},\Z)=(\Z/5\Z)^{11},\ H^5(S^{[2]}/\phi^{[2]},\Z)=(\Z/5\Z)^{4}.$$
\item[\mathversion{bold}$m=2$ and $p=7$\mathversion{normal}]
we have $\ell_+^2(S^{[2]})=1+1=2$ and $\ell_+^4(S^{[2]})=1+1+1=3$. Hence, $\eta(\phi^{[2]})=9$ and:
$$H^3(S^{[2]}/\phi^{[2]},\Z)\oplus H^7(S^{[2]}/\phi^{[2]},\Z)=(\Z/7\Z)^{7},\ H^5(S^{[2]}/\phi^{[2]},\Z)=(\Z/7\Z)^{3}.$$
\item[\mathversion{bold}$m=3$ and $p=5$\mathversion{normal}]
we have $\ell_+^2(S^{[3]})=2+1=3$, $\ell_+^4(S^{[3]})=3+2+2+1+1=9$ and $\ell_+^6(S^{[3]})=4+4+2+2+1+1=14$. Hence, $\eta(\phi^{[3]})=40$ and:
$$H^3(S^{[3]}/\phi^{[3]},\Z)\oplus H^{11}(S^{[3]}/\phi^{[3]},\Z)=(\Z/5\Z)^{37},\ H^5(S^{[3]}/\phi^{[3]},\Z)\oplus H^9(S^{[3]}/\phi^{[3]},\Z)=(\Z/5\Z)^{31},$$ $$H^7(S^{[3]}/\phi^{[3]},\Z)=(\Z/5\Z)^{13}.$$
\item[\mathversion{bold}$m=3$ and $p=7$\mathversion{normal}]
we have $\ell_+^2(S^{[3]})=1+1=2$, $\ell_+^4(S^{[3]})=1+1+1+1+1=5$ and $\ell_+^6(S^{[3]})=1+1+1+1+1+1=6$. Hence, $\eta(\phi^{[3]})=22$ and:
$$H^3(S^{[3]}/\phi^{[3]},\Z)\oplus H^{11}(S^{[3]}/\phi^{[3]},\Z)=(\Z/7\Z)^{20},\ H^5(S^{[3]}/\phi^{[3]},\Z)\oplus H^9(S^{[3]}/\phi^{[3]},\Z)=(\Z/7\Z)^{17},$$ $$H^7(S^{[3]}/\phi^{[3]},\Z)=(\Z/7\Z)^{8}.$$
\end{description}
\end{rmk}
Now, we are going to prove Theorem \ref{M5} and \ref{M7}.
To prove these two theorems, we first need an analogous of \cite[Lemma 7.11]{Lol3}. We recall the the discriminant group of a lattice $L$ is denoted by $A_L$ (see Section \ref{remindersL}).
\begin{lemme}\label{lemfin}
Let $G=\left\langle \phi\right\rangle$ be a natural automorphism group of prime order $p$ on $S^{[m]}$. We endow $H^2(S^{[m]},\Z)$ with the Beauville--Bogomolov form. We have:
\begin{itemize}
\item[(i)]
$A_{H^2(S^{[m]},\Z)^G}=(\Z/2(m-1)\Z)\oplus (\Z/p\Z)^{\ell_p^2(S^{[m]})},$
\item[(ii)]
$\discr H^2(S^{[m]},\Z)^G=2(m-1)p^{\ell_p^2(S^{[m]})}$.
\item[(iii)]
We denote $A_{H^{2}(S^{[m]},\Z)^{G},p}:=(\Z/p\Z)^{\ell_{p}^2(S^{[m]})}$.
Then, the projection $$\frac{H^{2}(S^{[m]},\Z)}{H^{2}(S^{[m]},\Z)^{G}\oplus \ker \rho}\rightarrow A_{H^{2}(S^{[m]},\Z)^{G},p}$$ is an isomorphism, 
where $\rho=\id+\phi^*+\cdots+(\phi^{*})^{p-1}$.
\item[(iv)]
Moreover, let $x\in H^{2}(S^{[m]},\Z)^{G}$. We have $\frac{x}{p}\in (H^{2}(S^{[m]},\Z)^{G})^{\vee}$ if and only if there is $z\in H^{2}(S^{[m]},\Z)$ such that $x=z+\phi^*(z)+\cdots+(\phi^{*})^{p-1}(z)$. 
\item[(v)]
Also:
$$\frac{\pi_*(H^{2}(S^{[m]},\Z))}{\pi_*(H^{2}(S^{[m]},\Z)^G)}=(\Z/p\Z)^{\ell_{p}^2(S^{[m]})}.$$
\end{itemize}
\end{lemme}
\begin{proof}
We denote by $\delta$ half the class of the diagonal in $H^2(S^{[m]},\Z)$. We endow $H^2(S,\Z)$ with the cup-product. Since our automorphism is natural, we have an isometry and an isomorphism of $\Z[G]$-modules:
\begin{equation}
H^{2}(S^{[m]},\Z)\simeq H^2(S,\Z)\oplus \Z\delta,
\label{Beauville}
\end{equation}
with $G$ which acts trivially on $\delta$ and $\delta^2=-2(m-1)$ (see \cite[Proposition~6 and Remark~1 (\S9)]{Beauville1983}).
Hence:
\begin{equation} 
H^2(S^{[m]},\Z)^G\simeq H^2(S,\Z)^G\oplus \Z\delta,
\label{rapporteur00}
\end{equation}
and $\ell_p^2(S^{[m]})=\ell_p^2(S)$.
Therefore, we obtain (ii) by \cite[Proposition 2.15]{Lol3}. 

From (\ref{rapporteur00}), we also have:
\begin{equation}
A_{H^2(S^{[m]},\Z)^G}=(\Z/2(m-1)\Z)\oplus A_{H^2(S,\Z)^G}.
\label{rapporteur000}
\end{equation}
However by Proposition \ref{Mongprop}:
$$\frac{H^2(S^{[m]},\Z)}{H^2(S^{[m]},\Z)^G\oplus \ker \rho}=\frac{H^2(S,\Z)}{H^2(S,\Z)^G\oplus \ker \rho}=(\Z/p\Z)^{\ell_p^2(S)}.$$
Since $H^2(S,\Z)$ is unimodular, we obtain by (\ref{DiscrUni}) that
$A_{H^2(S,\Z)^G}=(\Z/p\Z)^{\ell_p^2(S)}$. Hence (i) and (iii) follow from (\ref{rapporteur000}).

Then (iv) and (v) are proved exactly as (iv) and (v) of \cite[Lemma 7.11]{Lol3}. 
\end{proof}
Now we are ready to prove Theorem \ref{M5} and \ref{M7}. The proof is very similar to the proof of \cite[Theorems~1.2 and~1.3]{Lol3}. We only prove Theorem \ref{M7}, the proof of Theorem \ref{M5} is identical.

\begin{proof}[Proof of Theorem \ref{M7}]
The Beauville--Bogomolov form is invariant by deformation, hence we can assume without loss of generality that $X=S^{[m]}$ and $G$ is natural.

From \cite[Theorem 4.1]{Sarti} and (\ref{Beauville}), there is an isometry of lattices 
\[H^{2}(S^{[m]},\Z)^{G}\simeq U(7)\oplus \begin{pmatrix} 4 & 1 \\ 1 & 2\end{pmatrix}\oplus (-2(m-1)).\]
In the rest of the proof, we identify $H^{2}(S^{[m]},\Z)^{G}$ with the lattice $U(7)\oplus \begin{pmatrix} 4 & 1 \\ 1 & 2\end{pmatrix}\oplus (-2(m-1))$.

By Lemma \ref{lemfin} (iv), we have:
\begin{equation}
\frac{1}{7}\pi_{*}(U(7))\subset H^{2}(M^m_{7},\Z).
\label{LastEquations3}
\end{equation}
Let $(a, b)$ be an integral basis of the lattice $\begin{pmatrix} 4 & 1 \\ 1 & 2\end{pmatrix}$, with $B_{S^{[m]}}(a,a)=4$,  $B_{S^{[m]}}(b,b)=2$ and $B_{S^{[m]}}(a,b)=1$.
Idem, by Lemma \ref{lemfin} (iv), we have: 
$$\frac{\pi_*(a)+3\pi_*(b)}{7}\in H^{2}(M^m_{7},\Z).$$
The lattice generated by $a+3b$ and $a-4b$ is isomorphic to $\Lambda(7):=\begin{pmatrix} 4\times7 & -3\times7 \\ -3\times7 & 4\times7\end{pmatrix}$, with $\Lambda:= \begin{pmatrix} 4 & -3 \\ -3 & 4\end{pmatrix}$.
We have: 
\begin{equation}
\frac{1}{7}\pi_*(\Lambda(7))\in H^{2}(M^m_{7},\Z).
\label{LastEquations4}
\end{equation}
Then, by (\ref{LastEquations3}) and (\ref{LastEquations4}):
$$\pi_*\left(H^{2}(S^{[m]},\Z)\right)\supset\pi_{*}\left(\frac{1}{7}U(7)\oplus\frac{1}{7}\Lambda(7)\oplus(-2(m-1))\right).$$
Therefore, by (ii) and (v) of Lemma \ref{lemfin}, we obtain:
$$\pi_*\left(H^{2}(S^{[m]},\Z)\right)=\pi_{*}\left(\frac{1}{7}U(7)\oplus\frac{1}{7}\Lambda(7)\oplus(-2(m-1))\right).$$
So by Corollary \ref{p3}, 
$$H^2(M^m_{7},\Z)=\pi_{*}\left(\frac{1}{7}U(7)\oplus\frac{1}{7}\Lambda(7)\oplus(-2(m-1))\right).$$
Then, by \cite[Proposition 7.10]{Lol3}, the Beauville--Bogomolov form of $H^{2}(M^m_{7},\Z)$ gives the lattice
\begin{align*}
&\frac{1}{7}U\left(7\sqrt[m]{\frac{(2m)!7^{2m-1}}{m!2^mC_{M^m_{7}}}}\right)\oplus \frac{1}{7}\Lambda\left(7\sqrt[m]{\frac{(2m)!7^{2m-1}}{m!2^mC_{M^m_{7}}}}\right)\oplus\left(-2(m-1)\sqrt[m]{\frac{(2m)!7^{2m-1}}{m!2^mC_{M^m_{7}}}}\right)\\
&=U\left(\sqrt[m]{\frac{(2m)!7^{m-1}}{m!2^mC_{M^m_{7}}}}\right)\oplus \Lambda\left(\sqrt[m]{\frac{(2m)!7^{m-1}}{m!2^mC_{M^m_{7}}}}\right)\oplus\left(-14(m-1)\sqrt[m]{\frac{(2m)!7^{m-1}}{m!2^mC_{M^m_{7}}}}\right),
\end{align*}
where $C_{M^m_{7}}$ is the Fujiki constant of $M^m_{7}$.
Then, knowing that the Beauville--Bogomolov form is integral and indivisible, we have $C_{M^m_{7}}=\frac{7^{m-1}(2m)!}{m!2^m}$ and we get the lattice:
$$U\oplus \Lambda\oplus (-14(m-1)).$$
\end{proof}
We can also compute the Betti numbers. We provide the computation for $m=2$ or $3$. 
\begin{prop}
We have:
$$\begin{tabular}{|c|c|c|c|c|}
\hline
$X/G$ & $b_{2}$ & $b_{4}$ & $b_6$ & $\#\Sing X/G$ \\
\hline
$M_5^2$ & 7 & 60 & 7 & 14 \\
\hline
$M_7^2$ & 5 & 42 & 5 & 9 \\
\hline
$M_5^3$ & 7 & 67 & 522 & 40 \\
\hline
$M_7^3$ & 5 & 47 & 370 & 22 \\
\hline
\end{tabular}$$
\end{prop}
\begin{proof}
In Remark \ref{torsion2}, we have computed the $\ell_+^*(S^{[m]})$ for $p=5,7$ and $m=2,3$. Moreover in the proof of Corollary~\ref{degecoro}, we have seen that $\ell_-^*(S^{[m]})=0$ for $p=5$ (resp. $7$) and $m\leq 4$ (resp. $m\leq 6$). The Betti numbers of $S^{[m]}$ are well known and were determined by G\"ottsche \cite{Got} (see, for instance, \cite[Remark 2.1]{Kapfer2} for the explicit values). Hence, we can deduce $\ell_p^*(S^{[m]})$ from the following equation of Theorem \ref{gene19} (v):
$$\rk H^*(S^{[m]},\Z)=\ell_+^*(S^{[m]})+(p-1)\ell_-^*(S^{[m]})+p\ell_p^*(S^{[m]}).$$
Then, we deduce the Betti number from Theorem \ref{gene19} (v):
$$\rk H^*(S^{[m]}/\phi^{[m]},\Z)=\rk H^*(S^{[m]},\Z)^{\phi^{[m]}} =\ell_+^*(S^{[m]})+\ell_p^*(S^{[m]}).$$
\end{proof}
With the same method, we can also provide the Betti numbers of  $M_5^m$ and $M_7^m$ when $m\geq 4$. With enough patience or a computer, we compute the $\ell_+^*(S^{[m]})$ directly from (\ref{symy}) (see \cite{Kapfer2} for similar techniques).
\subsection{Quotients of projective manifolds}\label{QPM}
\begin{defi}
Let $X$ be a smooth projective manifold of dimension $n$. Let $p$ be a prime number. We denote $u:=c_1(\mathcal{O}_{X}(1))\in H^{2}(X,\F)$.
 Let $L^k:H^{n-k}(X,\F)\rightarrow H^{n+k}(X,\F)$ be the Lefschetz maps given by $L^k(x)=x\cdot u^k$, for all $1\leq k\leq n$. 

We say that $X$ is \emph{Lefschetz $p$-torsion-free} if $L^k$ is an isomorphism for all $1\leq k\leq n$.
\end{defi}
\begin{defi}
Let $\rho:X\hookrightarrow \mathbb{P}^{N}$ be a projective manifold. An automorphism group on $X$ is said \emph{linear according to $\rho$} if it extends to an automorphism group of $\mathbb{P}^{N}$.
\end{defi}
\begin{prop}\label{degehyper}
Let $X$ be a smooth projective manifold
endowed with a finite automorphism group $G$ which is linear according to some embedding. 
Let $p$ be a prime number such that $X$ is Lefschetz $p$-torsion-free.
Then the spectral sequence of equivariant cohomology of $(X,G)$ with coefficients in $\F$ degenerates at the second page.
\end{prop}
\begin{proof}
The proof of this proposition is very similar to the one of \cite[Proposition 4.6]{SmithTh}. We refer to \cite[Section~4.3]{SmithTh} for the notion of \emph{$G$-linearisation}.
Since $G$ is a linear automorphism group, the line bundle $\mathcal{L}:=\mathcal{O}_X(1)$ is $G$-linearisable.
Hence, we can consider $\mathcal{L}_G:=\mathcal{L}\times_G EG$, with $EG\rightarrow
BG$ an universal $G$-bundle. Let $f:X_G\rightarrow BG$ be the projection.
We set $u:=c_1(\mathcal{L}_G)\in H^2(X_G,\F)$ and $q:=\RR f_*u:\RR f_*\F\rightarrow \RR f_*\F[2]$. Let $n$ be the dimension of $X$.
By \cite[Proposition 2.1]{Deligne}, if $q^k:R^{n-k}f_*\F\rightarrow R^{n+k}f_*\F$ is an isomorphism for all $k$, then the spectral sequence of equivariant cohomology with coefficients in $\F$ degenerates at the second page. 
Since we are considering direct images of constant sheaves along locally trivial fibrations, we can check that the $q^k$ are isomorphisms 
fibrewise. That is, we need to prove that the Lefschetz maps $L^k:H^{n-k}(X,\F)\rightarrow H^{n+k}(X,\F)$ are isomorphisms. This is given by hypothesis.
\end{proof}
The condition on Lefschetz maps can easily be verified for smooth complete intersections.  
Let $X$ be a smooth complete intersection of multidegree $(a_1,...,a_q)$. We call $\prod_i^qa_i$ the \emph{degree} of $X$.
We recall the following well known result.
\begin{prop}\label{cohcomplete}
Let $X$ be a smooth complete intersection of dimension $n$ and degree $d$. 
Let $D:=c_1(\mathcal{O}_{X}(1))$.
Then:
\begin{equation*}
 H^{k}\left(X,\Z\right)=
\begin{cases}
\Z D^{k/2} & \text{if }\ k\ \text{is even and}\ k<n,\\
0 & \text{if}\ k\ \text{is odd and}\ k\neq n,\\
\Z \frac{D^{k/2}}{d}& \text{if}\ k\ \text{is even and}\ k>n.\\
\end{cases}
\end{equation*}
Moreover $H^n(X,\Z)$ is torsion-free.
\end{prop}
\begin{rmk}\label{rr}
Let $X$ be a smooth complete intersection of dimension $n$ and degree $d$.
Note that $b_n(X)$ can be computed via Riemann--Roch theorem. For instance if $X$ is a hypersurface $$b_n(X)=\sum_{k=0}^n(-1)^k\binom{n+2}{k}d^{n+1-k}+(-1)^{n+1}2\left\lceil \frac{n}{2}\right\rceil.$$
\end{rmk}
It follows from the following corollary.
\begin{cor}\label{degecorcomplete}
Let $X$ be a smooth complete intersection of degree $d$
endowed with a finite automorphism group $G$ which is linear according to some embedding. 
Let $p$ be a prime number which does not divide $d$.
Then the spectral sequence of equivariant cohomology of $(X,G)$ with coefficients in $\F$ degenerates at the second page.
\end{cor}

Then, we are going to see that Theorem \ref{recall1} and \ref{recall3} can easily be applied to complete intersections.
\begin{lemme}\label{torsionhyper}
Let $X$ be a smooth complete intersection endowed with a finite abelian group $G$ which is linear according to some embedding.
Then $H^*(\Fix G,\Z)$ is torsion-free.
\end{lemme}
\begin{proof}
Let $X\hookrightarrow \mathbb{P}^N$ be an embedding such that the automorphism group $G$ extends to an automorphism group $\mathcal{G}$ on $\Pj^N$.
Then we have:
\begin{equation}
\Fix G=\Fix \mathcal{G}\cap X. 
\label{fixGP}
\end{equation}
Let $Y$ be a connected component of $\Fix G$. By (\ref{fixGP}), there exists a connected component $H$ of $\Fix \mathcal{G}$ such that
$Y\subset H\cap X$. Since $H$ have degree 1, we can write $H$ as an intersection of hyperplanes $H=H_1\cap\cdots\cap H_l$. 
Then we can prove recursively that $X\cap H_1\cap\cdots\cap H_l$ has its cohomology torsion-free; moreover $X\cap H_1\cap\cdots\cap H_l$ is connected or of dimension 0.
It is true for $X$, we assume that it is true for $X\cap H_1\cap\cdots\cap H_i$ and we prove it for $X\cap H_1\cap\cdots\cap H_{i+1}$.
There are two possibilities or $X\cap H_1\cap\cdots\cap H_{i+1}$ is a hyperplane section of $X\cap H_1\cap\cdots\cap H_i$ or $X\cap H_1\cap\cdots\cap H_i\subset H_{i+1}$. In the second case $X\cap H_1\cap\cdots\cap H_i=X\cap H_1\cap\cdots\cap H_i\cap H_{i+1}$ and there is nothing to prove. In the first case, this is a consequence of the hyperplane Lefschetz theorem and the universal coefficient theorem. 

Hence $Y$ is a point or $Y=H\cap X$. In the both cases, $H^*(Y,\Z)$ is torsion-free.
\end{proof}
\begin{lemme}\label{torsionhyper2}
Let $\rho:X\hookrightarrow \Pj^N$ be a smooth complete intersection of dimension $2n$ and degree $d$ in $\Pj^N$, endowed with a finite abelian automorphism group $G$ which is linear according to $\rho$. 
We assume that $\Fix G$ contains $\Delta$ a connected component of dimension $n$. Let $j:\Delta\hookrightarrow X$ be the embedding and $\left[\Delta\right]:=j_*(1)$. Let $p$ be a prime number. If $p$ does not divide $d$ then $\left[\Delta\right]$ is not divisible by $p$ in $H^{2n}(X,\Z)$.
\end{lemme}
\begin{proof} 
We have $D=\rho^*(c_1(\mathcal{O}_{\Pj^N}(1)))$. 
Let $P:=(D^n)^{\bot}$ be the lattice orthogonal to $D^n$ in $H^{2n}(X,\Z)$. Let $y\in P$. By projection formula, we have:
$0=\rho_*(y\cdot D^n)=\rho_*(y)\cdot c_1(\mathcal{O}_{\Pj^N}(1))^n$.
It follows that $\rho_*(y)=0$ for all $y\in P$.
Moreover $H^{2n}(X,\Z)$ is a unimodular lattice, hence by (\ref{DiscrUni}), we have $\left[\Delta\right]=\frac{aD^n+y}{d}$, with $a\in \Z$.
It follows by projection formula that: $$\rho_*(\left[\Delta\right])=\frac{a}{d}c_1(\mathcal{O}_{\Pj^N}(1))^{n}\cdot\left[X\right]=ac_1(\mathcal{O}_{\Pj^N}(1))^{N-n}.$$
However as in the previous proof, $\Delta=X\cap H$ with $H$ of degree 1 in $\mathbb{P}^N$. 
Then $a$ divides $d$. That is $d=qa$ with $q\in\Z$.
Hence $\left[\Delta\right]$ cannot be divisible by $p$ since $p$ does not divide $d$.
\end{proof}
As a consequence of Corollary \ref{degecorcomplete} and Lemmas \ref{torsionhyper}, \ref{torsionhyper2}, we can state the following conclusion.
\begin{cor}\label{maincorcomplete}
Let $\rho:X\hookrightarrow\Pj^N$ be a smooth complete intersection of dimension $n$ and degree $d$ in $\Pj^N$, endowed with an automorphism group $G$ of order $p$ which is linear according to $\rho$. We assume that $p$ does not divide $d$. Then:
\begin{itemize}
\item[(i)]
If $G$ has only isolated fixed points then the consequences of Theorem \ref{main2} are verified for $H^*(X/G,\Z)$.
\item[(ii)]
If all fixed points of $G$ are simple then the consequences of Theorem \ref{recall1} and \ref{recall3} are verified for $H^*(X/G,\Z)$.
\end{itemize} 
\end{cor}
\begin{rmk}
Note that in the framework of complete intersections of dimension 2, Proposition \ref{CorK3} can directly be applied.
\end{rmk}
When $X$ is a hypersurface the linear condition on $G$ is verified in most cases.
\begin{rmk}
Let $\rho:X\hookrightarrow\Pj^{n+1}$ be a smooth hypersurface of dimension $n$ and degree $d$. If $n\geq2$, $d\neq2$ and $(n,d)\neq(2,4)$, then an automorphism group of $X$ is always linear according to $\rho$ by \cite[Theorem 2]{Matsumura}. The case $(n,d)\neq(2,4)$ is the case of K3 surfaces and have been treated in Section \ref{K3Sec}.
\end{rmk}
We propose an example of use of Theorems \ref{recall3} and \ref{main2}.
\begin{ex}
Let $X$ be the hypersurface of dimension $2n\geq4$ and degree $p+1$ with equation:
$$\sum_{i=0}^{n}x_i^{p+1}+\sum_{i=0}^nx_ix_{n+1+i}^p=0.$$
Let $g$ be the automorphism on $X$ induced by the following automorphism on $\mathbb{P}^{2n+1}$:
$$\xymatrix@R0pt{ \Pj^{2n+1}\ar[r]&\Pj^{2n+1}\\
(a_0:a_1:\cdots :a_n)\ar@{|->}[r]&(a_0:a_1:\cdots :a_n:\xi_pa_{n+1}:\xi_pa_{n+2}:\cdots:\xi_pa_{2n+1}),
}$$
with $\xi_p$ a $p^\mathrm{th}$ root of the unity. We denote $G=\left\langle g\right\rangle$. Then:
\begin{itemize}
\item[(i)]
$\alpha_{2n}(X)=0$,
\item[(ii)]
if $n$ is even:
$$\log_p\discr H^{2n}(X/G,\Z)=1,$$
if $n$ is odd:
$$\log_p\discr H^{2n}(X/G,\Z)=\left(\sum_{k=0}^{n-1}(-1)^k\binom{n+1}{k}(p+1)^{n-k}\right)-n+1.$$
\end{itemize}

\end{ex}
\begin{proof}
Statement (i) is a direct consequence of Corollary \ref{maincorcomplete} (ii). To obtain (ii), we need to determine $\Fix G$. 
Let $H_1:=V(x_{0},x_{1},...,x_{n})$ and $H_2:=V(x_{n+1},x_{n+2},...,x_{2n+1})$. The space $H_1$ can be identified to the projective space $\Pj^n$. Then $\Fix G$ has two connected components $H_2\simeq \Pj^n$ and a hypersurface in $\Pj^n$ of equation:
$$\sum_{i=0}^{n}x_i^{p+1}=0.$$
Furthermore, by Proposition \ref{lefschetz2}, we have:
$$h^{2*}(\Fix G,\Z)=\ell_+^{2*}(X)+\ell_-^{2*+1}(X).$$
If $n$ is even, it implies:
$$(n+1)+n=\ell_+^{2n}(X)+2n.$$
If $n$ is odd by Remark \ref{rr}, it implies:
$$(n+1)+\sum_{k=0}^{n-1}(-1)^k\binom{n+1}{k}(p+1)^{n-k}=\ell_+^{2n}(X)+2n.$$
However by Proposition \ref{vraisans19}:
$$\discr H^{2n}(X/G)=p^{\ell_+^{2n}(X)}.$$
\end{proof}
\begin{rmk}
Note that $b_{2n}(X/G)$ 
can also be computed in the previous example via the generalized Hurwitz formula.
For instance if $n=2$ and $p=3$, $\Fix G$ has two connected components a $\Pj^2$ and a plane curve of degree 4. It follows that:
$$\chi(\Fix G)=3+2-6=-1.$$
Moreover by Remark \ref{rr}:
$$\chi(X)=188.$$
So by Remark \ref{geneHurwitz}, we obtain:
$$\chi(X/G)=62.$$
So:
$$b_4(X/G)=58.$$
\end{rmk}
\begin{ex}
Let $X$ be the hypersurface of dimension $5$ and degree $3$ given by the equation:
$$x_0^2x_1+x_1^2x_2+x_2^2x_3+x_3^2x_4+x_4^2x_5+x_5^2x_6+x_6^2x_0=0.$$
Let $g$ be the automorphism of order $43$ on $X$ induced by the following automorphism on $\mathbb{P}^{6}$:
$$\xymatrix@R0pt{ \Pj^{6}\ar[r]&\Pj^{6}\\
(a_0:a_1:a_2:a_3:a_4:a_5:a_6)\ar@{|->}[r]&(\xi_{43}a_0:\xi_{43}^{41}a_1:\xi_{43}^{4}a_2:\xi_{43}^{35}a_3:\xi_{43}^{16}a_4:\xi_{43}^{11}a_5:\xi_{43}^{21}a_6),
}$$
with $\xi_{43}$ a $43^\mathrm{th}$ root of the unity. We denote $G=\left\langle g\right\rangle$ and $D$ a generator of $H^2(X,\Z)$. Let $\pi:X\rightarrow X/G$ be the quotient map.
Then:
\begin{itemize}
\item[(i)]
$H^{2k}(X/G,\Z)=\Z\pi_*(D^k)$ for all $1\leq k \leq 5$.
\item[(ii)]
$H^3(X/G,\Z)\oplus H^9(X/G,\Z)=(\Z/43\Z)^6$ and $H^5(X/G,\Z)\oplus H^7(X/G,\Z)=(\Z/43\Z)^6$.
\end{itemize}
\end{ex}
\begin{proof}
Indeed, $G$ has only 7 isolated fixed points. Hence by Corollary \ref{maincorcomplete} (i), the coefficients of surjectivity $\alpha_k(X)=0$ for all $k\geq 1$. Moreover the cohomology of even degree is torsion-free. By Remark \ref{rr}, we have $b_5(X)=42$. By Proposition \ref{degecorcomplete}, the spectral sequence of equivariant cohomology with coefficients in $\F$ of $(X,G)$ degenerates at the second page. Hence by 
Corollary \ref{lefschetz2}: $$\ell_-^5(X)=1,\ \text{and}\ \ell_+^5(X)=0.$$ So by Theorem \ref{gene19} (v), we have $b_5(X/G)=0$ (we could also have used the generalized Hurwitz formula).
Finally, as stated in Theorem \ref{main2}, we have that $H^3(X/G,\Z)\oplus H^9(X/G,\Z)=(\Z/43\Z)^{7-\ell_+^2(X)}=(\Z/43\Z)^{6}$ and $H^5(X/G,\Z)\oplus H^7(X/G,\Z)=(\Z/43\Z)^{7-\ell_+^4(X)}=(\Z/43\Z)^{6}$.
\end{proof}


\end{document}